\documentclass[12pt,a4paper]{article}
\usepackage{amsmath}
\usepackage{amssymb}
\usepackage{amsthm}
\usepackage{amsfonts}
\usepackage{latexsym}

\theoremstyle{plain}
\newtheorem{thm}{Theorem}[section]
\newtheorem{cor}{Corollary}[section]
\newtheorem{lem}{Lemma}[section]
\newtheorem{prop}{Proposition}[section]
\theoremstyle{definition}
\newtheorem{defn}{Definition}[section]
\newtheorem{notn}{Notation}[section]

\theoremstyle{remark}

\newtheorem{rem}{Remark}[section]

\title{Local trace formulae for commuting Hamiltonians
 in Toeplitz quantization}
\author{Roberto Paoletti\footnote{\noindent{\bf Address:}
Dipartimento di Matematica e Applicazioni, Universit\`a degli Studi
di Milano Bicocca, Via R. Cozzi 53, 20125 Milano,
Italy; {\bf e-mail}: roberto.paoletti@unimib.it }}
\date{}

\begin{document}

\maketitle
\begin{abstract}
Let $(M,J,\omega)$ be a quantizable compact K\"{a}hler manifold,
with quantizing Hermitian line bundle $(A,h)$, and associated Hardy space
$H(X)$, where $X$ is the unit circle bundle. Given a collection of
$r$ Poisson commuting quantizable Hamiltonian functions $f_j$ on $M$,
there is an induced Abelian unitary action on $H(X)$, generated
by certain Toeplitz operators naturally induced by the $f_j$'s.
As a multi-dimensional analogue of the usual Weyl law and trace formula,
we consider the problem of describing the asymptotic clustering of the 
joint eigenvalues of these Toeplitz operators along a given ray, 
and locally on $M$
the asymptotic concentration of the corresponding joint eigenfunctions.
This problem naturally leads to a \lq directional local trace formula\rq,
involving scaling asymptotics in the neighborhood of certain special loci
in $M$. Under
natural transversality assumption, we obtain asymptotic expansions
related to the local geometry of the Hamiltonian action and flow.
\end{abstract}
 
\section{Introduction}

This paper is concerned with certain asymptotic expansions related 
to the singularities of a distributional trace, which is associated to the joint quantization
of a family of pairwise commuting Hamiltonians in Toeplitz quantization. 
The emphasis will be on the local manifestation of these expansions, where locality is referred to 
the phase space of the the classical system, and to its relation to the underlying symplectic geometry
and dynamics.
Before stating the relevant results, we need to describe at some length
the general picture in which we are working.

\subsection{Quantization and distributional traces}

\subsubsection{Berezin-Toeplitz quantization in the Hardy space scheme}
Let $(M,J,\omega)$ be a $d$-dimensional compact K\"{a}hler manifold,
endowed with the symplectic volume form $\mathrm{d}V_M=:\omega^{\wedge d}/d!$

The symplectic manifold $(M,2\omega)$ may be viewed as a model for a
classical phase space. To any $f\in \mathcal{C}^\infty(M)$, the symplectic
structure $2\,\omega$ associates a
Hamiltonian vector field $\upsilon_f\in \mathfrak{X}(M)$, and the latter generates a Hamiltonian
flow $\phi^M_s:M\rightarrow M$ ($s\in \mathbb{R}$). 

\begin{defn}
 \label{defn:compatible Hamiltonian}
 The Hamiltonian $f\in \mathcal{C}^\infty(M)$ is \textit{compatible}
(with the K\"{a}hler structure $(\omega,J)$ of $M$)
if $\phi^M_\tau$ is holomorphic for every $\tau\in \mathbb{R}$.
\end{defn}

Albeit very special, compatible Hamiltonians are a very important and 
natural object of study; for instance, they are closely
related to holomorphic Lie group actions
on complex projective manifolds. \footnote{Compatible Hamiltonians are also called 
\textit{quantizable} in the literature \cite{cahen-gutt-rawnsley-I}.}

The following definition is standard in geometric quantization:

\begin{defn}
 \label{defn:quantizzabile}
The K\"{a}hler manifold $(M,J,\omega)$ is \textit{quantizable} if 
there exists a positive Hermitian homolomorphic
line bundle $(\mathcal{A},h)$ on $M$,
such that the unique compatible covariant derivative $\nabla$ on $\mathcal{A}$ has curvature $\Theta=-2i\,\omega$.
\end{defn}

It is well-known that the spaces $H^0\left(M,\mathcal{A}^{\otimes k}\right)$
of global holomorphic sections of powers of $\mathcal{A}$,
endowed with their natural Hilbert space structures, play the role of
a \lq quantum counterpart\rq\, of $(M,2\omega)$
at Plack's constant $\hbar=1/k$, $k=1,2,\ldots$.
Hardy space formalism provides a convenient repackaging of this picture, 
in which sections can be viewed as functions, and all values of $\hbar$ can be treated collectively,
as we now recall.

Let $\mathcal{A}^\vee$ be the dual line bundle to $\mathcal{A}$, with the induced Hermitian metric,
and let $X\subseteq A^\vee$ be the unit circle bundle, with projection 
$\pi:X\rightarrow M$ and connection 1-form $\alpha\in \Omega^1(X)$. Then $(X,\alpha)$ is a contact manifold,
with volume form
$\mathrm{d}V_X=:(\alpha/2\pi)\wedge \pi^*(\mathrm{d}V_M)$. We shall denote by $\partial _\theta$ the 
generator of the structure $S^1$-action on $X$.

The tangent bundle of $X$ splits as an invariant direct sum
\begin{equation}
 \label{eqn:somma diretta ortogonale orizzvert}
 TX=\mathcal{V}\oplus \mathcal{H},
\end{equation}
where $\mathcal{V}=\ker (\mathrm{d}\pi)$, the vertical tangent bundle,
is the rank-1 sub-bundle generated by $\partial_\theta$, and
$\mathcal{H}=\ker(\alpha)$ is the horizontal tangent bundle.
The complex structure $J$ naturally lifts to
a complex structure $J_H$ on the horizontal tangent bundle of $X$, and $J_H$ is a 
CR structure on $X$.

For any real-valued $f\in \mathcal{C}^\infty(M)$, there is a natural lift of $\upsilon_f$ to a contact
vector field $\widetilde{\upsilon}_f\in \mathfrak{X}(X)$, given by
\begin{equation}
 \label{eqn:contact lift vector}
\widetilde{\upsilon}_f=:\upsilon_f^\sharp-f\,\partial_\theta,
\end{equation}
where $\upsilon_f^\sharp$ is the $\alpha$-horizontal lift of $\upsilon_f$
\footnote{
Modifying $f$ by an additive constant will leave $\upsilon_f$ unchanged,
but alter
$\widetilde{\upsilon}_f$.}.
Thus
$\widetilde{\upsilon}_f$ generates a contact flow $\phi^X_s:X\rightarrow X$ on $(X,\alpha)$.
This flow preserves the splitting (\ref{eqn:somma diretta ortogonale orizzvert}); in addition, it
preserves the CR structure $J_H$ precisely when $f$ is compatible. 

The following is standard terminology (see \cite{bdm_sj}, \cite{bdm-g}, \cite{zel_tian} and \cite{bsz}):

\begin{defn}
 \label{defn:spazio di Hardy}
 The \textit{Hardy space}
$H(X)\subseteq L^2(X)$ of $X$ consists of the boundary values of
$L^2$-summable holomorphic functions on the unit disc bundle of $A^\vee$. 
The \textit{Szeg\"{o} projector} of $X$ is the $L^2$-orthogonal projector 
$\Pi:L^2(X)\rightarrow H(X)$; its distributional kernel $\Pi\in \mathcal{D}'(X\times X)$
is called the \textit{Szeg\"{o} kernel} of $X$. 
\end{defn}

As is well-known, there is a natural unitary isomorphism
$$
H(X)\cong H(M,\mathcal{A})=:\bigoplus_{\ell\ge 0}H^0\left(M,\mathcal{A}^{\otimes \ell}\right),
$$
where $\bigoplus$ is the Hilbert space direct sum.
The subspace of $H(X)$ corresponding to $H^0\left(M,A^{\otimes \ell}\right)$ is precisely the
$\ell$-th equivariant piece $H(X)_\ell\subseteq H(X)$ for the structure $S^1$-action
\cite{zel_tian}, \cite{bsz}.

By pull-back, the flow $\phi^X_s$ yields
a 1-parameter family of unitary automorphisms $U(s)=U_f(s)=:\left(\phi^X_{-s}\right)^*:L^2(X)\rightarrow L^2(X)$
($s\in \mathbb{R}$).
Furthermore, $f$ is compatible if and only if $\phi^X_s$ preserves the CR structure of $X$,
and this is equivalent to $H(X)$ being $U(s)$-invariant for every $s$. Therefore, a compatible
$f$ induces 
a 1-parameter family of unitary automorphisms 
\begin{equation}
 \label{eqn:1parameter unitary}
 \mathfrak{U}(s)=\mathfrak{U}_f(s):H(X)\rightarrow H(X).
\end{equation}

The family $\mathfrak{U}(s)$ is a quantization of the Hamiltonian flow $\phi^M_s$;
the quantization of the classical Hamiltonian $f$ should be a self-adjoint operator acting on $H(X)$,
and in Berezin-Toeplitz quantization this is given by a Toeplitz operator. In the Hardy space
picture, following \cite{bdm-g}, these are defined as follows.

\begin{defn}
 \label{defn:operatori di toplitz}
A $k$-th order Toeplitz operator on $X$ is a composition 
$T=:\Pi\circ Q\circ \Pi$, where $Q$ is a $k$-th order
pseudo-differential operator; $T$ is viewed as a possibly unbounded 
linear operator on $H(X)$.
\end{defn}

By the theory of \cite{bdm-g}, Toeplitz operators have a well-defined principal symbol.
Let
$$
\Sigma=:\big\{(x,\,r\alpha_x)\,:\,x\in X,\,r>0\big\}\subseteq T^*X\setminus (0).
$$
be 
the closed symplectic cone sprayed by $\alpha$. 

\begin{defn}
 \label{defn:simbolo principale toeplitz}
If $T$ is a $k$-th order Toeplitz operator on $X$,
its principal symbol $\mathfrak{s} _T:\Sigma\rightarrow \mathbb{C}$
is the $k$-th order homogeneous function on $\Sigma$ given by
the restriction of the principal symbol of $Q$ ($\mathfrak{s} _T$ is
independent of the choice of $Q$ in the definition of $T$).
\end{defn}

For instance, given $f\in \mathcal{C}^\infty(M)$ real valued, let $M_f:L^2(X)\rightarrow L^2(X)$
be the self-adjoint operator given by multiplication by $f\circ \pi$. Then
$T_f=:\Pi\circ M_f\circ \Pi$ is an invariant zeroth order Toeplitz operator, viewed as a self-adjoint endomorphism of $H(X)$. 
Its principal symbol is $\mathfrak{s}_{T_f}(x,r\,\alpha_x)=:f\big(\pi(x)\big)$.
Composing $T_f$ with the \lq number operator\rq\, $D=-i\,\partial_\theta$
turns it into a first order 
operator $T'_{f}$, with principal symbol $\mathfrak{s} _{T'_f}(x,r\,\alpha_x)=r\,f\big(\pi(x)\big)$.

When $f$ is compatible, there is another first-order Toeplitz operator associated to it,
that captures more explicitly the associated dynamics. 
Namely, $\widetilde{\upsilon}_f$ is a skew-Hermitian operator on $L^2(X)$ and leaves $H(X)$
invariant; therefore,
the restriction 
\begin{equation}
 \label{eqn:restriction of vf}
 \mathfrak{T}_f=:\left. i\,\widetilde{\upsilon}_f\right|_{H(X)}:H(X)\rightarrow H(X)
\end{equation}
is a first-order (formally) self-adjoint Toeplitz
operator; its principal 
symbol is again $\mathfrak{s}_{\mathfrak{T}_f}(x,r\,\alpha_x)=r\,f\big(\pi(x)\big)$.
Then $\mathfrak{T}_f$ generates $\mathfrak{U}(\cdot)$, i.e. $\mathfrak{U}(s)=e^{is\,\mathfrak{T}_f}$.

\subsubsection{Distributional traces of Toeplitz operators}
\label{sct:distributional trace}

Geometric quantization aims to relate the asymptotic properties of a quantized system to the underlying
classical dynamics and geometry. These properties may be of either global or local nature on $(M,2\omega)$.
For instance, the spectral asymptotics of a Toeplitz operator yield information of a global nature, while the 
asymptotic concentration of its eigenfunctions is a local result. Local properties can be turned into global ones by 
integration.

In particular, suppose that $f>0$. Then $\mathfrak{T}_f$ has eigenvalues on $H(X)$
$$\lambda_1\le \lambda_2\le \ldots,$$ 
repeated according to multiplicity, with $\lambda_j\uparrow +\infty$;
there is a complete orthonormal system $(e_j)$ of $H(X)$ formed by corresponding eigenvectors.

The distributional kernels of $\mathfrak{T}_f$ and 
$\mathfrak{U}(s)$ may be represented in terms of these spectral data\footnote{We shall not distinguish notationally an operator
from its distributional kernel.}: 
\begin{equation}
 \label{eqn:spectral_repr}
 \mathfrak{T}_f(x,y)=\sum_j\lambda_j\,e_j(x)\cdot \overline{e_j(y)},\,\,\,\,\,\,\,\,\,
 \mathfrak{U}(s,x,y)=\sum_je^{is\lambda_j}\,e_j(x)\cdot \overline{e_j(y)},
\end{equation}
where $x,y\in X$ and $s\in \mathbb{R}$.

The distributional trace 
$$\mathrm{tr}(\mathfrak{U})=:\sum_j e^{i\lambda_j s}:\chi=\chi(s)\mapsto \sum_j\widehat{\chi}(-\lambda_j)$$ 
is then a well-defined distribution on the real
line, and its singularities are concentrated on the set of periods of the contact flow $\phi^X$
\cite{bdm-g}. The trace formula
in \textit{loc. cit.} describes the singularity at each period
(for Toeplitz operators related to general symplectic cones). 
In particular, by a Tauberian argument the estimate
of the \lq big\rq\, singularity at the origin yields a Weyl law for the counting function of the $\lambda_j$'s
\footnote{For pseudodifferential operators, corresponding results had appeared in \cite{hor1} and
\cite{dg}}. 

In the present Berezin-Toeplitz context,
local version of these results (i.e., local Weyl laws and local trace formulae)
where obtained in \cite{p_weyl}, \cite{p_ijgmmp}, \cite{pao_JMP}, \cite{pao_equiv_weyl}.

\subsubsection{Commuting Hamiltonians and Abelian contact actions}
\label{sctn:commuting_contact}
We aim to generalize these results to a collection
of Poisson commuting compatible 
Hamiltonians $f_1,\ldots,f_r\in \mathcal{C}^\infty(M)$, 
meaning that $\{f_k,f_l\}=0$ for $k,l=1,\ldots,r$, where $\{\,,\,\}$ is the usual Poisson Lie bracket
of $\mathcal{C}^\infty(M)$. Let us write $\upsilon_k$ for $\upsilon_{f_k}$, and similarly for
$\widetilde{\upsilon}_k$. Then $[\upsilon_k,\upsilon_l]=0$ on $M$.

In addition, under the previous hypothesis, for every $j,\,k=1,\ldots,r$, we have on $X$
$$
\left[\widetilde{\upsilon}_j,\widetilde{\upsilon}_k\right]=\left[\upsilon_j,\upsilon_k\right]^\sharp
-\{f_j,f_k\}\,\partial_\theta=0.
 $$
Therefore, we obtain commuting self-adjoint first order Toeplitz operators $\mathfrak{T}_k$, 
given by the restriction to $H(X)$ of
$i\,\widetilde{\upsilon}_k$, $k=1,\ldots,r$. 

For instance, suppose that an $\ell$-dimensional compact torus $\mathbf{T}$ acts on $M$ in a holomorphic and
Hamiltonian manner, and let $\Phi:M\rightarrow \mathrm{Lie}(\mathbf{T})^\vee$ be the moment map to the
Lie coalgebra of $\mathbf{T}$. If $\mathbf{v}_j\in \mathrm{Lie}(\mathbf{T})$, $j=1,\ldots,r$, then the functions
$f_j=:\langle\Phi,\mathbf{v}_j\rangle$ are compatible and Poisson commute. 

Let $\phi^M_{j,s}:M\rightarrow M$ and $\phi^X_{j,s}:X\rightarrow X$ be the Hamiltonian and contact flows
associated to each $f_j$ ($s\in \mathbb{R}$). Thus 
\begin{equation}
 \label{eqn:commutativity of flows}
\phi^M_{k,s}\circ \phi^M_{l,s'}=\phi^M_{l,s'}\circ \phi^M_{k,s}\,\,\,\,\,\mathrm{and}\,\,\,\,\,
\phi^X_{k,s}\circ \phi^X_{l,s'}=\phi^X_{l,s'}\circ \phi^X_{k,s},
\end{equation}
for all $k,l=1,\ldots,r$ and $s,\,s'\in \mathbb{R}$.

Let us define $\phi^M:\mathbb{R}^r\times M\rightarrow M$ by
$$
\phi^M(\mathbf{s},\cdot)=\phi^M_\mathbf{s}=:\phi^M_{1,s_1}\circ\cdots\phi^M_{r,s_r}:M\rightarrow M\,\,\,\,\,\,\,\,\,
(\mathbf{s}=(s_j)\in \mathbb{R}^r);
$$
in view of (\ref{eqn:commutativity of flows}), this is an holomorphic action. It is furthermore Hamiltonian,
with moment map
\begin{equation}
\label{eqn:moment_map}
 \Phi=(f_1,\ldots,f_r)^{\mathrm{t}}:M\rightarrow \left(\mathbb{R}^r\right)^\vee\cong \mathbb{R}^r,
\end{equation}
where the latter isomorphism is by means of the standard scalar product.

In the same manner, we obtain a contact action $\phi^X:\mathbb{R}^r\times X\rightarrow X$, given by
$$
\phi^X(\mathbf{s},\cdot)=\phi^X_\mathbf{s}=:\phi^X_{1,s_1}\circ\cdots\phi^X_{r,s_r}:X\rightarrow X\,\,\,\,\,\,\,\,\,
(\mathbf{s}=(s_j)\in \mathbb{R}^r),
$$
which lifts $\phi^M$ in a natural manner. Pulling-back, we have the unitary representations
of $\mathbb{R}$
$$
\mathfrak{U}_j(s)=:\left(\phi^X_{j,-s}\right)^*:H(X)\rightarrow H(X),
$$
which may be combined into a unitary representation
$\mathfrak{U}:\mathbb{R}^r\times H(X)\rightarrow H(X)$, given by
\begin{equation}
 \label{eqn:unitary_operator_r1}
 \mathfrak{U}(\mathbf{s})=\mathfrak{U}(\mathbf{s},\cdot)=:\mathfrak{U}_1(s_1)\circ
 \cdots\circ\mathfrak{U}_r(s_r)=\left(\phi^X_{-\mathbf{s}}\right)^*:H(X)\rightarrow H(X).\end{equation}

\subsubsection{The joint spectrum and the associated trace}
\label{sctn:joint spectrum and trace}
                                                                                                    
Each $\mathfrak{T}_k$ is $S^1$-invariant, and therefore preserves the finite-dimensional
$S^1$-equivariant pieces $H(X)_\ell$, $l=0,1,2,\ldots$; the spectrum
of $\mathfrak{T}_k$ is the union over $\ell\in \mathbb{N}$ 
of the finite spectra of its restrictions to the $H(X)_\ell$'s. 

Furthermore, there is a complete orthonormal system $(e_j)$ of $H(X)$ 
composed of joint eigenvectors of the
$\mathfrak{T}_k$'s. That is, for each $j=1,2,\ldots$ and $k=1,\ldots,r$ we have
$$
\mathfrak{T}_k(e_j)=\lambda_{kj}\,e_j,
$$
where $\Lambda_j=:(\lambda_{1j},\ldots,\lambda_{rj})^\mathrm{t}\in \mathbb{R}^r$ 
is a \textit{joint eigenvalue} of the $\mathfrak{T}_k$'s.

We see in particular that for every $j$ we have
\begin{eqnarray}
 \label{eqn:unitary_operator_r}
 \mathfrak{U}(\mathbf{s})(e_j)&=&\mathfrak{U}_1(s_1)\circ
 \cdots\circ\mathfrak{U}_r(s_r)(e_j)\nonumber\\
 &=&e^{i\,(\lambda_{1j}\,s_1+\cdots+\lambda_{rj}\,s_r)}\,e_j=
 e^{i\,\langle\Lambda_j,\mathbf{s}\rangle\,}\,e_j.
\end{eqnarray}

In general, a given joint eigenvalue $\beta\in \mathbb{R}^r$ of the commuting system 
$\mathfrak{T}=:(\mathfrak{T}_k)$ needn't have finite
multiplicity: it may happen that $\beta=\Lambda_j$ for infinitely many $j$'s. 
Nonetheless, as in the case $r=1$, infinite multiplicities do not occur if
$\mathbf{0}\not\in \Phi(M)$ because in this case $\Lambda_j\rightarrow \infty$ 
(Lemma \ref{lem:eigenvalues drift to infinity}).

If $\mathbf{0}\not\in \Phi(M)$, therefore, the $\Lambda_j$'s drift to infinity and (just to fix ideas) 
may be ordered lexicographically in a non-decreasing sequence
$\Lambda_1\le \Lambda_2\le\ldots$, where each joint eigenvalue appears repeated according to its multiplicity.
For each $\mathbf{s}=(s_j)\in \mathbb{R}^r$, we obtain a first order self-adjoint Toeplitz operator of the form 
$$
\langle\mathfrak{T},\mathbf{s}\rangle=:\sum_{k=1}^rs_k\,\mathfrak{T}_k,
$$
with eigenvalues $\langle\Lambda_j,\mathbf{s}\rangle=\sum_{k=1}^r\lambda_{kj}s_k$ relative to the eigenvectors
$e_j$. 
Clearly, $\langle\mathfrak{T},\mathbf{s}\rangle$ is the restriction to $H(X)$ of $i\,\widetilde{\upsilon}_{\Phi^\mathbf{s}}$,
where
\begin{equation}
 \label{eqn:contact vector field s}
\widetilde{\upsilon}_{\Phi^\mathbf{s}}=\upsilon^\sharp_{\Phi^\mathbf{s}}-\Phi^\mathbf{s}\,\partial_\theta,
\end{equation}
and $\Phi^\mathbf{s}=:\langle\Phi,\mathbf{s}\rangle=\sum_{k=1}^r s_k\,f_k$,
and its Schwartz kernel is
$$
\langle\mathfrak{T},\mathbf{s}\rangle(x,y)=\sum_{j=1}^{+\infty}\langle\Lambda_j,\mathbf{s}\rangle\,e_j(x)\cdot
\overline{e_j(y)} \,\,\,\,\,\,\,\,\,\,\,(x,y\in X,\,\,\,\,\mathbf{s}\in \mathbb{R}^r).
$$
Similarly, we see from (\ref{eqn:unitary_operator_r}) that
\begin{eqnarray}
 \label{eqn:unitary_operator_r_spectral}
 \mathfrak{U}(\mathbf{s},x,y)&=&\sum_{j=1}^{+\infty}e^{i\langle\Lambda_j,\mathbf{s}\rangle}\,
 e_j(x)\cdot \overline{e_j(y)}=e^{i\langle\mathfrak{T},\mathbf{s}\rangle}(x,y).
\noindent
\end{eqnarray}

Then the distributional trace 
\begin{equation}
 \label{eqn:distributional trace Abelian}
 \mathrm{tr}(\mathfrak{U})=:\sum_j e^{i\langle \Lambda_j,\cdot\rangle}
\end{equation}
is a well-defined temperate distribution on $\mathbb{R}^r$, whose singularities 
encapsulate asymptotic information on the distribution of the $\Lambda_j$'s.

As in the 1-dimensional case, 
the singular support of $\mathrm{tr}(\mathfrak{U})$ is contained in the set of periods of $\phi^X$,
\begin{equation}
 \label{eqn:periods on X}
 \mathrm{Per}(\phi^X)=:\left\{\mathbf{s}\in \mathbb{R}^r\,:\,\exists\,x\in X\,\,
\mathrm{such\,that}\,\,\phi^X_\mathbf{s}(x)=x\right\};
\end{equation}
however, unlike the case $r=1$, 
$\mathrm{Per}(\phi^X)$ needn't consist of isolated points for $r\ge 2$. 

So let us fix a period $\mathbf{s}_0\in \mathrm{Per}(\phi^X)$ and a covector 
$\beta\in \left(\mathbb{R}^r\right)^\vee$ of unit length. 
As a measure of the singularity of $\mathrm{tr}(\mathfrak{U})$ at $\mathbf{s}_0$ in the direction $\beta$, 
we can consider the asymptotics 
for $\lambda\rightarrow \infty$ of the Fourier transform
\begin{eqnarray}
 \label{eqn:fourier_transform_0}
\mathcal{F}\big(\chi_{\mathbf{s}_0}\cdot \mathrm{tr}(\mathfrak{U})\big)(\lambda\,\beta)=
\left\langle \mathrm{tr}(\mathfrak{U}), \chi_{\mathbf{s}_0}\,e^{-i\lambda\,\langle \beta,\cdot\rangle}\right\rangle,
\end{eqnarray}
where $\chi_{\mathbf{s}_0}$ is a bump function supported in a small neighborhood of
$\mathbf{s}_0$. We shall take $\chi_{\mathbf{s}_0}(\cdot)=:\chi(\cdot - \mathbf{s}_0)$, 
where $\chi\in \mathcal{C}^\infty _0\left(\mathbb{R}^r\right)$ 
is a bump function vanishing for $\|\mathbf{s}\|\ge \epsilon$. 
We then obtain for (\ref{eqn:fourier_transform_0}):
\begin{eqnarray}
 \label{eqn:fourier_transform}
\mathcal{F}\big(\chi_{\mathbf{s}_0}\cdot \mathrm{tr}(\mathfrak{U})\big)(\lambda\,\beta)&=&
\sum_j\left\langle e^{i\langle \Lambda_j,\cdot\rangle}, \chi_{\mathbf{s}_0}
\,e^{-i\lambda\,\langle \beta,\cdot\rangle}\right\rangle\\
&=&e^{-i\,\lambda \langle \beta,\mathbf{s}_0\rangle}\sum_j e^{i\,\langle\Lambda_j,\mathbf{s}_0\rangle}\,\widehat{\chi}(\lambda \,\beta-\Lambda_j).\nonumber
\end{eqnarray}
%

In particular, for $\mathbf{s}_0=\mathbf{0}$ (\ref{eqn:fourier_transform}) reduces to
\begin{eqnarray}
 \label{eqn:fourier_transform_ray}
\mathcal{F}\big(\chi\cdot \mathrm{tr}(\mathfrak{U})\big)(\lambda\,\beta)
&=&\sum_j \widehat{\chi}(\lambda \,\beta-\Lambda_j),
\end{eqnarray}
which, for $\lambda\rightarrow+\infty$, detects the rate at which the 
$\Lambda_j$'s asymptotically accumulate in the neighborhood of the ray
$\mathbb{R}_+\,\beta$.

On the other hand, (\ref{eqn:fourier_transform}) may be expressed as the genuine trace of the smoothing operator
\begin{equation}
 \label{eqn:smoothing_operator}
\mathcal{S}_\chi(\lambda\,\beta,\mathbf{s}_0)=:
\int_{\mathbb{R}^r}\chi_{\mathbf{s}_0}(\mathbf{s})\,e^{-i\lambda\cdot\langle \beta,\mathbf{s}\rangle}\,
\mathfrak{U}(\mathbf{s})\,
\mathrm{d}\mathbf{s}.
\end{equation}
In other words, if $\mathcal{S}_\chi(\lambda\,\beta,\mathbf{s}_0,\cdot,\cdot)\in \mathcal{C}^\infty(X\times X)$ denotes the 
Schwartz kernel of $\mathcal{S}_\chi(\lambda\,\beta,\mathbf{s}_0)$, then
\begin{equation}
 \label{eqn:kernel_smoothing}
 \mathcal{S}_\chi(\lambda\,\beta,\mathbf{s}_0,x,y)=\sum_j 
 e^{-i\,\langle \lambda \beta-\Lambda_j,\mathbf{s}_0\rangle}\,\widehat{\chi}(\lambda \,\beta-\Lambda_j)
 \,e_j(x)\cdot\overline{e_j(y)}.
\end{equation}
and
\begin{equation}
 \label{eqn:integral_kernel_S}
\mathcal{F}\big(\chi_{\mathbf{s}_0}\cdot \mathrm{tr}(\mathfrak{U})\big)(\lambda\,\beta)=
\int_X\,\mathcal{S}_\chi(\lambda\,\beta,\mathbf{s}_0,x,x)\,\mathrm{dV}_X(x).
\end{equation}
It is suggestive to view $ \mathcal{S}_\chi(\lambda\,\beta,\mathbf{0} )$ as a \lq smoothed spectral projector\rq,
corresponding to a cluster of joint eigenvalues traveling to infinity along the ray $\mathbb{R}_+\beta$.
 
Here we shall analyze the local asymptotics of $\mathcal{S}_\chi(\lambda\,\beta,\mathbf{s}_0,\cdot,\cdot)$.
Although our methods apply with minor changes to the general case, 
to simplify the exposition we shall restrict our treatment to the on-diagonal asymptotics
(which is the one relevant to trace applications). 
For instance, for $\mathbf{s}=\mathbf{0}$ we obtain
\begin{equation}
 \label{eqn:kernel_smoothing_0}
 \mathcal{S}_\chi(\lambda\,\beta,\mathbf{0},x,x)=\sum_j 
 \widehat{\chi}(\lambda \,\beta-\Lambda_j)
 \,\big|e_j(x)\big|^2,
\end{equation}
which detects the asymptotic distribution of the \lq probability amplitudes\rq \,of the eigenfunctions 
corresponding to joint eigenvalues asymptotically clustering along the axis $\mathbb{R}_+\,\beta$.

\subsubsection{Moment map directional transversality}
\label{sctn:moment map tansversality}

Before stating our results, we need to introduce some further pieces of notation
and definitions.

\begin{notn}
 \label{notn:R_n_t}
 We shall view 
$\mathbb{R}^r$ as an Abelian Lie group, with Lie algebra $T_\mathbf{0}\mathbb{R}^r\cong\mathbb{R}^r$ itself, 
and coalgebra $\left(\mathbb{R}^r\right)^\vee$, which we shall identify with $\mathbb{R}^r$ by
means of the standard scalar product. Since it will be convenient to distinguish the various
roles of $\mathbb{R}^r$ in our arguments, 
we shall write $\mathfrak{t}=:T_\mathbf{0}\mathbb{R}^r$, and
write the moment map (\ref{eqn:moment_map}) as $\Phi=(f_k):M\rightarrow \mathfrak{t}^\vee$. We shall generally
denote elements of $\mathbb{R}^r$, viewed as group elements, by $\mathbf{s}_0,\,\mathbf{s},\ldots$,  
elements of $\mathfrak{t}$, viewed as tangent vectors at the origin, by
$\xi,\,\eta,\ldots$, and the general element of $\mathfrak{t}^\vee$ as $\beta$.
\end{notn}

\begin{defn}
 \label{defn:induced_vector_fields}
Any $\xi\in \mathfrak{t}$ induces in a standard manner vector fields $$\xi_M\in \mathfrak{X}(M)\,\,
\mathrm{and}\,\,\xi_X\in \mathfrak{X}(X)$$ on $M$ and $X$, respectively. For any $m\in M$ and $x\in X$, we then
have evaluation maps $\mathrm{val}_m:\mathfrak{t}\rightarrow T_mM$ and 
$\mathrm{val}_x:\mathfrak{t}\rightarrow T_xX$, given by 
$$\mathrm{val}_m:\xi\mapsto \xi_M(m)\,\,\mathrm{and}\,\,\mathrm{val}_x:\xi\mapsto\xi_X(x),$$
respectively. 
\end{defn}

\begin{rem}
\label{rem:contact lift vector}
Let $(e_1,\ldots,e_r)$ be the canonical basis of $\mathfrak{t}=\mathbb{R}^r$.
In intrinsic notation, $\Phi=\sum_j f_j\,e_j^*$, where $(e_j^*)$ is the dual basis.
We have, in particular, $e_{jM}=\upsilon_j$, $\Phi^{e_j}=\langle\Phi,e_j\rangle=f_j$,
$e_{jX}=\widetilde{\upsilon}_j$.
More generally, for any $\xi\in \mathfrak{t}$, $\xi_M$ is the Hamiltonian vector field
associated to $\Phi^\xi=:\langle\Phi,\xi\rangle$, and $\xi_X$ its contact lift according to
(\ref{eqn:contact lift vector}):
$$
\xi_X=\xi_M^\sharp-\Phi^\xi\,\partial_\theta.
$$
\end{rem}

\begin{defn}
 \label{defn:inverse_image_ray}
Suppose $\beta\in \mathfrak{t}^\vee$, $\beta\neq 0$. We shall set
$M_\beta=:\Phi^{-1}\big(\mathbb{R}_+\,\beta\big)$ and $X_\beta=:\pi^{-1}(M_\beta)$.
\end{defn}

Our local analysis requires that 
$\Phi:M\rightarrow \left(\mathbb{R}^r\right)^\vee\cong \mathbb{R}^r$ be transverse to $\mathbb{R}_+\,\beta$.
Thus $M_\beta$ is an invariant compact submanifold of $M$
of codimension $r-1$.

\begin{rem}
 When $\phi^M$ descends to an action of the torus $\mathbf{T}^r=\mathbb{R}^r/\mathbb{Z}^r$, $M_\beta$ 
is also connected (\S 2.1 of \cite{pao_IJM}).
\end{rem}

\begin{rem}
 \label{rem:equivalent_transversality}
This transversality assumption is equivalent to the contact action 
$\phi^X:\mathbb{R}^r\times X\rightarrow X$ being locally
free on $X_\beta$.
In turn, this is also equivalent to
the following condition: for any $m\in M_\beta$, 
the restriction of $\mathrm{val}_m$
to $\ker\big(\Phi(m)\big)\subseteq \mathfrak{t}$ is injective 
(\S 2.2 of \cite{pao_IJM}; see \S \ref{sct:transversality}
below).
\end{rem}

\begin{defn}
\label{defn:matrice prodotto scalare indotto}
 Assume that $\Phi:M\rightarrow \mathfrak{t}^\vee$ is transverse to $\mathbb{R}_+\,\beta$, for some
$\beta\in \mathfrak{t}^\vee$ of unit norm. Then, in view of Remark \ref{rem:equivalent_transversality}, 
for any $m\in M_\beta$ the vector subspace $\ker\Phi(m)\subseteq \mathfrak{t}$ inherits
two Euclidean structures 
$$
\langle\cdot,\cdot\rangle _0, \langle\cdot,\cdot\rangle _1: \ker\Phi(m)\times \ker\Phi(m)\rightarrow \mathbb{R},
$$ 
where the former is the restriction of the Euclidean product of $\mathfrak{t}$, and the latter is the pull-back
of the Euclidean product on $T_mM$ under $\mathrm{val}_m$.  Let $\mathcal{K}=(v_l)$ be any \textit{orthonormal} basis of 
 $\ker\Phi(m)$ with respect to $\langle\cdot,\cdot\rangle _0$, and let $D(m)=D(m,\mathcal{K})$ be 
the representative matrix of 
$\langle\cdot,\cdot\rangle _1$ with respect to $\mathcal{K}$, i.e. 
$$
D(m)_{kl}=\langle v_k,v_l\rangle _1=g_m\big(v_{kM}(m),v_{lM}(m)\big),
$$
where $g$ is the Riemannian metric on $M$. 
Then $\det D(m)>0$ is independent of the choice of $\mathcal{K}$, and we can define a $\mathcal{C}^\infty$ function
$\mathcal{D}:M_\beta\rightarrow \mathbb{R}_+$ by setting
$$\mathcal{D}(m)=:\sqrt{\det D(m)}.$$
\end{defn}

\subsubsection{Periods and singularities}

Let us adopt the short-hand $m_\mathbf{s}=:\phi^M_{-\mathbf{s}}(m)$ and
$x_\mathbf{s}=:\phi^X_{-\mathbf{s}}(x)$ ($m\in M,\,x\in X,\,\mathbf{s}\in \mathbb{R}^r$).

\begin{defn}
 \label{defn:notation_fixed_periods}
For any $\mathbf{s}\in \mathbb{R}^r$, let us denote by 
$$M(\mathbf{s})=:\mathrm{Fix}\left(\phi^M_\mathbf{s}\right)=\{m\in M:m=m_\mathbf{s}\}$$
and
$$X(\mathbf{s})=:\mathrm{Fix}\left(\phi^X_\mathbf{s}\right)=\{x\in X:x=x_\mathbf{s}\}$$
the fixed loci of $\phi^M_\mathbf{s}:M\rightarrow M$ and $\phi^X_\mathbf{s}:X\rightarrow X$, respectively.
\end{defn}

\begin{rem}
In general $X(\mathbf{s})$ is the inverse image in $X$
of the union of 
some connected components of $M(\mathbf{s})$ (but perhaps not all of them).
\end{rem}

\begin{defn}
 \label{defn:periods}
 The period sets of $\phi^M$ and $\phi^X$ are, respectively,
 $$
 \mathrm{Per}\left(\phi^M\right)=:
 \left\{\mathbf{s}\in \mathbb{R}^r\,:\,M(\mathbf{s})\neq \emptyset\right\}
 $$
 and
 $$
 \mathrm{Per}\left(\phi^X\right)=:
 \left\{\mathbf{s}\in \mathbb{R}^r\,:\,X(\mathbf{s})\neq \emptyset\right\}.
 $$
If $\mathbf{s}\in \mathrm{Per}\left(\phi^X\right)$, we shall set 
$$
X_\beta(\mathbf{s})=:X_\beta\cap X(\mathbf{s}).
$$
\end{defn}

Clearly, $\mathrm{Per}\left(\phi^X\right)\subseteq \mathrm{Per}\left(\phi^M\right)$, and the 
inclusion is generally strict. We then have (see \S \ref{sctn:trace_wave_front} and 
\S \ref{sctn:functorial} below):

\begin{prop}
 \label{prop:wave_front}
 If $x\in X$, let us set $m_x=:\pi(x)$. 
Then the wave front set of $\mathrm{tr}(\mathfrak{U})\in \mathcal{D}'\left(\mathbb{R}^r\right)$ is
$$
\mathrm{WF}\big(\mathrm{tr}(\mathfrak{U})\big)=
\big\{\big(\mathbf{s},r\,\Phi(m_x)\big)\,:\,\mathbf{s}\in \mathrm{Per}(\phi^X),\,x\in X(\mathbf{s}),\,r>0\big\}.
$$
\end{prop}

\begin{cor}
 \label{cor:sing_supp}
 The singular support of $\mathrm{tr}(\mathfrak{U})$ is
 $$\mathrm{SS}\big(\mathrm{tr}(\mathfrak{U})\big)=
 \mathrm{Per}(\phi^X).
 $$
\end{cor}

\subsubsection{Heisenberg local coordinates}
\label{sctn:heisenberg lc}
Finally, our local scaling asymptotics are expressed in terms of a system $\gamma_x$ of 
Heisenberg local coordinates (HLCS) on
$X$ centered at $x\in X$. We shall refer to \cite{sz} for a precise definition and a complete
discussion of Heisenberg local coordinates (HLC), and simply list some of their salient properties. 

A HLCS centered at $x\in X$ is commonly represented in additive notation,
$$
\gamma_x:(-\pi,\pi)\times B_{2d}(\mathbf{0},\delta)\rightarrow X,\,\,\,\,
(\theta,\mathbf{v})\mapsto x+(\theta,\mathbf{v}),
$$
where $B_{2d}(\mathbf{0},\delta)\subseteq \mathbb{R}^{2d}$ is the open
ball of center the origin and radius $\delta$. We then have:

\begin{enumerate}
 \item the standard $S^1$-action is expressed by a translation in $\theta$;
 \item $\mathbf{v}\in B_{2d}(\mathbf{0},\delta)\mapsto m_x+\mathbf{v}=:\pi \big(x+(\theta,\mathbf{v})\big)$
 ($\theta$ being irrelevant) is a system of local coordinates on $M$ centered at $m_x$;
 \item $\gamma_x$ induces a unitary isomorphism $T_xX\cong \mathbb{R}\oplus \mathbb{R}^{2d}$,
 compatible with the decomposition of $T_xX=V_x\oplus H_x$ as an orthogonal direct sum of the vertical
 and horizontal tangent space.
 \item HLC can be locally and smoothly deformed with the base point $x$: for any $x\in X$, there exist
 an open neighborhood $x\in X'\subseteq X$ and a $\mathcal{C}^\infty$ map 
 $$
 \gamma:X'\times (-\pi,\pi)\times B_{2d}(\mathbf{0},\delta)\rightarrow
 X,
 $$
 such that $\gamma_y(\theta,\mathbf{v})=:\gamma(y,\theta,\mathbf{v})$ is a system of HLC centered at
 $y$, for each $y\in X'$.
\end{enumerate}

 One often writes $x+\mathbf{v}$ for $x+(0,\mathbf{v})$. 

In a HLCS, the universal nature of near-diagonal scaling asymptotics
of certain kernels variously related to the Szeg\"{o} kernel, such as the ones studied in this paper,
is particularly transparent.
In particular, these asymptotics generally involve a universal exponent, given by a quadratic function
on $\mathbb{R}^{2d}\times \mathbb{R}^{2d}$, that we shall now define (following \cite{bsz} and \cite{sz}).

\begin{defn}
 \label{defn:psi2}
Let us define
$\psi_2:\mathbb{R}^{2d}\times \mathbb{R}^{2d}\rightarrow \mathbb{R}$ by setting
$$
\psi_2(\mathbf{v},\mathbf{w})=:-i\,\omega_0(\mathbf{v},\mathbf{w})-\frac{1}{2}\,\|\mathbf{v}-\mathbf{w}\|^2,
$$
where $\omega_0$ is the standard symplectic structure, and $\|\cdot\|$ is the standard Euclidean norm.
\end{defn}

\begin{notn}
\label{notn:representing matrix}
Given $x\in X_\beta(\mathbf{s}_0)$, and a choice of a HLC system centered at $x$,
we shall let $A=A_{m_x}$ be the corresponding 
unitary (i.e., symplectic and orthogonal) matrix representing 
$d_x\phi^X_{-\mathbf{s}_0}:T_{m_x}M\rightarrow T_{m_x}M$. 
Then, since the action is holomorphic and Abelian, $A\,J_{m_x}=J_{m_x}\,A$ and
$A\,\xi_M(m_x)=\xi_M(m_x)$ for every $\xi\in \mathfrak{t}$.
Therefore, we also have
$A\,J_m\big(\xi_M(m_x)\big)=J_m\big(\xi_M(m_x)\big)$.
\end{notn}

\subsection{The statements}

Our main result on the singularities of $\mathrm{tr}(\mathfrak{U})$
can be viewed euphemistically as a \lq directional local trace formula\rq. Before we state it, let us collect here all
of our assumptions:

\bigskip

\noindent
\textbf{General Hypothesis}:
In the previous general setting, let us assume:

\begin{enumerate}
 \item $f_1,\ldots,f_r:M\rightarrow \mathbb{R}$ are $\mathcal{C}^\infty$, Poisson commuting, and compatible
 with the K\"{a}hler structure $(\omega,J)$ (Definition \ref{defn:compatible Hamiltonian});
 \item If $\Phi=:(f_1,\ldots,f_r)^t:M\rightarrow\mathfrak{t}^\vee$, then
$\mathbf{0}\not\in\Phi(M)\subseteq \mathfrak{t}^\vee$;
 \item $\mathbf{s}_0\in \mathrm{Per}\left(\phi^X\right)$;
 \item $\chi:\mathbb{R}^r\rightarrow \mathbb{R}$ is $\mathcal{C}^\infty$ and compactly supported
 in a ball of center the origin and sufficiently small radius $\epsilon>0$ (although it's unnecessary, 
 we may assume that $\chi,\,\widehat{\chi}\ge 0$);
\item $\chi_{\mathbf{s}_0}(\cdot)=:\chi(\cdot -\mathbf{s}_0)$ in (\ref{eqn:smoothing_operator});
 \item $\beta\in \mathfrak{t}^\vee$ has unit norm, and $\beta\in \mathbb{R}_+\cdot\Phi(m_x)$ for some $x\in X(\mathbf{s}_0)$
 (recall that $m_x=\pi(x)$).
 \item $\Phi$ is transverse to $\mathbb{R}_+\cdot \beta$.
\end{enumerate}

\begin{rem}
 \label{rem:convex body}
 As we have remarked, 
Condition 2 ensures that $\Lambda_j\rightarrow\infty$ in $\mathbb{R}^r$, so that every joint eigenvalue has finite multiplicity
(Lemma \ref{lem:eigenvalues drift to infinity}).
\end{rem}

\begin{thm}
 \label{thm:main rapid decrease}
Assume that the General Hypothesis holds, and choose $$D>0,\,\,\,\,\,\,\delta\in \left(0,1/2\right).$$ 
Then, as $\lambda\rightarrow+\infty$, uniformly for
\begin{equation}
 \label{eqn:assumption lower bound distance}
 \mathrm{dist}_X\big(y,X_\beta(\mathbf{s}_0)\big)\ge D\,\lambda^{\delta-1/2},
\end{equation}
(Definition \ref{defn:periods}) we have
$$
\mathcal{S}_\chi(\lambda\,\beta,\mathbf{s}_0,y,y)=O\left(\lambda^{-\infty}\right).
$$
\end{thm}

Notice that the previous condition may be rewritten
$$
\mathrm{dist}_M\Big(m_y,\pi\big(X_\beta(\mathbf{s}_0)\big)\Big)\ge D\,\lambda^{\delta-1/2},
$$
where $m_y=:\pi(y)$, and that $\pi\big(X_\beta(\mathbf{s}_0)\big)$ is a union of connected components
of $M_\beta(\mathbf{s}_0)$.

Theorem \ref{thm:main rapid decrease} shows that the asymptotics
of $\mathcal{S}_\chi(\lambda\,\beta,\mathbf{s}_0,y,y)$ concentrate in a shrinking neighborhood
of $X_\beta(\mathbf{s}_0)$; this leads to considering appropriate scaling asymptotic near
$X_\beta(\mathbf{s}_0)$, as we shall now make precise. 

It will be proved in \S \ref{sctn:transversality_fixed_loci} that
$X_\beta(\mathbf{s}_0)$ is a submanifold of $X$, and its normal space
at any $x\in X_\beta(\mathbf{s}_0)$ splits naturally as an orthogonal direct sum
\begin{eqnarray}\label{eqn:normal bundle Xbeta}
\lefteqn{N_x\big(X_\beta(\mathbf{s}_0)\big)}\\
&\cong&
\ker \left(d_{m_x}\phi^M_{\mathbf{s}_0}-\mathrm{id}_{T_{m_x}M}\right)^\perp\oplus^\perp
\big[J_{m_x}\circ \mathrm{val}_{m_x}\big(\ker \Phi(m_x)\big)\big];\nonumber
\end{eqnarray}
here $T_{m_x}M$
and its subspaces are viewed as vector subspaces of $T_xX$, by identifying 
$T_{m_x}M$ with the horizontal tangent
space $H_x\subset T_xX$.

In addition, as recalled in \S \ref{sctn:heisenberg lc}, in the neighborhood of any $x_0\in X_\beta(\mathbf{s}_0)$
we can find a smoothly varying family of HLCS's. 
Thus, locally near $x_0$, any $z\in X$ within a distance $D\,\lambda^{\delta-1/2}$
from $X_\beta(\mathbf{s}_0)$ can be written 
$z=x+\mathbf{v}$, for unique $x\in X_\beta(\mathbf{s}_0)$ and $\mathbf{v}\in N_x\big(X_\beta(\mathbf{s}_0)\big)$,
with $\|\mathbf{v}\|\le D'\,\lambda^{\delta-1/2}$ for some $D'>0$ (we may take $D'=D+\epsilon'$ for any
$\epsilon'>0$). 
In turn, in view of the previous
direct sum decomposition we can also write
\begin{equation}
\label{eqn:vtauxi}
 \mathbf{v}=\mathbf{w}+\mathbf{n}
\end{equation}
where 
\begin{equation}
 \label{eqn:decomposition normal space m}
\mathbf{w}\in \ker \left(d_{m_x}\phi^M_{\mathbf{s}}-\mathrm{id}_{T_{m_x}M}\right)^\perp\,\,\,\,\mathrm{and}\,\,\,\,
\mathbf{n}=J_{m_x}\big(\xi_M(m_x)\big)
\end{equation}
with $\xi\in \ker\Phi(m)$; here both $\mathbf{w}$ and $\xi$ are also uniquely determined, and both
have norms $O\left(\lambda^{\delta-1/2}\right)$.

In order to obtain the desired scaling asymptotics in a shrinking neighborhood
of $X_\beta(\mathbf{s}_0)$, we shall replace the local
parametrization $y=x+\mathbf{v}$ by its rescaled version 
$$y_\lambda=:x+\mathbf{v}/\sqrt{\lambda},$$
where now $\|\mathbf{v}\|=O\left(\lambda^{\delta}\right)$.

\begin{thm}
\label{thm:main_directional_trace}
Assume that the General Hypothesis holds, and choose arbitrary constants 
$$D>0,\,\,\,\,\,\,\delta\in (0,1/6).$$
Then, uniformly in $x\in X_\beta(\mathbf{s}_0)$ and 
$\mathbf{v}=\mathbf{w}+\mathbf{n}\in N_x\big(X_\beta(\mathbf{s}_0)\big)$ as in
(\ref{eqn:vtauxi}) and 
(\ref{eqn:decomposition normal space m}) 
with $\|\mathbf{v}\|\le D\,\lambda^{\delta}$, the following
asymptotic expansion holds:
\begin{eqnarray}
 \label{eqn:espansione integrated teorema}
\lefteqn{\mathcal{S}_\chi(\lambda\,\beta,\mathbf{s}_0,y_\lambda,y_\lambda)}\nonumber\\
&\sim&
\frac{2^{\frac{r+1}{2}}\,\pi}{\|\Phi(m_x)\|}\cdot\left(\frac{\lambda}{\pi\,\|\Phi(m_x)\|}\right)^{d+\frac{1-r}{2}}\,
\frac{e^{-i\,\lambda\,\langle\beta,\mathbf{s}_0\rangle}}{\mathcal{D}(m)}\,
e^{\left[\psi_2(A\mathbf{w},\mathbf{w})-2\,\|\mathbf{n}\|^2\right]/\|\Phi(m_x)\|}\nonumber\\
&&
\cdot \sum_{\ell\ge 0}\lambda^{-\ell/2}\,\mathcal{R}_\ell (x;\mathbf{n},\mathbf{w}),
\end{eqnarray}
where $\psi_2$ is as in Definition \ref{defn:psi2}, $A$ is defined in Notation \ref{notn:representing matrix} 
in \S \ref{sctn:heisenberg lc},
and
$\mathcal{R}_\ell (x;\cdot,\cdot)$ is a polynomial of degree $\le 3\ell$, and parity $(-1)^\ell$.
We have $\mathcal{R}_0=\chi(\mathbf{0})$.
\end{thm}

Notice that, with $\mathbf{w}$ and $\mathbf{n}$ as in (\ref{eqn:decomposition normal space m}),
we have for some constant $a>0$
\begin{eqnarray*}
 \Re\left(\psi_2(A\mathbf{w},\mathbf{w})-2\,\|\mathbf{n}\|^2\right)
 &=&-\frac{1}{2}\,\|A\mathbf{w}-\mathbf{w}\|^2-2\,\|\mathbf{n}\|^2\\
 &\le& -a\,\left(\|\mathbf{w}\|^2+\|\mathbf{n}\|^2\right);
\end{eqnarray*}
therefore (\ref{eqn:espansione integrated teorema}) describes, in rescaled coordinates, an exponential decrease of
$\mathcal{S}_\chi(\lambda\,\beta,\mathbf{s}_0,y_\lambda,y_\lambda)$ along normal directions to
$X_\beta(\mathbf{s}_0)$.

Inserting the local asymptotics in Theorem \ref{thm:main_directional_trace} in (\ref{eqn:integral_kernel_S})
we obtain a global asymptotic expansion for 
$\mathcal{F}\big(\chi_{\mathbf{s}_0}\cdot \mathrm{tr}(\mathfrak{U})\big)(\lambda\,\beta)$.
Before stating this, we need to introduce a further Poincar\'{e} type invariant. Given $m\in M(\mathbf{s}_0)$,
let $$\kappa_{m,\mathbf{s}_0}:N_m\big(M(\mathbf{s}_0)\big)\rightarrow N_m\big(M(\mathbf{s}_0)\big)$$ be the restriction of 
$\mathrm{id}_{T_mM}-\mathrm{d}_m\phi^M_{-\mathbf{s}_0}$. Then $\kappa_{m,\mathbf{s}_0}$ is an automorphism,
and its determinant 
\begin{equation}
 \label{eqn:determinant fixed locus}
 \mathfrak{k}(m,\mathbf{s}_0)=:\det\big(\kappa_{m,\mathbf{s}_0}\big)
\end{equation}
is locally constant on $M(\mathbf{s}_0)$.

\begin{defn}
 \label{defn:determinant fixed locus component}
 Let $M_\beta(\mathbf{s}_0)_j$, $1\le j\le b$, where $b=b(\beta,\mathbf{s}_0)$, 
denote the connected components of $\pi\big(X_\beta(\mathbf{s}_0)\big)$;
these are some of the connected components of $M_\beta(\mathbf{s}_0)$, but perhaps not all of them.
We shall denote by $ \mathfrak{c}_j(\mathbf{s}_0)\in \mathbb{C}$ the constant value of
$\mathfrak{k}(m,\mathbf{s}_0)$ on $M_\beta(\mathbf{s}_0)_j$. Also, let $c_j$ 
(respectively, $f_j=:d-c_j$)
be the complex codimension (respectively, complex dimension) of
of $M(\mathbf{s}_0)$ in $M$ along $M_\beta(\mathbf{s}_0)_j$, that is, the complex codimension
(respectively, dimension)
in $M$ of the unique connected component of $M(\mathbf{s}_0)$ containing $M_\beta(\mathbf{s}_0)_j$.
\end{defn}

\begin{cor}
 \label{cor:fourier transform}
Under the same assumptions as in Theorem \ref{thm:main_directional_trace}, and with the notation
(\ref{eqn:fourier_transform_0}), we have 
$$
\mathcal{F}\big(\chi_{\mathbf{s}_0}\cdot \mathrm{tr}(\mathfrak{U})\big)(\lambda\,\beta)
=\sum_{j=1}^b\mathcal{F}_j\big(\chi_{\mathbf{s}_0}\cdot \mathrm{tr}(\mathfrak{U})\big)(\lambda\,\beta),
$$
where each summand admits an asymptotic expansion
\begin{eqnarray*}
 \mathcal{F}_j\big(\chi_{\mathbf{s}_0}\cdot \mathrm{tr}(\mathfrak{U})\big)(\lambda\,\beta)\sim
\frac{2\pi}{\mathfrak{c}_j(\mathbf{s}_0)}\,e^{-i\,\lambda\,\langle\beta,\mathbf{s}_0\rangle}
\,\left(\frac{\lambda}{\pi}\right)^{f_j+1-r}\cdot\sum_{k\ge 0}\lambda^{-k}\,\mathcal{U}_{jk}(\mathbf{s}_0,\beta),
\end{eqnarray*}
with the leading order term given by
\begin{equation*}
 \mathcal{U}_{j0}(\mathbf{s}_0,\beta)=:\chi(\mathbf{0})\cdot
 \int_{M_\beta(\mathbf{s}_0)_j}\dfrac{1}{\|\Phi(m)\|^{f_j+2-r}}\,\frac{1}{\mathcal{D}(m)}\,
\mathrm{d}V_{M_\beta(\mathbf{s}_0)_j}(m).
\end{equation*}
\end{cor}

\noindent
Here $\mathrm{d}V_{M_\beta(\mathbf{s}_0)_j}$ is the Riemannian volume density on
$M_\beta(\mathbf{s}_0)_j$.

\begin{rem}
 For the case $r=1$, see \cite{p_ijgmmp}.
\end{rem}

\section{Preliminaries}

\subsection{The moment map}

In the following, let $\phi^M:\mathbb{R}^r\times M\rightarrow M$ be an Hamiltonian and holomorphic action,
with moment map $\Phi=(f_1,\ldots,f_r)^\mathrm{t}:M\rightarrow \mathbb{R}^r$, such that 
$\mathbf{0}\not\in \Phi(M)$. 

\subsubsection{$\Lambda_j\rightarrow\infty$ and $\mathrm{tr}(\mathfrak{U})$ as a temperate distribution}
\label{sctn:trace_wave_front}

Our first remark is that under the given assumption the joint eigenvalues $\Lambda_j$ drift to infinity
in $\mathbb{R}^r$ as $j\rightarrow+\infty$:

\begin{lem}
 \label{lem:eigenvalues drift to infinity}
 Given that $\mathbf{0}\not\in \Phi(M)$, we have $\Lambda_j\rightarrow\infty$ as $j\rightarrow+\infty$.
\end{lem}

\begin{proof}
 [Proof of Lemma \ref{lem:eigenvalues drift to infinity}]
 The self-adjoint first order Toeplitz operator $\mathfrak{T}_k$
 has eigenvalue $\lambda_{kj}\in \mathbb{R}$ on $e_j$. 
 Therefore, the second order Toeplitz operator $\mathfrak{T}_k^2\ge 0$ has eigenvalue $\lambda_{jk}^2$ on
 $e_j$; its principal symbol is
 $$
 \sigma_{\mathfrak{T}_k^2}\big((x,r\alpha_x)\big)=\sigma_{\mathfrak{T}_k}\big((x,r\alpha_x)\big)^2
 =r^2\,f_k(m_x)^2.
 $$
 Let us define
 \begin{equation}
  \label{eqn:toplitz norm operator}
  \|\mathfrak{T}\|=:\left(\sum_{k=1}^r\mathfrak{T}_k^2\right)^{1/2}.
 \end{equation}
 Then $\|\mathfrak{T}\|$ is a first order Toeplitz operator, with eigenvalue $\|\Lambda_j\|$ on $e_j$;
by the theory of \cite{bdm-g} and the corresponding results for pseudodifferential operators \cite{seeley}, \cite{shubin},
its principal symbol is 
 $$
 \sigma_{\|\mathfrak{T}\|}(x,r\alpha_x)=r\,\left(\sum_{k=1}^r f(m_x)^2\right)^{1/2}=r\,\|\Phi(m_x)\|>0.
 $$
 It follows that $\|\Lambda_j\|\rightarrow+\infty$ as $j\rightarrow+\infty$ \cite{bdm-g}.
 
\end{proof}

\begin{cor}
\label{cor:joint eigenvalues finite multipliplity}
Given that $\mathbf{0}\not\in \Phi(M)$, every joint eigenvalue has finite multiplicity. 
\end{cor}


Thus there exists
$j_0$ such that $\Lambda_j\neq \mathbf{0}$ for $j\ge j_0$. We can strengthen the previous
statement as follows:

\begin{lem}
\label{lem:convergent_series}
If $a>0$ is sufficiently large, then
 $$
\sum_{j\ge j_0}\|\Lambda_j\|^{-a}<+\infty.
$$
\end{lem}

\begin{proof}[Proof of Lemma \ref{lem:convergent_series}]
Let $\|\mathfrak{T}\|$ be as in (\ref{eqn:toplitz norm operator}). 
We can assume that $\|\mathfrak{T}\|$ is the restriction to $H(X)$ of a first-order self-adjoint pseudodifferential
operator $Q$, with everywhere positive principal symbol, and commuting with $\Pi$ \cite{bdm-g}. If 
$\eta_1\le \eta_2\le\cdots$ is the sequence of the eigenvalues of $Q$, repeated according to multiplicity, 
we then have $\eta_l>0$ for $l\ge l_1$ for some appropriate $l_1\gg 0$, and
$$
\sum_{l\ge l_1}\eta_l^{-a}<+\infty
$$ 
for every $a\gg 0$ (Theorem 12.2 of \cite{gr_sj}). Since the $\|\Lambda_j\|$'s are the eigenvalues 
of the restriction of $Q$ to $H(X)$, they form a subsequence of the $\eta_l$'s, and the statement follows.

\end{proof}

\begin{cor}
\label{cor:temperate_distribution}
$\sum_j\delta_{-\Lambda_j}$ and 
$$\mathrm{tr}(\mathfrak{U})=\sum_j e^{i\langle\Lambda_j,\cdot\rangle}
=\mathcal{F}\left(\sum_j\delta_{-\Lambda_j}\right)$$
are temperate distribution on $\mathbb{R}^r$.
\end{cor}

\subsubsection{An intrinsic vector field}
\label{sctn:intrinsic vector field}
For every $m\in M$ there is a unique $\Xi(m)\in \mathfrak{t}$ such that
$\Xi(m)\in \ker \Phi(m)^\perp$, $\|\Xi(m)\|=1$ (with respect to the standard Euclidean product), and 
$$\langle\Phi(m),\Xi(m)\rangle=\|\Phi(m)\|.$$ Equivalently, if $\eta(m)\in \mathfrak{t}$ corresponds to
$\Phi(m)\in \mathfrak{t}^\vee$ under the isomorphism $\mathfrak{t}\cong \mathfrak{t}^\vee$ induced by the
standard Euclidean product, then $$\Xi(m)=\eta(m)/\|\eta(m)\|=\eta(m)/\|\Phi(m)\|.$$

We thus obtain a $\mathcal{C}^\infty$ map $m\in M\mapsto \Xi(m)\in \mathfrak{t}$, taking value in the unit sphere,
and a vector field
$V\in \mathfrak{X}(M)$ on $M$, intrinsically associated to $\Phi$, given by
$$
V(m)=:\Xi(m)_M(m)\,\,\,\,\,\,\,\,\,\,\,(m\in M),
$$
in the notation of Definition \ref{defn:induced_vector_fields}.

Furthermore, given $\nu\in \mathfrak{t}$ and $m\in M$ we have a unique orthogonal decomposition

\begin{equation}
 \label{eqn:variable decomposition}
\nu=\nu'(m)+a(\nu,m)\,\Xi(m),
\end{equation}
where $\nu'(m)\in \ker \Phi(m)$, and $a(\nu,m)=\langle \nu,\Xi(m)\rangle_\mathfrak{t}$.

\subsubsection{Transversality }
\label{sct:transversality}
Given that $\Phi(m)\neq \mathbf{0}$ for every $m\in M$, we obtain a $\mathcal{C}^\infty$ map to the unit sphere:
$$
\Phi_\mathrm{u}=:\frac{1}{\|\Phi\|}\,\Phi:M\rightarrow S^{r-1}\subseteq \mathfrak{t}^\vee\cong \mathbb{R}^r.
$$

If $\beta\in S^{r-1}\subseteq \mathfrak{t}^\vee$ (the unit sphere), we have set $M_\beta=:\Phi^{-1}(\mathbb{R}_+\,\beta)$ e $X_\beta=:\pi^{-1}(M_\beta)$. 
Clearly, $M_\beta=\Phi_\mathrm{u}^{-1}(\beta)$.

\begin{lem}
\label{lem:transversality_equivalent}
Consider $\beta\in \mathfrak{t}^\vee$ of unit norm.
Under the previous assumptions, the following conditions are equivalent:
\begin{enumerate}
\item $\Phi$ is transverse to $\mathbb{R}_+\,\beta$;
\item $\phi^X$ is locally free on $X_\beta$, that is, the stabilizer subgroup in $\mathbb{R}^r$ of any $x\in X_\beta$ is discrete;
\item for any $x\in X_\beta$, $\mathrm{val}_x:\mathfrak{t}\rightarrow T_xX$, $\xi\mapsto \xi_X(x)$, is injective;
\item for any $m\in M_\beta$, the restriction of the evaluation, $\mathrm{val}_m:\ker\Phi(m)\rightarrow T_mM$,
is injective;
\item $\beta$ is a regular value of $\Phi_\mathrm{u}$.
\end{enumerate}

\end{lem}

The first four points follow from the discussion in \S 2.2 of \cite{pao_IJM}, and the latter is straightforward.

\begin{cor}
 \label{cor:isolated_periods_transv}
Given that $\Phi$ is transverse to $\mathbb{R}_+\cdot \beta$, there is a $\phi^X$-invariant
tubular neighborhood $X'\subseteq X$ of $X_\beta$ on which $\phi^X$ is locally free. 
\end{cor}

\begin{cor}
 \label{cor:distance_comparison}
Given that $M$ is compact and that $\mathbf{0}\not\in \Phi(M)$,
if $\beta\in \mathfrak{t}^\vee$ has unit norm and $\Phi$ is transverse to $\mathbb{R}_+\cdot\beta$, 
then there exists a constant $C>0$ such that for all $m\in M$ we have
$$
\mathrm{dist}_{\mathfrak{t}^\vee}\big(\mathbb{R}_+\cdot\Phi(m),\beta\big)\ge C\, \mathrm{dist}_M(m,M_\beta).
$$
\end{cor}

\begin{cor}
 \label{cor:distance_comparison_1}
Under the assumptions of Corollary \ref{cor:distance_comparison}, there exists a constant $C>0$ such that for all
$m\in M$ and $t>0$ we have
$$
\big\|t\,\Phi(m)-\beta\big\|\ge C\, \mathrm{dist}_M(m,M_\beta)
$$
where $\|\cdot\|$ is the standard Euclidean norm on $\mathfrak{t}^\vee\cong \mathbb{R}^r$.

\end{cor}

\subsubsection{Transversality and locally isolated periods}

We are interested in the diagonal asymptotics of (\ref{eqn:kernel_smoothing}), and as we
shall see these are non trivial only in the vicinity of 
$X_\beta(\mathbf{s}_0)$.

In general, periods of $\phi^X$ needn't be isolated; nonetheless, under the previous transversality assumptions,
$\mathbf{s}_0$ is indeed an isolated period in a neighborhood
of $X_\beta(\mathbf{s}_0)$. Let us formalize this point by
giving first a definition.

\begin{defn}
 \label{defn:locally isolated periods}
Suppose $\mu:G\times P\rightarrow P$ is a $\mathcal{C}^\infty$ action of a Lie group
$G$ on a manifold $P$, and let $P'\subseteq P$ be a $G$-invariant open subset. For any $g\in G$, let
$P(g)\subseteq P$ be the fixed locus of $\mu_g:P\rightarrow P$. We shall say that
$g_0\in G$ is an isolated period of $\mu$ on $P'$ if  
$P(g_0)\cap P'\neq\emptyset$ and there exists an open neighborhood
$G'\subseteq G$ of $g_0$ such that
$P(g)\cap P'=\emptyset$ for all $g\in G'\setminus\{g_0\}$. 
If $p\in P(g_0)$, we shall say that $g_0$ is a locally isolated period at $p$ if it is an isolated period
on $P'$ for some open $\mu$-invariant neighborhood $P'$ of $p$.
\end{defn}

\begin{prop}
 \label{prop:isolated_periods}
Let $G$ be an Abelian Lie group, $(P,\varphi)$ a Riemannian manifold, and
$\mu:G\times P\rightarrow P$ a $\mathcal{C}^\infty$ action of $G$ on $P$ as
a group of Riemannian isometries. Suppose that $p_0\in P$ and that $\mu$ 
is locally free at $p_0$ (i.e., $\xi_P(p_0)\neq 0\in T_{p_0}P$, for all $\xi\in \mathfrak{g}$
with $\xi\neq \mathbf{0}$,
where $\mathfrak{g}$ is the Lie algebra of $G$). Consider $g_0\in G$ with 
$\mu_{g_0}(p_0)=p_0$; then $g_0$ is a locally isolated period of $\mu$ at $p_0$.
\end{prop}
 
The statement is quite straightforward when the action is proper. In fact, the hypothesis
implies that the stabilizer subgroup $\mathrm{St}(p_0)\subseteq G$ of $p_0$ is discrete,
whence there exists an open neighborhood $G'\subseteq G$ of $g_0$, such that
$G'\cap \mathrm{St}(p_0)=\{g_0\}$. On the other hand, since the action is proper
there exists an invariant open neighborhood $P'\subseteq P$ such that
$\mathrm{St}(p)\subseteq \mathrm{St}(p_0)$ for every $p'\in P'$ (see e.g. Appendix B of \cite{gkk}), 
and this implies the statement. The claim follows, therefore, whenever $\phi^X$ descends
to an action of the compact torus $\mathbb{T}^r=\mathbb{R}^r/\mathbb{Z}^r$. Since 
we do not wish to impose this condition, we give a general proof.

\begin{proof}[Proof of Proposition \ref{prop:isolated_periods}]
Let $d_P$ and $d_G$ denote the dimensions of $P$ and $G$, respectively.
Let us fix some Euclidean scalar product on $\mathfrak{g}$, and
let $B_{\mathfrak{g}}(\mathbf{0},\delta)\subseteq \mathfrak{g}$ 
be the ball centered at the origin and of radius $\delta>0$.
To abridge notation, let us set $g\cdot p=:\mu(g,p)$.

Then for sufficiently small $\delta$
the map $$\gamma:B_{\mathfrak{g}}(\mathbf{0},\delta)\longrightarrow P,\,\,\,\,\,\xi\mapsto e^\xi\cdot p_0$$ 
is a diffeomorphism
onto its image; here $\xi\mapsto e^\xi$ is of course the exponential map on $G$. 

Thus $Q(\delta)=:\gamma\big(B_{\mathfrak{g}}(\mathbf{0},\delta)\big)\subseteq P$ is a smooth $d_G$-dimensional submanifold
passing through $p_0$. Let $N\subseteq \left.TP\right|_{Q}$ be Riemannian normal bundle to $Q$
and let $N(\epsilon)\subseteq N$
be the $\epsilon$-neighborhood of the zero section for some $\epsilon>0$ sufficiently small.
Also, let $N(\delta,\epsilon)$ be the pull-back of $N(\epsilon)$ to $B_{\mathfrak{g}}(\mathbf{0},\delta)$.
Thus, perhaps after passing to smaller values of $\delta$ and $\epsilon$
if necessary, the normal
exponential map provides a smooth map
$$
\widetilde{\gamma}:N(\delta,\epsilon)\longrightarrow P,\,\,\,\,(\xi,\mathbf{n})\mapsto \exp_P\big(e^{\xi}\cdot p_0,\mathbf{n}\big),
$$
where $\exp_P$ is the exponential map of $P$, defined on some open neighborhood of the zero section
in $TP$.

Then, again perhaps after passing to smaller $\epsilon,\,\delta$ if necessary, $\widetilde{\gamma}$ is a diffeomorphism onto its image 
$R(\delta,\epsilon)$,
which is an open tubular neighborhood of $Q(\delta)$.

Now let $R'=R(\delta',\epsilon')$ be similarly constructed, but with suitably smaller $\delta',\,\epsilon'>0$.
Consider $r=\exp_P\big(e^{\xi}\cdot p_0,\mathbf{n}\big)\in R'$, and
$g=g_0\,e^\eta$ with $\eta\in  B_{\mathfrak{g}}(\mathbf{0},\delta')$, and suppose $r\in P(g)$.
Then, because $G$ is Abelian and acts as a group of Riemannian isometries we have
\begin{eqnarray*}
 g\cdot r&=&g_0\,e^\eta\cdot \exp_P\left(e^{\xi}\cdot p_0,\mathbf{n}\right)\\
&=&\exp_P\left(e^{\xi+\eta}\cdot p_0,\mathrm{d}_{e^{\xi}\cdot p_0}\mu_{g_0\,e^\eta}(\mathbf{n})\right)\\
&=&r=\exp_P\left(e^{\xi}\cdot p_0,\mathbf{n}\right).
\end{eqnarray*}
This forces however $e^{\xi+\eta}\cdot p_0=e^\xi\cdot p_0$, whence $\eta=0$ and so $g=g_0$.
Thus $g_0$ is an isolated period of $\mu$ on $R'$.

Now let $R''=:G\cdot R'$ be the $\mu$-saturation of $R'$. Since $G$ is Abelian, if $r''\in R''$ and $r''=g\cdot r'$ for some
$r'\in R'$, then $r'$ and $r''$ have the same stabilizer. Therefore, $R''$ is an invariant open neighborhood
of $p_0$, and $g_0$ is an isolated period of $\mu$ on $R''$. 

\end{proof}

\begin{cor} 
\label{cor:periodi isolated}
Under the assumptions of Corollary
\ref{cor:isolated_periods_transv}, there is a $\phi^X$-invariant neighborhood $X'$ of $X_\beta(\mathbf{s}_0)$
on which $\mathbf{s}_0$ is an isolated period.
\end{cor}

\begin{proof}[Proof of Corollary \ref{cor:periodi isolated}]
By Corollary \ref{cor:isolated_periods_transv}
and Proposition \ref{prop:isolated_periods}, every $x\in  X_\beta(\mathbf{s}_0)$ has an invariant neighborhood
$X'_x$ on which $\mathbf{s}_0$ is an isolated period. By compactness of $X_\beta(\mathbf{s}_0)$, we may find finitely
many such neighborhoods, say $X_1',\ldots,X_k'$, whose union $X'$ contains $X_\beta(\mathbf{s}_0)$,
and such that $\mathbf{s}_0$ is the only period of $\phi^X$ on $X'_j$ contained in $B_{\mathbb{R}^r}(\mathbf{s}_0,\delta_j)$
for some $\delta_j>0$. Then $X'$
is invariant and $\mathbf{s}_0$ is the only period on $X'$ in $B_{\mathbb{R}^r}(\mathbf{s}_0,\delta)$,
where $\delta=\min (\delta_j)$. 
\end{proof}

\subsubsection{Transversality and fixed loci}
\label{sctn:transversality_fixed_loci}

Since $\phi^M_{\mathbf{s}}:M\rightarrow M$ is holomorphic and symplectic, 
$M(\mathbf{s})$ is a (compact) complex submanifold of $M$ (Definition \ref{defn:notation_fixed_periods}), 
and its tangent space at
any $m\in M(\mathbf{s})$ is 
\begin{equation}
 \label{eqn:tangent_space_fixed_locus}
T_mM(\mathbf{s})=\ker \left(d_m\phi^M_{\mathbf{s}}-\mathrm{id}_{T_mM}\right),
\end{equation}
a complex subspace of $T_mM$. 

In particular, since $\mathbb{R}^r$ is an Abelian Lie group we have for any $\xi\in \mathfrak{t}$
and $m\in M(\mathbf{s})$ that
\begin{equation}
\label{eqn:Abelian_inclusion}
 \xi_M(m),\,J_m\big(\xi_M(m)\big)\in \ker \left(d_m\phi^M_{\mathbf{s}}-\mathrm{id}_{T_mM}\right).
\end{equation}

On the other hand, if $\beta\in \mathfrak{t}^\vee$, $\beta\neq \mathbf{0}$
and $\Phi$ is transverse to 
$\mathbb{R}_+\cdot \beta$, then $M_\beta$ is a (real) compact submanifold of $M$, of codimension $r-1$;
for any $m\in M_\beta$, by the discussion in \cite{pao_IJM} the normal bundle $N_m(M_\beta)$ to $M_\beta$ at $m$ is given by
\begin{equation}
  \label{eqn:normal_bundle_ray}
N_m(M_\beta)=J_m\circ \mathrm{val}_m\big(\ker \Phi(m)\big)\subseteq T_mM.
\end{equation}

Since $X_\beta$ is a union of connected components of $\pi^{-1}(M_\beta)$, its normal space
$N_x(X_\beta)$ at any $x\in X_\beta$ is the horizontal lift of the normal space $N_{m_x}(M_\beta)$,
where $m_x=\pi(x)$. In view of Remark \ref{rem:contact lift vector}, this is
\begin{equation}
 \label{eqn:normal bundles MXbeta}
N_x(X_\beta)=N_{m_x}(M_\beta)^\sharp=\mathrm{val}_x\big(\ker\Phi(m_x)\big).
\end{equation}

\begin{lem}
\label{lem:transversality}
 Suppose as above that $\Phi:M\rightarrow \mathfrak{t}$ is transverse to $\mathbb{R}_+\cdot \beta$.
Then for any $\mathbf{s}\in \mathbb{R}^r$ the following holds:
\begin{enumerate}
 \item $M(\mathbf{s})$ and $M_\beta$ are transverse submanifolds of
$M$;
\item for any $m\in M_\beta(\mathbf{s})=:M(\mathbf{s})\cap M_\beta$,
the normal bundle to $M_\beta(\mathbf{s})$ at $m$ is the \textit{orthogonal} direct sum 
$$
N_m\big(M_\beta(\mathbf{s})\big)=\ker \left(d_m\phi^M_{\mathbf{s}}-\mathrm{id}_{T_mM}\right)^\perp\oplus^\perp 
\big[J_m\circ \mathrm{val}_m\big(\ker \Phi(m)\big)\big].
$$
\end{enumerate}

\end{lem}

\begin{proof}[Proof of Lemma \ref{lem:transversality}]
We need to show that $T_mM=T_mM(\mathbf{s})+ T_mM_\beta$ for any $m\in M(\mathbf{s})\cap M_\beta$, and this is equivalent
to $N_mM(\mathbf{s})\cap N_mM_\beta=(0)$. In view of (\ref{eqn:tangent_space_fixed_locus}), (\ref{eqn:Abelian_inclusion})
and (\ref{eqn:normal_bundle_ray}), this is
\begin{eqnarray*}
N_mM(\mathbf{s})\cap N_mM_\beta&=&\ker \left(d_m\phi^M_{\mathbf{s}}-\mathrm{id}_{T_mM}\right)^\perp\cap J_m\circ \mathrm{val}_m\big(\ker\Phi(m)\big) \\
&\subseteq& \ker \left(d_m\phi^M_{\mathbf{s}}-\mathrm{id}_{T_mM}\right)^\perp\cap \ker \left(d_m\phi^M_{\mathbf{s}}-\mathrm{id}_{T_mM}\right)
=(0).
\end{eqnarray*}

Therefore, $M(\mathbf{s},\beta)$ is a submanifold of $M$, and at any $m\in M(\mathbf{s},\beta)$ we have
$T_mM(\mathbf{s},\beta)=T_mM(\mathbf{s})\cap T_mM(\beta)$. Thus the normal bundle is
$N_mM(\mathbf{s},\beta)=N_mM(\mathbf{s})+ N_mM(\beta)$, and the inclusion above also 
shows that this is an orthogonal direct sum.

\end{proof}

\subsection{$\mathcal{U}$ and the singularities of its trace}

\subsubsection{$\mathfrak{U}$ as a complex FIO}

In the present compatible setting, 
the operator $\mathfrak{U}(\mathbf{s})$ in (\ref{eqn:unitary_operator_r}) 
has a simple expression in terms of $\phi^X_\mathbf{s}$ and
the Szeg\"{o} projector $\Pi$. Namely, let $(e_j)$ be a complete orthonormal system of $H(X)$, so that
$$
\Pi(x,y)=\sum_j e_j(x)\cdot \overline{e_j(y)}.
$$
Then the distributional kernel (\ref{eqn:unitary_operator_r_spectral})
of $\mathfrak{U}(\mathbf{s})=\left(\phi^X_{-\mathbf{s}}\right)^*\circ \Pi$ may also be expressed as
\begin{equation}
 \label{eqn:explicit_szego_U}
\mathfrak{U}(\mathbf{s},x,y)=\sum_je_j\left(\phi^X_{-\mathbf{s}}(x)\right)\cdot \overline{e_j(y)}
=\Pi\left(\phi^X_{-\mathbf{s}}(x),y\right).
\end{equation}
Therefore,
since the singular support of $\Pi$ is the diagonal \cite{f}, the singular support of $\mathfrak{U}(\mathbf{s})$
is the graph of $\phi^X_{-\mathbf{s}}$. 

In addition, by \cite{bdm_sj} near the diagonal we have a microlocal description of $\Pi$ as an
FIO with complex phase, of the form
\begin{equation}
 \label{eqn:szego_microlocal}
\Pi(x,y)\sim \int_0^{+\infty}e^{it\psi(x,y)}\,s(t,x,y)\,\mathrm{d}t,
\end{equation}
where $\Im \psi\ge0$, and $s(t,x,y)\sim\sum_{j\ge 0}t^{d-j}\,s_j(x,y)$ (see also the discussions in \cite{zel_tian}, \cite{sz}). 
Thus, near the graph of $\phi^X_{-\mathbf{s}}$
we have with $x_\mathbf{s}=:\phi^X_{-\mathbf{s}}(x)$
\begin{equation}
 \label{eqn:U_microlocal}
\mathfrak{U}(\mathbf{s},x,y)\sim \int_0^{+\infty}e^{it\psi(x_\mathbf{s},y)}\,s(t,x_\mathbf{s},y)\,\mathrm{d}t.
\end{equation}

\begin{rem}
 \label{rem:diff_psi} With $\psi$ as in (\ref{eqn:szego_microlocal}),
one has
$
\mathrm{d}_{(x,x)}\psi=(\alpha_x,-\alpha_x)
$ for any $x\in X$, and more generally for any $\vartheta\in \mathbb{R}$
$$
\mathrm{d}_{(e^{i\vartheta}\,x,x)}\psi=\left(e^{i\vartheta}\,\alpha_{e^{i\vartheta}\,x},\,-e^{i\vartheta}\,\alpha_{x}\right).
$$
\end{rem}

\begin{rem}
\label{rem:espansione per psi}
As shown in \S 3 of \cite{sz},
in a system of HLC centered at $x\in X$, the phase $t\,\psi$ satisfies the following expansion:
\begin{eqnarray*}
 \lefteqn{t\,\psi\big(x+(\theta,\mathbf{v}),x+\mathbf{w}\big)}\\
 &=&it\,\left[1-e^{i\theta}\right]-it\,\psi_2(\mathbf{v},\mathbf{w})\,e^{i\theta}
 +t\,R_3(\mathbf{v},\mathbf{w})\,e^{i\theta},
\end{eqnarray*}
where $\psi_2$ is as in Definition \ref{defn:psi2},
while $R_3:\mathbb{R}^{2d}\times \mathbb{R}^{2d}\rightarrow \mathbb{C}$ is $\mathcal{C}^\infty$ and vanishes to
third order at the origin.
\end{rem}

The description of $\Pi$ as an FIO in (\ref{eqn:szego_microlocal}), 
in view of Corollary 1.3 of \cite{bdm_sj} and Remark \ref{rem:diff_psi} above, 
implies, as is well-known, that the wave front of $\Pi$ is
$$
\mathrm{WF}(\Pi)=\big\{\big((x,x),\,r\,(\alpha_x,-\alpha_x)\big)\,:\,r>0,\,x\in X\big\}
\subseteq T^*(X\times X).
$$
It follows that the wave front of $\mathfrak{U}(\mathbf{s})$ is
\begin{equation}
 \label{eqn:U_wave_front}
 \mathrm{WF}\big(\mathfrak{U}(\mathbf{s})\big)
 =\big\{\big((x_\mathbf{s},x),\,r\,(\alpha_{x_\mathbf{s}},-\alpha_x)\big)\,:\,r>0,\,x\in X\big\}
\subseteq T^*(X\times X).
\end{equation}

We can view $\mathfrak{U}$ as an operator 
$\mathcal{C}^\infty\left(\mathbb{R}^r\times X\right)\rightarrow\mathcal{C}^\infty(X)$,
with distributional kernel 
$\mathfrak{U}\in \mathcal{D}(\mathbb{R}^r\times X\times X\times X)$; given (\ref{eqn:U_microlocal}) and Remark
\ref{rem:diff_psi}, 
its wave front is
\begin{eqnarray}
 \label{eqn:wave_front_U}
 \mathrm{WF}(\mathfrak{U})&=&\Big\{\Big((\mathbf{s},x,x_\mathbf{s}),\,r\,\big(\Phi(m_x),\,
 \alpha_x,\,-\alpha_{x_\mathbf{s}},)\Big)\,:\nonumber\\
&&\mathbf{s}\in \mathbb{R}^r,\, x\in X,\,r>0\Big\},
\end{eqnarray}
where $m_x=:\pi(x)$; we have used that $\alpha$ is $\phi^X$-invariant.

\subsubsection{Functorial description of $\mathrm{tr}(\mathfrak{U})$}
\label{sctn:functorial}

Since we have chosen a volume form on $X$, there are naturally
induced volume forms (whence densities and half-densities) on $\mathbb{R}^r\times X$ and
$\mathbb{R}^r\times X\times X$; in terms of the latter, we may extend the pull-back operation
of functions under $\mathcal{C}^\infty$ maps involving these manifolds to $\mathcal{C}^\infty$ 
densities. 
Similarly, the push-forward operation, which by duality is naturally defined on
densities under proper $\mathcal{C}^\infty$ maps, extends with the given choices to $\mathcal{C}^\infty$
functions. In addition, these functorial operations may be extended continuously 
to generalized densities, as far as the appropriate conditions involving wave fronts are met
(\cite{hor}, \cite{d}).
The identification between functions, densities and half-densities will be
left implicit in the following.

Let us then consider the diagonal map $\Delta:\mathbb{R}\times X\rightarrow \mathbb{R}\times X\times X$, 
$(\mathbf{s},x)\mapsto (\mathbf{s},x,x)$. In view of (\ref{eqn:wave_front_U})
and the condition $\Phi(m)\neq \mathbf{0}$ $\forall\,m\in M$, the 
pull-back 
$$\Delta^*(\mathfrak{U})=\sum_{j=1}^{+\infty}e^{i\langle\Lambda_j,\mathbf{s}\rangle}\,
 e_j(x)\cdot \overline{e_j(x)}\in \mathcal{D}'\left(\mathbb{R}^r\times X\right)$$
is well-defined; by (\ref{eqn:wave_front_U}) and
the functorial properties of wave fronts (see \cite{hor} and \cite{d}), it has wave front
\begin{eqnarray}
 \label{eqn:wave_front_U_Delta}
 \mathrm{WF}\left(\Delta^*(\mathfrak{U})\right)&=&\Big\{\Big((\mathbf{s},x),\,r\,\big(\Phi(m_x),0)\Big)\,:\nonumber\\
&&\mathbf{s}\in \mathbb{R}^r,\, x\in \mathrm{Fix}\left(\phi^X_\mathbf{s}\right),\,r>0\Big\}.
\end{eqnarray}
where $\mathrm{Fix}\left(\phi^X_\mathbf{s}\right)=\{x\in X:x=x_\mathbf{s}\}$.

Moreover, since the projection $p:\mathbb{R}^r\times X\rightarrow X$ is proper, the push-forward 
$p_*\left(\Delta^*(\mathfrak{U})\right)\in \mathcal{D}'\left(\mathbb{R}^r\right)$ is also
well-defined, and by orthonormality of the $e_j$'s we have
$$
p_*\left(\Delta^*(\mathfrak{U})\right)=\sum_{j=1}^{+\infty}e^{i\langle\Lambda_j,\mathbf{s}\rangle}
=\mathrm{tr}(\mathfrak{U}).
$$
In addition, given (\ref{eqn:wave_front_U_Delta}) its wave front is
\begin{eqnarray}
 \label{eqn:wave_front_U_Delta_pi}
 \mathrm{WF}\big(\mathrm{tr}(\mathfrak{U})\big)&=&\Big\{\Big(\mathbf{s},\,r\,\Phi(m_x)\Big)\,:\,
\mathbf{s}\in \mathbb{R}^r,x\in \mathrm{Fix}\left(\phi^X_\mathbf{s}\right),\,r>0\Big\}\nonumber\\
&=&\bigcup_{\mathbf{s}\in \mathbb{R}^r}\{\mathbf{s}\}\times \mathrm{WF}\big(\mathrm{tr}(\mathfrak{U})\big)_\mathbf{s},
\end{eqnarray}
where for each $\mathbf{s}\in \mathbb{R}^r$ we have set
$$
\mathrm{WF}\big(\mathrm{tr}(\mathfrak{U})\big)_\mathbf{s}=:
\bigcup_{x\in \mathrm{Fix}\left(\phi^X_\mathbf{s}\right)}\mathbb{R}_+\cdot \Phi(m_x)
=\bigcup_{x\in \mathrm{Fix}\left(\phi^X_\mathbf{s}\right)}\mathbb{R}_+\cdot \Phi_\mathrm{u}(m_x).
$$
This proves Proposition \ref{prop:wave_front} and Corollary \ref{cor:sing_supp}.

We are interested in estimating the asymptotics of (\ref{eqn:fourier_transform}). In view of the above we have:

\begin{cor}
Under the previous assumptions, suppose $$(\mathbf{s}_0,\beta_0)\in T^*\left(\mathbb{R}^r\right)\setminus 
\mathrm{WF}\big(\mathrm{tr}(\mathfrak{U})\big),\,\,\,\,\,\,\,\,\|\beta_0\|=1.$$ 
Then there exists $\epsilon>0$ such that for every 
$\chi\in \mathcal{C}^\infty_0\big(B_r(\mathbf{0},\epsilon)\big)$ we have
$$
\left\langle \mathrm{tr}(\mathfrak{U}), \chi_{\mathbf{s}_0}\,e^{-i\lambda\,\langle \beta,\cdot\rangle}\right\rangle
=O\left(\lambda^{-\infty}\right)
$$
uniformly in $\beta\in \mathbb{R}^r$ with $\|\beta\|=1$, $\|\beta-\beta_0\|<\epsilon$.

\end{cor}

Here $B_r(\mathbf{0},\epsilon)\subseteq \mathbb{R}^r$ is the open ball of center the origin and radius $\epsilon$,
while $\chi_{\mathbf{s}_0}(\mathbf{s})=\chi(\mathbf{s}-\mathbf{s}_0)$.

\subsubsection{The smoothing operator}

As in the case $r=1$, the operators $\mathfrak{U}(\mathbf{s})$ may be averaged with a
weight of rapid decrease to obtain
a smoothing operator. 

\begin{lem}
\label{lem:smoothing_operator_kernel}
For any $\chi\in \mathcal{S}\left(\mathbb{R}^r\right)$, the operator
$$
S_\chi=:\int_{\mathbb{R}^r}\chi(\mathbf{s})\,\mathfrak{U}(\mathbf{s})\,\mathrm{d}\mathbf{s}
$$
is smoothing, and its kernel $S_\chi(\cdot,\cdot)\in \mathcal{C}^\infty(X\times X)$ is given by
\begin{equation}
 \label{eqn:explicit_kernel}
 S_\chi(x,y)=\sum_j\widehat{\chi}(-\Lambda_j)\,e_j(x)\cdot \overline{e_j(y)}.
\end{equation}

\end{lem}

The following is an adaptation of an argument in \S 12 of \cite{gr_sj}.

\begin{proof}[Proof of Lemma \ref{lem:smoothing_operator_kernel}]
 Let $Q$ be as in the proof of Lemma \ref{lem:convergent_series}.
 Since the $e_j$'s are orthonormal eigenfunctions of $Q$, with eigenvalues
 $\|\Lambda_j\|$, a standard argument based on the Sobolev inequalities
shows that for some fixed $j_0$ and all $j\ge j_0$ we have
$$
\|e_j\|_{\mathcal{C}^k}\le 
C_k\,\|\Lambda_j\|^{k+2d+1}.
$$
Since $\widehat{\chi}\in \mathcal{S}\left(\mathbb{R}^r\right)$, for any $N>0$ there
exists $C_N>0$ such that for all $j\ge j_0$ we have
$$\left|\widehat{\chi}(-\Lambda_j)\right|\le C_N'\,\|\Lambda_j\|^{-N}.$$
Thus Lemma \ref{lem:convergent_series} implies that (\ref{eqn:explicit_kernel}) converges 
in $\mathcal{C}^\infty(X\times X)$. Given this, that  (\ref{eqn:explicit_kernel}) is indeed the distributional kernel
of $S_\chi$ follows by first applying it to finite linear combinations of the $e_j$'s, and then
using a density argument.

\end{proof}

\section{Proof of Theorem \ref{thm:main rapid decrease}}

\subsection{Concentration near $X_\beta(\mathbf{s}_0)$}

\subsubsection{Concentration near $X(\mathbf{s}_0)$}
We shall first prove that $\mathcal{S}_\chi(\lambda\,\beta,\mathbf{s}_0,y,y)=O\left(\lambda^{-\infty}\right)$, unless
$y$ belongs to a small tubular neighborhood of $X(\mathbf{s}_0)$; 
this will allow us to represent $\Pi$ as an FIO with complex phase, 
without changing the asymptotics.

We have by (\ref{eqn:smoothing_operator}) and (\ref{eqn:explicit_szego_U}):
\begin{eqnarray}
 \label{eqn:fourier_transform_1}
\mathcal{S}_\chi(\lambda\,\beta,\mathbf{s}_0,y,y)
&=&\int_{\mathbb{R}^r}\,
\chi_{\mathbf{s}_0}(\mathbf{s})\,e^{-i\lambda\,\langle\beta,\mathbf{s}\rangle}\,
\Pi\left(y_\mathbf{s},y\right)\,\mathrm{d}\mathbf{s}.
\end{eqnarray}
where $y_\mathbf{s}=\phi^X_{-\mathbf{s}}(y)$.

On the support of $\chi_{\mathbf{s}_0}$, $\|\mathbf{s}-\mathbf{s}_0\|<\epsilon$.
Hence for some $C_1>0$ we have uniformly in $y\in X$:
\begin{equation}
 \label{eqn:distance fixed point locus}
 \mathrm{dist}\left(y_{\mathbf{s}_0},y_{\mathbf{s}}\right)\le
C_1\,\epsilon .
\end{equation}

On the other hand, there exist constants $C_3\ge C_2>0$ such that for every $\epsilon>0$
one has
\begin{equation}
 \label{eqn:distance fixed point image}
 C_3\,\mathrm{dist}_X\big(y,X(\mathbf{s}_0)\big)\ge 
\mathrm{dist}\left(y_{\mathbf{s}_0},y\right)\ge C_2\,\mathrm{dist}_X\big(y,X(\mathbf{s}_0)\big).
\end{equation}

Indeed, since $\phi^X$ is an action by isometries, at any $x\in X(\mathbf{s}_0)$ we have
$$
T_x\big(X(\mathbf{s}_0)\big)=\ker\left(\mathrm{d}_x\phi^X_{\mathbf{s}_0}-id_{T_xX}\right);
$$
hence, there exist constants $C_3'\ge C_2'>0$ such that
$$
C_3'\,\|\mathbf{n}\|\ge
\left\|\mathrm{d}_x\phi^X_{\mathbf{s}_0}(\mathbf{n})-\mathbf{n}\right\|\ge
C_2'\,\|\mathbf{n}\|,
$$
whenever $\mathbf{n}\in T_x\big(X(\mathbf{s}_0)\big)^\perp\subseteq T_xX$.
Then (\ref{eqn:distance fixed point image}) follows by writing, in a tubular neighborhood of
$X(\mathbf{s}_0)$, $y=x+\mathbf{n}$ in a smoothly varying HLC system
centered at $x$, and letting, say, $C_3=2\,C_2'$, $C_2=C_2'/2$.

Let then be $Z_1\subseteq X$ be the locus where 
$\mathrm{dist}_X\big(y,X(\mathbf{s}_0)\big)\ge 2\,(C_1/C_2)\,\epsilon$. If $y\in Z_1$, then
\begin{eqnarray}
 \label{eqn:distance_comparison_0}
\mathrm{dist}\left(y_{\mathbf{s}},y\right)&\ge& \mathrm{dist}\left(y_{\mathbf{s}_0},y\right)-
\mathrm{dist}\left(y_{\mathbf{s}_0},y_{\mathbf{s}}\right)\nonumber\\
&\ge&2\,C_1\,\epsilon-C_1\epsilon\ge C_1\,\epsilon.
\end{eqnarray}

As the singular support of $\Pi(\cdot,\cdot)\in \mathcal{D}'(X\times X)$ is the diagonal
in $X\times X$, it follows from (\ref{eqn:distance_comparison_0})
that the function 
$$
\gamma:(\mathbf{s},y)\in \mathbb{R}^r\times Z_1\mapsto 
\chi_{\mathbf{s}_0}(\mathbf{s})\,
\Pi\left(y_{\mathbf{s}},y\right)
$$
is well-defined and $\mathcal{C}^\infty$, and therefore its Fourier transform in $\mathbf{s}$,
$$
\widehat{\gamma}_y(\beta')=:
\int_{\mathbb{R}^r}\,
\chi_{\mathbf{s}_0}(\mathbf{s})\,e^{-i\langle\beta',\mathbf{s}\rangle}\,
\Pi\left(\phi^X_{-\mathbf{s}}(y),y\right)\,\mathrm{d}\mathbf{s}
$$
decreases rapidly for $\beta'\rightarrow\infty$, uniformly in $y\in Z_1$.

Setting $\beta'=\lambda\,\beta$, with $\beta$ of unit norm, we have proved:

\begin{lem}
 \label{lem:rapid_decrease}
 $\mathcal{S}_\chi(\lambda\,\beta,\mathbf{s}_0,y,y)=O\left(\lambda^{-\infty}\right)$ uniformly in
 $y\in Z_1$.
\end{lem}

\bigskip

\subsubsection{$S_\chi$ as an oscillatory integral}
By virtue of Lemma \ref{lem:rapid_decrease}, in the following we can assume 
$$\mathrm{dist}_X\big(y,X(\mathbf{s}_0)\big)\le 2\,(C_1/C_2)\,\epsilon,$$
whence for $\chi_{\mathbf{s}_0}(\mathbf{s})\neq 0$ we have
\begin{eqnarray*}
\mathrm{dist}_X\left(y_{\mathbf{s}},y\right)&\le&
\mathrm{dist}_X\left(y_{\mathbf{s}},y_{\mathbf{s}_0}\right)+
\mathrm{dist}_X\left(y_{\mathbf{s}_0},y\right)\\
&\le&C_1\,\epsilon+2\,(C_1\,C_3/C_2)\,\epsilon=D_0\,\epsilon,
\end{eqnarray*}
for some constant $D_0>0$. In this range, as in (\ref{eqn:szego_microlocal}) and
(\ref{eqn:U_microlocal}) we can represent $\Pi$ as an FIO 
(any smoothing remainder term contributing negligibly to the asymptotics, as above).

Thus we can rewrite (\ref{eqn:fourier_transform_1}) as
\begin{eqnarray}
 \label{eqn:fourier_transform_2}
\mathcal{S}_\chi(\lambda\,\beta,\mathbf{s}_0,y,y)
&\sim&\int_0^{+\infty}\int_{\mathbb{R}^r}\,
\chi_{\mathbf{s}_0}(\mathbf{s})\,e^{i[t\,\psi (y_\mathbf{s},y)-\lambda\,\langle\beta,\mathbf{s}\rangle]}\,
s(t,y_\mathbf{s},y)\,\mathrm{d}\mathbf{s}\,\mathrm{d}t
\nonumber\\
&= &\lambda\,\int_0^{+\infty}\int_{\mathbb{R}^r}\,
\chi_{\mathbf{s}_0}(\mathbf{s})\,e^{i\lambda \,\Psi_\beta(y,t,\mathbf{s})}\,
s(\lambda\,t,y_\mathbf{s},y)\,\mathrm{d}\mathbf{s}\,\mathrm{d}t,
\end{eqnarray}
where we have performed the change of variables $t\mapsto \lambda\,t$, and set 
\begin{equation}
 \label{eqn:phase_Psi_beta}
 \Psi_\beta(y,t,\mathbf{s})=:t\,\psi (y_\mathbf{s},x)-\langle\beta,\mathbf{s}\rangle.
\end{equation}

For $D\gg 0$, let $\varrho=\varrho_D:\mathbb{R}\rightarrow \mathbb{R}_{\ge 0}$ be $\mathcal{C}^\infty$, identically
equal to $1$ on $[1/D,D]$, and supported in $[1/(2D),2D]$.

\begin{lem}
 \label{lem:compact_integration_t}
If $1\gg \epsilon>0$ and $D\gg 0$,
only a rapidly decreasing contribution to the asymptotics of (\ref{eqn:fourier_transform_2}) is lost, if the integrand
is multiplied by $\varrho(t)$.
\end{lem}

In particular, as far as the asymptotics are concerned,
we may assume without loss that integration in $t$ is compactly supported in $[1/(2D),2D]$.

\begin{proof}[Proof of Lemma \ref{lem:compact_integration_t}]
 Let $A=:\max \{\|\Phi(m)\|:m\in M\}$, $a=:\min \{\|\Phi(m)\|:m\in M\}$.
Then $A\ge a>0$. 

Suppose first that $y\in X(\mathbf{s}_0)$.
Since $d_{(y,y)}\psi=(\alpha_y,-\alpha_y)$ \cite{bdm_sj}, in view of (\ref{eqn:contact vector field s}) we have with $m=m_y$
$$\left.\partial_\mathbf{s}\psi (y_\mathbf{s},y)\right|_{\mathbf{s}_0}=\Phi(m_y),$$
whence $A\ge \left\|\left.\partial_\mathbf{s}\psi (y_\mathbf{s},y)\right|_{\mathbf{s}_0}\right\|\ge a$.

Therefore, by continuity if $\epsilon>0$ is sufficiently small and $\|\mathbf{s}'-\mathbf{s}_0\|<\epsilon$,
$\mathrm{dist}_X\big(y,X(\mathbf{s}_0)\big)\le 2\,(C_1/C_2)\,\epsilon$ as we are assuming then
$$2A\ge \left\|\left.\partial_\mathbf{s}\psi (y_\mathbf{s},y)\right|_{\mathbf{s}'}\right\|\ge a/2.$$

Consequently, in the same range we have
$$
\left\|\left.\partial_\mathbf{s}\Psi_\beta\right|_{\mathbf{s}'}\right\|=
\|t\,\Phi(m_y)-\beta\|\ge \min\left\{\frac{1}{2}\, ta-1,1-2\,tA\right\};
$$ 
Thus, if say $t\ge 6/a$ then
$$
\left\|\left.\partial_\mathbf{s}\Psi_\beta\right|_{\mathbf{s}'}\right\|
\ge \frac{1}{2}\, \left(\frac{t}{2}+\frac{3}{a}\right)\,a-1
=\frac{t}{4}+\frac{1}{2}.
$$
Similarly, if $0<t<1/(3A)$, then
$$
\left\|\left.\partial_\mathbf{s}\Psi_\beta\right|_{\mathbf{s}'}\right\|\ge 1-2tA\ge 1/3.
$$
In either case, iterated integration by parts in $\mathbf{s}$ (which is legitimate in view of the cut-off
$\chi_{\mathbf{s}_0}$), shows that the corresponding contribution to the asymptotics is $O\left(\lambda^{-\infty}\right)$.
The details are left to the reader.
\end{proof}

We have therefore
\begin{eqnarray}
 \label{eqn:fourier_transform_3}
\lefteqn{\mathcal{S}_\chi(\lambda\,\beta,\mathbf{s}_0,y,y)}\nonumber\\
&\sim&\lambda\,\int_{1/(2D)}^{2D}\int_{\mathbb{R}^r}\,
e^{i\lambda \,\Psi_\beta(y,t,\mathbf{s})}\,\chi_{\mathbf{s}_0}(\mathbf{s})\,\varrho(t)\,
s(\lambda\,t,y_\mathbf{s},y)\,\mathrm{d}\mathbf{s}\,\mathrm{d}t,
\end{eqnarray}
where now integration is compactly supported in $(t,\mathbf{s})$.

\subsubsection{Localization near $X_\beta(\mathbf{s}_0)$}

We have already shown that (\ref{eqn:fourier_transform_1}) is rapidly decreasing outside a tubular
neighborhood of $X(\mathbf{s}_0)$ of radius $D_1\,\epsilon$. The following Lemma will show
that in fact there is no loss in further restricting our analysis to an \lq oblate\rq\, tubular neighborhood
$Z_\epsilon(\mathbf{s}_0,\beta)$ of $X_\beta (\mathbf{s}_0)$.
As we shall see later, this result is instrumental to proving a considerably sharper asymptotic
confinement property.

\begin{lem}
\label{lem:double_localization}
If $D_2\gg 0$,
(\ref{eqn:fourier_transform_1}) is rapidly decreasing uniformly
for $$\mathrm{dist}_X\big(y,X(\mathbf{s}_0)\big)\le D_1\,\epsilon\,\,\,\,\,\mathrm{and}\,\,\,\,\,
\mathrm{dist}_X\big(y,X_\beta\big)\ge D_2\,\epsilon.$$
\end{lem}

\begin{proof}[Proof of Lemma \ref{lem:double_localization}]
 Suppose first that $y\in X(\mathbf{s}_0)$ and $\mathrm{dist}_X\big(y,X_\beta\big)\ge
 D_2\,\epsilon$ for some $D_2>0$. Then 
 \begin{eqnarray*}
  \big\|\partial _\mathbf{s}\Psi_\beta(y,t,\mathbf{s}_0)\big\|=\big\|t\,\Phi(m_y)-\beta\big\|
  \ge C\,D_2\,\epsilon,
 \end{eqnarray*}
where $C>0$ is as in Corollary \ref{cor:distance_comparison_1}.

Suppose now that $\|\mathbf{s}'-\mathbf{s}_0\|<\epsilon$ (as will be the case for 
$\chi_{\mathbf{s}_0}(\mathbf{s}')\neq 0$), and $\mathrm{dist}_X\big(y,X(\mathbf{s}_0)\big)\le D_1\,\epsilon$.
Pick $x\in X(\mathbf{s}_0)$ with $\mathrm{dist}_X(y,x)\le D_1\,\epsilon$. 
Then, for some appropriate $A_1>0$ we have
\begin{eqnarray*}
 \big\|\partial_\mathbf{s}\Psi_\beta(y,t,\mathbf{s}')\big\|&\ge&
  \big\|\partial _\mathbf{s}\Psi_\beta(x,t,\mathbf{s}_0,)\big\|-A_1\,\big(\|\mathbf{s}'-\mathbf{s}_0\|
 +\mathrm{dist}_X(y,x)\big)\\
 &\ge&\big[C\,D_2-A_1\,\big(1+D_1\big)\big]\,\epsilon.
\end{eqnarray*}
Again, the claim follows by iterated integration by parts in $\mathrm{d}\mathbf{s}$.

\end{proof}

In particular, since $\phi^X$ is locally free on $X_\beta$, perhaps after passing to a smaller
neighborhood we may assume without loss that it is locally free on $Z_\epsilon(\mathbf{s}_0,\beta)$.

\subsection{Asymptotic concentration in $\mathbf{w}$ and $\tau_M$}

\subsubsection{A local parametrization of $Z_\epsilon(\mathbf{s}_0,\beta)$ by HLC}

We have already remarked that $X_\beta (\mathbf{s}_0)$ is an $S^1$-bundle over the union of some
connected components of $M_\beta(\mathbf{s}_0)$, which is the transverse intersection of
$M(\mathbf{s}_0)$ and $M_\beta$ (Lemma \ref{lem:transversality}).
We can then use a smoothly varying system of HLC centered at $x\in X(\mathbf{s}_0,\beta)$
to locally
parametrize points $y\in Z_\epsilon(\mathbf{s}_0,\beta)$ as $y=x+\mathbf{v}$, where 
and $\mathbf{v}\in N_{m_x}\big(M_\beta(\mathbf{s}_0)\big)$ have norms
bounded linearly in $\epsilon$ (\S \ref{sctn:transversality_fixed_loci}); in general, this is possible only locally along
$X_\beta(\mathbf{s}_0)$.

In turn, by Lemma \ref{lem:transversality} we can uniquely decompose $\mathbf{v}$ as
an orthogonal direct sum $\mathbf{v}=\mathbf{w}+J_m\big(\xi_M(m)\big)$, where 
$\mathbf{w}\in \ker \left(d_m\phi^M_{\mathbf{s}}-\mathrm{id}_{T_mM}\right)^\perp$
and $\xi\in \ker \Phi(m)$. In addition, since $\mathbf{s}\sim \mathbf{s}_0$, we can
write $\mathbf{s}=\mathbf{s}_0+\tau$, where $\tau$, a small displacement in $\mathbb{R}^r$,
is thought of as an element of $\mathfrak{t}$. 
Here $\|\mathbf{w}\|,\,\|\xi\|,\,\|\tau\|\le D\,\epsilon$ for some appropriate constant 
$D$.
In this notation, Theorem \ref{thm:main rapid decrease} is equivalent to the statement that
$\mathcal{S}_\chi(\lambda\,\beta,\mathbf{s}_0,y,y)$ is rapidly decreasing as $\lambda\rightarrow\infty$,
unless $\max\{\|\mathbf{w}\|,\|\xi\|\}=O\left(\lambda^{\delta-1/2}\right)$.

In this section, we shall establish Theorem \ref{thm:main_directional_trace}
\lq in the $\mathbf{w}$-direction\rq, and establish that locus where 
$\|\tau_M(m)\|\ge C\,\lambda^{\delta-1/2}$ contributes negligly to the asymptotics
of (\ref{eqn:fourier_transform_1}).

\subsubsection{The bound on $\mathbf{w}$ coming $\partial_t\Psi_\beta$}

\begin{prop}
\label{prop:estimate_distance}
 There exists a constant $a>0$ such that, perhaps after passing to a smaller $\epsilon>0$,
the following holds. 
Suppose $x\in X_\beta(\mathbf{s}_0)$, 
 \begin{equation}
  \label{eqn:y parametrized}
  y=y(x,\mathbf{w},\xi)=x+\Big(\mathbf{w}+J_m\big(\xi_M(m)\big)\Big)\in Z_\epsilon (\mathbf{s}_0,\beta), 
 \end{equation}
 and $\mathbf{s}=\mathbf{s}_0+\tau$. Then we have
 \begin{equation*}
  \mathrm{dist}_X\big(y_\mathbf{s},y\big)^2\ge \mathrm{dist}_M\big(\pi(y)_\mathbf{s},\pi(y)\big)^2
 \ge  a\,\left(\|\mathbf{w}\|^2+\|\tau_M(m_x)\|^2\right).
 \end{equation*}
 
\end{prop}

Before delving into the proof, let us introduce a piece of notation.
We shall let $R_j$ be a general $\mathcal{C}^\infty$ function defined on some open neighborhood
of the origin in a Euclidean space, and vanishing to $j$-th order at the origin; 
$R_j$ is allowed to vary from line to line. Furthermore, let $A$ be as in Notation \ref{notn:representing matrix}
in \S \ref{sctn:heisenberg lc}.

\begin{proof}[Proof of Proposition \ref{prop:estimate_distance}]

The first inequality 
is obvious, since the projection $\pi:X\rightarrow M$ is a Riemannian submersion
and interwines $\phi^X$ and $\phi^M$, so let us focus on the second.

Consider the system of adapted local coordinates on $M$ centered at $m_x$, underlying
the given HLC system on $X$ centered at $x$ \cite{sz}. 
Adapted local coordinates needn't be 
holomorphic, and induce a unitary isomorphism $\mathbb{C}^d\cong T_{m_x}M$. 
It is unnecessary but convenient to assume that they are given by geodesic coordinates centered at $m_x$.
By construction, 
\begin{equation}
 \label{eqn:local_coordinates_z}
 \pi(y)=m_x+\Big(\mathbf{w}+J_m\big(\xi_M(m_x)\big)\Big).
\end{equation}
Since $\phi^X_{\mathbf{s}_0}$ is a Riemannian isometry fixing $m_x$, 
\begin{eqnarray*}
\pi(y)_{\mathbf{s}_0}&=&m_x+d_x\phi^X_{-\mathbf{s}_0}\big(\mathbf{w}+J_m\big(\xi_M(m_x)\big)\Big)\\
&=&m_x+A\,\Big(\mathbf{w}+J_m\big(\xi_M(m_x)\big)\Big)=
m_x+\Big(A\mathbf{w}+J_m\big(\xi_M(m_x)\big)\Big).
\end{eqnarray*}
where $A=A_x$ is as above. Thus, since $\mathbf{s}=\mathbf{s}_0+\tau$,
\begin{eqnarray}
 \label{eqn:local_coordinates_z_action}
 \pi(y)_{\mathbf{s}}&=&\big(\pi(y)_{\mathbf{s}_0})_\tau\\
&=&m_x+\Big(A\mathbf{w}+J_m\big(\xi_M(m_x)\big)-\tau_M(m_x)+
 \langle\tau,R_1(\mathbf{w},\xi,\tau)\rangle\Big).\nonumber
\end{eqnarray}

Now, since adapted local coordinates are isometric at the origin, perhaps after restricting the domain of definition
we may assume that
\begin{equation}
 \label{eqn:comparison_norm_HLC}
 2\,\|\mathbf{v}_1-\mathbf{v}_2\|\ge 
\mathrm{dist}_M\big(m_x+\mathbf{v}_1,m_x+\mathbf{v}_2)\ge \frac{1}{2}\,\|\mathbf{v}_1-\mathbf{v}_2\|.
\end{equation}
Thus we see from (\ref{eqn:local_coordinates_z}), (\ref{eqn:local_coordinates_z_action}) and
(\ref{eqn:comparison_norm_HLC}) that for some appropriate constant $a>0$
\begin{eqnarray*}
 \lefteqn{\mathrm{dist}_M\big(\pi(y)_\mathbf{s},\pi(y)\big)^2}\\
 &\ge&\frac{1}{2+O(\epsilon)}\,\left[\|A\mathbf{w}-\mathbf{w}\|^2+\|\tau_M(m_x)\|^2\right]\ge 
 a\,\left(\|\mathbf{w}\|^2+\|\tau_M(m_x)\|^2\right),
\end{eqnarray*}
since $A-I$ is invertible on $\ker \left(A-I\right)^\perp$.
\end{proof}

\begin{cor}
\label{cor:rapid_decrease_w}
Let $y=y(x,\mathbf{w},\xi)$ be as in (\ref{eqn:y parametrized}). Then for any fixed $D,\,\delta>0$
uniformly for $\|\mathbf{w}\|\ge D\,\lambda^{\delta-1/2}$ we have 
 $$
 \mathcal{S}_\chi(\lambda\,\beta,\mathbf{s}_0,y,y)=O\left(\lambda^{-\infty}\right).
 $$
\end{cor}

\begin{proof}[Proof of Corollary \ref{cor:rapid_decrease_w}]
In view of (\ref{eqn:phase_Psi_beta}), we have 
$$
\partial _t\Psi_\beta(y,t,\mathbf{s})=\psi (y_\mathbf{s},y).
$$
Thus, 
$
|\partial _t\Psi_\beta(y,\mathbf{s})|\ge |\Im \psi (y_\mathbf{s},y)|$. On the other hand,
by Corollary 1.3 of \cite{bdm_sj} 
for some $D'>0$ we have
$$
\Im \psi (y_\mathbf{s},y)\ge D'\,\mathrm{dist}_X(y_\mathbf{s},y)^2.
$$
Therefore, given Proposition \ref{prop:estimate_distance} and recalling that HLC are isometric at the origin,
\begin{equation}
 \label{eqn:t_derivative_bound}
\big| \partial _t\Psi_\beta(y,t,\mathbf{s})\big|=|\psi (y_\mathbf{s},y)|\ge |\Im \psi (y_\mathbf{s},y)|
\ge b\,\|\mathbf{w}\|^2
\end{equation}
for some constant $b>0$. For $\|\mathbf{w}\|\ge D\,\lambda^{\delta-1/2}$, therefore, we
get
$$
\big| \partial _t\Psi_\beta(y,t,\mathbf{s})\big|\ge b\,D\,\lambda^{2\,\delta-1}.
$$
Hence
iterated integration by
parts in $t$ introduces at each step a factor $\lambda^{1-2\,\delta}\,\lambda^{-1}=\lambda^{-2\,\delta}$.
 
\end{proof}

The same argument proves the following:

\begin{cor}
\label{cor:rapid_decrease_tau_M}
 For any fixed $D>0$, the locus $(\tau,t)$ where $\|\tau_M(m_x)\|\le D\,\lambda^{\delta-1/2}$
 contributes negligibly to the asymptotics of $\mathcal{S}_\chi(\lambda\,\beta,\mathbf{s}_0,y,y)$.
 $$
 \mathcal{S}_\chi(\lambda\,\beta,\mathbf{s}_0,y,y)=O\left(\lambda^{-\infty}\right),
 $$
 uniformly for $\|\tau_M(m_x)\|\ge D\,\lambda^{\delta-1/2}$.
\end{cor}

Thus without loss we may assume from now on that for some fixed $D>0$
\begin{equation}
 \label{eqn:restriction in w}
 \|\mathbf{w}\|\le D\,\lambda^{\delta-1/2},
\end{equation}
and restrict our analysis of the oscillatory integral 
(\ref{eqn:fourier_transform_3}) to the locus
\begin{equation}
 \label{eqn:w_tau_M_reduction}
 \left\{\tau\in \mathfrak{t}:\|\tau_M(m_x)\|\le D\,\lambda^{\delta-1/2}\right\}.
\end{equation}
More precisely, we have
 
\begin{cor}
 Only a rapidly decreasing contribution to the asymptotics of $\mathcal{S}_\chi(\lambda\,\beta,\mathbf{s}_0,y,y)$
 is lost,
 if the integrand is multiplied by a rescaled cut-off function of the form
 $\rho\left(\lambda^{1/2-\delta}\,\tau_M(m_x)\right)$, where $\rho$ is
 compactly supported in a neighborhood of the origin, and identically equal to $1$ sufficiently
 close to $\mathbf{0}\in \mathbb{R}^r$.
\end{cor}

Let us set $\mathbf{s}_\tau=:\mathbf{s}_0+\tau$ for brevity. Then we get
 \begin{eqnarray}
 \label{eqn:fourier_transform_scaled cut off}
\lefteqn{\mathcal{S}_\chi(\lambda\,\beta,\mathbf{s}_0,y,y)}\\
&\sim&\lambda\,\int_{1/(2D)}^{2D}\int_{\mathfrak{t}}\,
e^{i\lambda \,\Psi_\beta(y,t,\mathbf{s}_\tau)}\,\rho\left(\lambda^{1/2-\delta}\,\tau_M(m_x)\right)\,\chi(\tau)\,
\varrho(t)\,
s(\lambda\,t,y_{\mathbf{s}_\tau},z)\,\mathrm{d}\tau\,\mathrm{d}t.\nonumber
\end{eqnarray}

However, the latter asymptotic equality does not yet allow us to reduce to the case
$\|\tau\|\le D\,\lambda^{\delta-1/2}$, because the evaluation map
$\tau\mapsto \tau_M(m_x)$ needn't be injective.

\subsubsection{Domain concentration in $\tau$ coming from $\mathrm{dist}_X(y_\mathbf{s},y)$ and $\partial_t\Psi_\beta$}

With $\mathbf{s}_\tau=\mathbf{s}_0+\tau$ as above, and a constant $D>0$,
let us set
$$
B_\lambda(y)=:\left\{\tau\in \mathfrak{t}:\,\mathrm{dist}_X(y_{\mathbf{s}_\tau},y)\ge D\,\lambda^{\delta-1/2}\right\}.
$$

The same argument used in the proof of Corollary \ref{cor:rapid_decrease_w} implies the following

\begin{cor}
\label{cor:A B lambda decomposition}
 The locus $B_\lambda(y)$ contributes negligibly to the asymptotics of 
 $\mathcal{S}_\chi(\lambda\,\beta,\mathbf{s}_0,y,y)$
\end{cor}

Using first Lemma 3.2 of \cite{pao_JMP} and then Corollary 2.2 of \cite{pao_IJM},
with the given choice of geodesic adapted coordinates centered at $m_x$ we get
\begin{eqnarray}
 \label{eqn:local_coordinate_action}
y_{\mathbf{s}_\tau}&=&\phi^X_{-\tau-\mathbf{s}_0}\Big(x+\big(\mathbf{w}+J_m\big(\xi_M(m_x)\big)\Big)\nonumber\\
&=&\phi^X_{-\tau}\circ \phi^X_{-\mathbf{s}_0}\Big(x+\big(\mathbf{w}+J_{m_x}\big(\xi_M(m_x)\big)\Big)\nonumber\\
&=&\phi^X_{-\tau}\Big(x+\Big(R_3(\mathbf{w},\xi), A\mathbf{w}+J_m\big(\xi_M(m_x)\Big)\Big)\nonumber\\
&=&x+\Big(\langle \Phi(m_x),\tau\rangle +\omega_{m_x}\Big(\tau_M(m_x), 
A\mathbf{w}+J_{m_x}\big(\xi_M(m_x)\Big)+R_3(\tau,\mathbf{w},\xi),\nonumber\\
&&A\mathbf{w}+J_{m_x}\big(\xi_M(m_x)\big)-\tau_M(m_x)+R_2(\tau,\mathbf{w},\xi)\Big)\nonumber\\
&=&x+\Big(\langle \Phi(m_x),\tau\rangle +g_{m_x}\big(\tau_M(m_x), 
\xi_M(m_x)\big)+R_3(\tau,\mathbf{w},\xi),\nonumber\\
&&A\mathbf{w}+J_m\big(\xi_M(m_x)\big)-\tau_M(m_x)+R_2(\tau,\mathbf{w},\xi)\Big),
\end{eqnarray}
where $A$ is again as in \S \ref{sctn:heisenberg lc}.
We have used that $\tau_M(m_x)$ and $A\mathbf{w}$ live in orthogonal complex subspaces (with respect to the
Hermitian structure of $T_{m_x}M$), and therefore are symplectically orthogonal as well.

Since HLC are isometric at the origin, perhaps after restricting the domain of definition we have 
\begin{eqnarray}
 \label{eqn:comparison_norm_HLC X}
 \lefteqn{2\,\big\|(\theta_1-\theta_2,\mathbf{v}_1-\mathbf{v}_2)\big\|\ge  
\mathrm{dist}_X\big(x+(\theta_1,\mathbf{v}_1),x+(\theta_2,\mathbf{v}_2)\big)}\nonumber\\
&\ge& \frac{1}{2}\,\big\|(\theta_1-\theta_2,\mathbf{v}_1-\mathbf{v}_2)\big\|=
\frac 12\,\sqrt{(\theta_1-\theta_2)^2+\|\mathbf{v}_1-\mathbf{v}_2\|^2}.
\end{eqnarray}
Thus we see from (\ref{eqn:local_coordinate_action}) that, with $y$
as in (\ref{eqn:y parametrized}),
\begin{eqnarray}
 \label{eqn:distance_comparison_2}
 \lefteqn{\mathrm{dist}_X(y_{\mathbf{s}_\tau},y)}\\
 &\ge&\frac 12\Big|\langle \Phi(m_x),\tau\rangle +g_{m_x}\big(\tau_M(m_x), 
\xi_M(m_x)\big)\Big|+R_3(\tau,\mathbf{w},\xi).\nonumber
\end{eqnarray}

By virtue of Corollary \ref{cor:A B lambda decomposition},
we obtain the following:
\begin{lem}
\label{lem:non trivial asymptotics}
Let $y=y(x,\mathbf{w},\xi)$ be as in (\ref{eqn:y parametrized}).
Given a constant $E>0$, 
the locus of those $\tau\in \mathfrak{t}$ such that
$$
\Big|\langle \Phi(m_x),\tau\rangle +g_{m_x}\big(\tau_M(m_x), 
\xi_M(m_x)\big)\Big|+R_3(\tau,\mathbf{w},\xi)\ge E\,\lambda^{\delta-1/2},
$$
contributes negligibly to the asymptotics of $\mathcal{S}_\chi(\lambda\,\beta,\mathbf{s}_0,y,y)$.
\end{lem}

We conclude that there is no loss of generality in further restricting integration in $\mathrm{d}\tau$ in
(\ref{eqn:fourier_transform_scaled cut off}) to the locus in $\mathrm{t}$ where
\begin{equation}
 \label{eqn:further tau reduction}
 \Big|\langle \Phi(m_x),\tau\rangle +g_{m_x}\big(\tau_M(m_x), 
\xi_M(m_x)\big)\Big|+R_3(\tau,\mathbf{w},\xi)< E\,\lambda^{\delta-1/2}.
\end{equation}

\noindent
This may be accomplished $\mathcal{C}^\infty$-wise by redefining $\rho$ if necessary
(but with same type of scaling), and will be
left implicit in the following.

\subsubsection{Domain concentration coming from $\partial_\tau\psi_\beta$}

Since $\Psi_\beta$ is complex valued, $\partial_\tau\Psi_\beta(y,t,\mathbf{s}_\tau)\in \mathfrak{t}^\vee\otimes \mathbb{C}$
(recall that $\mathbf{s}_\tau=\mathbf{s}_0+\tau$).
Let us now define 
\begin{eqnarray}
 \label{eqn:definition of A_lambda}
A'_\lambda(y)&=:&\left\{(t,\tau):\,\|\partial_\tau\Psi_\beta(y,t,\mathbf{s}_\tau)\|
<2 D\,\lambda^{\delta-1/2}\right\},\\
A''_\lambda(y)&=:&\left\{(t,\tau):\,\|\partial_\tau\Psi_\beta(y,t,\mathbf{s}_\tau)\|
>D\,\lambda^{\delta-1/2}\right\}.\nonumber
\end{eqnarray}

Then $\{A'_\lambda(y,\tau),A''_\lambda(y,\tau)\}$ is an open cover of $\mathfrak{t}$, and we may find a
partition of unity subordinate to it, $\{\varsigma_\lambda,1-\varsigma_\lambda\}$ of the form
$$
\varsigma_\lambda(t,\tau)=:\varsigma\left(\lambda^{1/2-\delta}\,\partial_\tau\Psi_\beta(y,t,\mathbf{s}_\tau)\right),
$$
for an appropriate bump function $\varsigma\in \mathcal{C}^\infty_0\left(\mathfrak{t}^\vee\otimes \mathbb{C}\right)$, 
supported in an open ball of radius $2D$ centered at the origin $\mathbf{0}\in \mathfrak{t}$, 
and identically equal to $1$ within distance $D$ from the origin.
We can then rewrite (\ref{eqn:fourier_transform_scaled cut off}) as 
\begin{equation}
 \label{eqn:cut off splits}
 \mathcal{S}_\chi(\lambda\,\beta,\mathbf{s}_0,y,y)=
 \mathcal{S}_\chi(\lambda\,\beta,\mathbf{s}_0,y,y)'+\mathcal{S}_\chi(\lambda\,\beta,\mathbf{s}_0,y,y)'',
\end{equation}
where  $\mathcal{S}_\chi(\lambda\,\beta,\mathbf{s}_0,y,y)'$ and $\mathcal{S}_\chi(\lambda\,\beta,\mathbf{s}_0,y,y)''$
are given by (\ref{eqn:fourier_transform_scaled cut off}), but
with the integrand multiplied by $\varsigma_\lambda(\tau)$ and $1-\varsigma_\lambda(\tau)$, respectively.

\begin{prop}
\label{prop:asymptotic decrease 2nd term}
$\mathcal{S}_\chi(\lambda\,\beta,\mathbf{s}_0,y,y)''=O\left(\lambda^{-\infty}\right)$
for $\lambda\rightarrow\infty$. 
\end{prop}

\begin{proof}[Proof of Proposition \ref{prop:asymptotic decrease 2nd term}]
Let us define
$$
Z(\partial_\tau\Psi_\beta,y)=:\big\{\tau\in \mathfrak{t}:
\partial_\tau\Psi_\beta(y,\mathbf{s}_\tau)=\mathbf{0}\big\}.
$$
Let $(X_j)$ be the standard linar coordinates on $\mathfrak{t}\cong \mathbb{R}^r$.
On $\mathfrak{t}\setminus Z(\partial_\tau\Psi_\beta,y)$, we may consider the differential operator
$$
L=:\frac{1}{\sum_{j=1}^r\big|\partial_{X_j}\Psi_\beta(y,\mathbf{s}_\tau)\big|^2}\,
\sum_j\partial_{X_j}\overline{\Psi_\beta(y,\mathbf{s}_\tau)}\,\partial_{X_j}.
$$
Then $L(\Psi_\beta)=1$, and so $L\left(e^{i\lambda\,\Psi_\beta}\right)=i\lambda\,e^{i\lambda\,\Psi_\beta}$.
Let us also define
$$
\rho_\lambda(\tau)=:\rho\left(\lambda^{1/2-\delta}\,\tau_M(m_x)\right)\,\big(1-\varsigma_\lambda(\tau)\big)
\,\,\,\,(\tau\in \mathfrak{t}),
$$  
$$
\mathcal{A}_\lambda(y,\tau,t)=:\rho_\lambda(\tau)\,
\varrho(t)\,\chi(\tau)\,
s(\lambda\,t,y_{\mathbf{s}_\tau},y).
$$
Then we obtain
\begin{eqnarray}
 \label{eqn:integrazione per parti}
\lefteqn{ \mathcal{S}_\chi(\lambda\,\beta,\mathbf{s}_0,y,y)''}\\
&\sim&\lambda\,\int_{1/(2D)}^{2D}\int_{\mathfrak{t}}\,
e^{i\lambda \,\Psi_\beta(y,t,\mathbf{s}_\tau)}\,\mathcal{A}_\lambda(y,\tau,t)\,\mathrm{d}\tau\,\mathrm{d}t\nonumber\\
&=&-i\,\int_{1/(2D)}^{2D}\int_{\mathfrak{t}}\,
L\left(e^{i\lambda \,\Psi_\beta(y,t,\mathbf{s}_\tau)}\right)\,\mathcal{A}_\lambda(y,\tau,t)\,\mathrm{d}\tau\,\mathrm{d}t\nonumber\\
&=&i\,\sum_j\int_{1/(2D)}^{2D}\int_{\mathfrak{t}}\,e^{i\lambda \,\Psi_\beta(y,t,\mathbf{s}_\tau)}
\partial_{X_j}\left(
\frac{\overline{\partial_{X_j}\Psi_\beta(y,\mathbf{s}_\tau)}}{\sum_l\big|\partial_{X_l}\Psi_\beta(y,\mathbf{s}_\tau)\big|^2}\,
\mathcal{A}_\lambda(y,\tau,t)\right)\,\mathrm{d}\tau\,\mathrm{d}t\nonumber\\
&=&-i\,\sum_j\int_{1/(2D)}^{2D}\int_{\mathfrak{t}}\,e^{i\lambda \,\Psi_\beta(y,t,\mathbf{s}_\tau)}
P\big(
\mathcal{A}_\lambda(y,\tau,t)\big)\,\mathrm{d}\tau\,\mathrm{d}t,\nonumber
\end{eqnarray}
where $P=L^{\mathrm{t}}$ is the transpose operator,
given by
$$
P(h)=:-\sum_{j=1}^r\partial_{X_j}
\left(\frac{\partial_{X_j}\overline{\Psi_\beta(y,\mathbf{s}_\tau)}}{\sum_l\big|\partial_{X_l}\Psi_\beta(y,\mathbf{s}_\tau)\big|^2}\,
\cdot h\right).
$$ 
\noindent
Using the asymptotic expansion of $s(\lambda\,t,z_{\mathbf{s}_0+\tau},z)$, one sees that the integrand on the last line
of (\ref{eqn:integrazione per parti}) is bounded by $C_k\,\lambda^{d+1-2\delta}$.

Iterating the integration by parts in $\tau$, as above, we obtain for any $k\ge 1$
\begin{eqnarray}
 \label{eqn:integrazione per parti k}
\lefteqn{ \mathcal{S}_\chi(\lambda\,\beta,\mathbf{s}_0,y,y)''}\\
&\sim&(-i)^k\,\lambda^{1-k}\int_{1/(2D)}^{2D}\int_{\mathfrak{t}}\,
L^k\left(e^{i\lambda \,\Psi_\beta(y,t,\mathbf{s}_\tau)}\right)\,\mathcal{A}_\lambda(y,\tau,t)\big)\,\mathrm{d}\tau\,\mathrm{d}t\nonumber\\
&=&(-i)^k\,\lambda^{1-k}\int_{1/(2D)}^{2D}\int_{\mathfrak{t}}\,
e^{i\lambda \,\Psi_\beta(y,t,\mathbf{s}_\tau)}\,P^k\big(
\mathcal{A}_\lambda(y,\tau,t)\big)\,\mathrm{d}\tau\,\mathrm{d}t\nonumber,\nonumber
\end{eqnarray}

One can then check inductively the following:

\begin{lem}
 \label{lem:general term bound}
Let us set $V_j=:\partial_{X_j}\Psi_\beta(y,\mathbf{s}_\tau)$, and $V=(V_j)$. 
Then, for any $k\ge 1$, $P^k\big(
\mathcal{A}_\lambda(\tau,t)\big)$ is a linear combination of terms of the form
$$
\frac{\mathcal{P}_a\left(V,\overline{V}\right)}{\|V\|^{2b}}\,\lambda^{c(1/2-\delta)}\,B_\lambda(\tau,t),
$$
where $\mathcal{P}_a$ is a homogeneous polynomial of degree $a$, with coefficients
depending on the derivatives of $V$, and $a,b,c\in \mathbb{N}$, $2b-a+c\le 2k$; also,
$|B_\lambda|
\le C'_{a,b,c}\,\lambda^d$ for $\lambda\gg 0$.
\end{lem}

\noindent
The bound on $B_\lambda$ follows from the asymptotic expansion for the amplitude $s$ of $\Pi$
in (\ref{eqn:szego_microlocal}).

On the other hand, in view of the definition of $A''_\lambda(y)$ in (\ref{eqn:definition of A_lambda}),
on the support of $1-\varsigma_\lambda(\tau)$
each summand in Lemma \ref{lem:general term bound} satisfies an estimate of the form
\begin{eqnarray*}
\lefteqn{\left|\frac{\mathcal{P}_a(V)}{\|V\|^{2b}}\,\lambda^{c(1/2-\delta)}\,B_\lambda(\tau,t)\right|
\le
C_{a,b,c}\,\frac{\|V\|^a}{\|V\|^{2b}}\,\lambda^{d+c(1/2-\delta)}}\\
&=&
C_{a,b,c}\,\frac{1}{\|V\|^{2b-a}}\,\lambda^{d+c(1/2-\delta)}\le
D_{a,b,c}\,\lambda^{d+k(1-2\delta)}
\end{eqnarray*}
as $\lambda\rightarrow+\infty$.
Inserting this in (\ref{eqn:integrazione per parti k}), we obtain an upper bound of
the form $C_k\,\lambda^{d+1-2k\delta}$. This completes the proof of the Proposition.
\end{proof}

We conclude from (\ref{eqn:cut off splits}) and Proposition \ref{prop:asymptotic decrease 2nd term}
that 
\begin{eqnarray}
 \label{eqn:cut off reduction}
 \mathcal{S}_\chi(\lambda\,\beta,\mathbf{s}_0,y,y)&\sim& 
 \mathcal{S}_\chi(\lambda\,\beta,\mathbf{s}_0,y,y)'\\
 &\sim&\lambda\,\int_{1/(2D)}^{2D}\int_{\mathfrak{t}}\,
e^{i\lambda \,\Psi_\beta(y,t,\mathbf{s}_\tau)}\,\mathcal{B}_\lambda(y,\tau,t)\,\mathrm{d}\tau\,\mathrm{d}t\nonumber
\end{eqnarray}
where now
\begin{equation}
 \label{eqn:definition of B_lambda}\mathcal{B}_\lambda(y,\tau,t)=:\rho\left(
\lambda^{1/2-\delta}\,
\tau_M(m_x)\right)\,\,\varsigma_\lambda(t,\tau)\,
\varrho(t)\,\chi(\tau)\,
s(\lambda\,t,y_{\mathbf{s}_\tau},y).
\end{equation}

The domain of integration in (\ref{eqn:cut off reduction}) is then $A'_\lambda(y)$.

\subsubsection{The reduction in $\tau$ and the bound in $\xi$}

We shall now combine (\ref{eqn:w_tau_M_reduction}), (\ref{eqn:further tau reduction})
and the domain reduction obtained in (\ref{eqn:cut off reduction}), 
always assuming (\ref{eqn:restriction in w}).

Since $\tau\mapsto \tau_X(x)$ is injective whenever $x\in X_\beta$,
we have $\|\tau_X(x)\|\ge a\,\|\tau\|$  for some
constant $a=a_\beta>0$, depending only on $\beta$.

On the other hand, in HLC centered at $x$ we have
$$\|\tau_X(x)\|=\Big\|\big(\langle\Phi(m_x),\tau\rangle,-\tau_M(m_x)\big)\Big\|.$$
On the domain of integration of (\ref{eqn:fourier_transform_scaled cut off}),
we are in the range (\ref{eqn:w_tau_M_reduction}); therefore,
\begin{equation}
 \label{eqn:estimate tau vertical}
 \big|\langle\Phi(m_x),\tau\rangle\big|\ge a\,\|\tau\| +O\left(\lambda^{\delta-1/2}\right).
\end{equation}

On the other hand, we are now assuming that
on the same domain (\ref{eqn:further tau reduction}) also holds;
combining (\ref{eqn:estimate tau vertical}) with (\ref{eqn:further tau reduction}), we conclude that
on the domain of integration of (\ref{eqn:fourier_transform_scaled cut off}), further reduced according to
(\ref{eqn:further tau reduction}),
we have for appropriate constants $D_1,\,D_1'>0$:
\begin{eqnarray*}
E\,\lambda^{\delta-1/2}&\ge& \left|\Big(\langle \Phi(m_x),\tau\rangle +g_m\big(\tau_M(m_x), 
\xi_M(m_x)\Big)+R_3(\tau,\mathbf{w},\xi)\right|\\
&\ge &D_1\,\|\tau\|+R_3(\tau,\mathbf{w},\xi)+O\left(\lambda^{\delta-1/2}\right)\\
&\ge&D_1'\,\|\tau\|+R_3(\xi)+O\left(\lambda^{\delta-1/2}\right).
\end{eqnarray*}
We have used that $\|\tau\|\gg R_j(\tau)$ for $j=2,3$ e $\|\tau\|<\epsilon$, $\epsilon$ small.
Therefore we obtain the following:
\begin{lem}
 \label{lem:bound tau xi}
 In the domain of integration of (\ref{eqn:fourier_transform_scaled cut off}), and with the
 reduction (\ref{eqn:further tau reduction}) implicit, for some constant $D_5>0$ we have
 \begin{equation*}
 \|\tau\|\le D_5\,\|\xi\|^3+O\left(\lambda^{\delta-1/2}\right).
\end{equation*}
\end{lem}

On the other hand,
in view of (\ref{eqn:comparison_norm_HLC}),
given (\ref{eqn:restriction in w}) and (\ref{eqn:w_tau_M_reduction}) we have 
$$
\mathrm{dist}_M\big(\pi(y)_\mathbf{s},\pi(y)\big)\le 4D\,\lambda^{\delta-1/2}.
$$
Given that $\pi:X\rightarrow M$ is a Riemannian submersion with fibers the $S^1$-orbits in $X$,
there exists $\vartheta=\vartheta(y,\mathbf{s})\in (-\pi,\pi]$
such that
$$
\mathrm{dist}_X\left(y_\mathbf{s},e^{i\vartheta}\,y\right)\le 4D\,\lambda^{\delta-1/2}.
$$
In view of Remark \ref{rem:diff_psi}, identifying $\mathrm{d}\psi$ with its local coordinate expression,
\begin{eqnarray}
 \label{eqn:estimate_diff_psi}
\mathrm{d}_{(y_\mathbf{s},y)}\psi&=&\mathrm{d}_{(e^{i\vartheta}\,y,y)}\psi+O\left(\lambda^{\delta-1/2}\right)
\nonumber\\
&=&\left(e^{i\vartheta}\,\alpha_{e^{i\vartheta}\,y},\,-e^{i\vartheta}\,\alpha_{y}\right)+
O\left(\lambda^{\delta-1/2}\right)\nonumber\\
&=&\left(e^{i\vartheta}\,\alpha_{y},\,-e^{i\vartheta}\,\alpha_{y}\right)+
O\left(\lambda^{\delta-1/2}\right),
\end{eqnarray}
where on the last line we have used that $\alpha$ is $S^1$-invariant, and therefore it does not depend
on the $\theta$-coordinate in a HLC system (recall that in HLC the $S^1$-action on $X$ is expressed by a
translation in the angular coordinate).

Given $\xi\in \ker\Phi(m)$, let us introduce the linear functional on $\mathfrak{t}$
$$
L_m(\xi):\tau\mapsto g_m\big(\tau_M(m),\xi_M(m)\big).
$$
Given (\ref{eqn:estimate_diff_psi}) and (\ref{eqn:local_coordinate_action}), we see from
(\ref{eqn:phase_Psi_beta}) that
\begin{eqnarray*}
\partial_\tau\Psi_\beta(y,t,\mathbf{s}_\tau)&=&e^{i\vartheta}\,\Big(t\,\Phi(m_x)+L_m(\xi)\Big)-\beta+R_2(\tau,\xi)+
O\left(\lambda^{\delta-1/2}\right)\nonumber\\
&=&\left[e^{i\vartheta}\,t\,\Phi(m_x)-\beta\right]+ e^{i\vartheta}\, L_{m_x}(\xi)+
R_2(\tau,\xi)+O\left(\lambda^{\delta-1/2}\right).
\end{eqnarray*}

\begin{lem}
\label{lem:estimate tau xi}
In the range of the present discussion, 
 \begin{equation*}
\big\| \partial_\tau\Psi_\beta(y,t,\mathbf{s}_\tau)\big\|
\ge
b\,\|\xi\|+R_2(\tau)+
O\left(\lambda^{\delta-1/2}\right)
\end{equation*}
for some constant $b>0$.
\end{lem}

\begin{proof}[Proof of Lemma \ref{lem:estimate tau xi}]
We have
$$
e^{i\vartheta}\,t\,\Phi(m)-\beta=
e^{i\vartheta}\,t\,\Phi(m)-\Phi_u(m)\in \mathrm{span}_{\mathbb{C}}\big\{\Phi(m)\big\}\subseteq 
\mathfrak{t}^\vee\otimes\mathbb{C},
$$
while every non-zero element of
$$
\mathcal{L}_m=:\Big\{L_m(\xi)\,:\,\xi\in \ker\Phi(m)\Big\}\otimes\mathbb{C}\subseteq \mathfrak{t}^\vee\otimes\mathbb{C}
$$
is non-vanishing on $\ker\Phi(m)$, as the evaluation map $\mathrm{val}_m:\mathfrak{t}\rightarrow T_mM$
is injective on $\ker\Phi(m)$; in particular $L_m(\xi)(\xi)=\|\xi_M(m)\|^2>0$ for any $\xi\in \ker\Phi(m)$, $\xi\neq 0$.
Hence $\xi\in \ker\Phi(m)\mapsto L_m(\xi)\in \mathcal{L}_m$ is an isomorphism, and
$$\mathcal{L}_m\cap \mathrm{span}\{\Phi(m)\}=(\mathbf{0});
$$
this implies for some constants $a_1,\,a_2>0$ and every $\xi,\,t$
$$
\left\|L_m(\xi)+\Big(e^{i\vartheta}\,t\,\Phi(m)-\beta\Big)\right\|\ge 
a_1\,\left(\left\|L_m(\xi)\right\|+\Big\|e^{i\vartheta}\,t\,\Phi(m)-\beta\Big\|\right)
\ge a_2\,\|\xi\|.
$$
To complete the proof, we need only remark that 
$\|\xi\|\gg R_2(\xi),\,R_1(\tau)\,R_1(\xi)$, since $\tau$ and $\xi$ are bounded linearly in $\epsilon$,
and $\epsilon$ is assumed very small.
\end{proof}

On the domain of integration $A'_\lambda(y)$, we then obtain
\begin{equation}
 \label{eqn:2nd bound}
\|\xi\|+R_2(\tau)=O\left(\lambda^{\delta-1/2}\right)\,\Longrightarrow\,\|\xi\|\le A\,\|\tau\|^2
+O\left(\lambda^{\delta-1/2}\right).
\end{equation}

Pairing (\ref{eqn:2nd bound}) with the bound in Lemma \ref{lem:bound tau xi}, we obtain first that
that in the domain of integration of (\ref{eqn:cut off reduction}) we have
\begin{equation}
 \label{eqn:reduction tau xi}
 \|\tau\|\le C'\,\lambda^{\delta-1/2},\,\,\,\,\,\,\|\xi\|\le C''\,\lambda^{\delta-1/2}
\end{equation}
for appropriate constants $C',\,C''>0$.

\subsubsection{Proof of Theorem \ref{thm:main rapid decrease}}

Summing up, we have established that for every $\delta\in (0,1/2)$ and any given positive constant
$a_\delta>0$ there exists $b_\delta>0$ such that
$\mathcal{S}_\chi(\lambda\,\beta,\mathbf{s}_0,y,y)=O\left(\lambda^{-\infty}\right)$ 
with $y=y(x,\mathbf{w},\xi)$ as in (\ref{eqn:y parametrized}),
if 
$\|\mathbf{w}\|\ge a_\delta\,\lambda^{\delta-1/2}$ or $\|\xi\|\ge b_\delta\,\lambda^{\delta-1/2}$
for $\lambda\gg 0$.
Let us now choose an arbitrary constant $a>0$ and suppose that 
$\max\big\{\|\mathbf{w}\|,\,\|\xi\|\big\}\ge a\,\lambda^{\delta-1/2}$.
Choose $\delta'\in (0,\delta)$. Then for $\lambda\gg 0$ we have
$$\max\big\{\|\mathbf{w}\|,\,\|\xi\|\big\}\ge a\,\lambda^{\delta-1/2}> 
\max\{a_{\delta'},b_{\delta'}\}\,\lambda^{\delta'-1/2}.$$
We conclude the following:

\begin{cor}
 \label{cor:equivalent rapid decay}
 For any positive constant $a>0$, we have
 $\mathcal{S}_\chi(\lambda\,\beta,\mathbf{s}_0,y,y)=O\left(\lambda^{-\infty}\right)$ 
with $y=y(x,\mathbf{w},\xi)$,
uniformly in $(\mathbf{w},\xi)$ satisfying 
$\max\big\{\|\mathbf{w}\|,\,\|\xi\|\big\}\ge a\,\lambda^{\delta-1/2}$
for $\lambda\gg 0$.
\end{cor}

\begin{proof}[Proof of Theorem \ref{thm:main rapid decrease}]
If (\ref{eqn:assumption lower bound distance}) holds with $y=y(x,\mathbf{w},\xi)$,
then in view of (\ref{eqn:comparison_norm_HLC X}) we need to have
\begin{eqnarray*}
 \max\{\|\mathbf{w}\|,\|\xi\|\}&\ge&\frac{1}{\sqrt{2}}\,\sqrt{\|\mathbf{w}\|^2+\|\xi\|^2}\ge
 \frac{1}{2\sqrt{2}}\,\mathrm{dist}_X(y,x)\ge \frac{D}{2\sqrt{2}}\,\lambda^{\delta-1/2}.
\end{eqnarray*}
Thus the statement of the Theorem follows from Corollary \ref{cor:equivalent rapid decay}.

\end{proof}


\section{Proof of Theorem \ref{thm:main_directional_trace} and Corollary \ref{cor:fourier transform}}

Before delving into the proof, let us note 
that in the course of the proof of Theorem \ref{thm:main rapid decrease}
 we have established the following: in (\ref{eqn:cut off reduction}) only a negligible contribution to
the asymptotics is lost, if integration in $\tau$ is restricted to a neighborhood of origin of
radius $C'\,\lambda^{\delta-1/2}$, for some $C'>0$
(see (\ref{eqn:reduction tau xi})). Hence we may rewrite (\ref{eqn:cut off reduction}) as follows:
\begin{eqnarray}
 \label{eqn:cut off reduction Clambda}
 \mathcal{S}_\chi(\lambda\,\beta,\mathbf{s}_0,y,y)&\sim& 
 \lambda\,\int_{1/(2D)}^{2D}\int_{\mathfrak{t}}\,
e^{i\lambda \,\Psi_\beta(y,t,\mathbf{s}_\tau)}\,\mathcal{C}_\lambda(y,\tau,t)\,\mathrm{d}\tau\,\mathrm{d}t\nonumber
\end{eqnarray}
where now
\begin{equation}
 \label{eqn:definition of C_lambda}
\mathcal{C}_\lambda(y,\tau,t)=:\gamma\left(
\lambda^{1/2-\delta}\,
\tau\right)\,
\varrho(t)\,\chi(\tau)\,
s(\lambda\,t,y_{\mathbf{s}_\tau},y),
\end{equation}
with $\gamma\in \mathcal{C}^\infty_0(\mathfrak{t})$ compactly supported and identically equal to one on an appropriate neighborhood
of the origin.
In view of (\ref{eqn:phase_Psi_beta}), we can further rewrite (\ref{eqn:cut off reduction Clambda}) as follows:
\begin{eqnarray}
 \label{eqn:fourier_transform_4}
\lefteqn{\mathcal{S}_\chi(\lambda\,\beta,\mathbf{s}_0,y,y)}\\
&\sim&\lambda\,e^{-i\,\lambda\,\langle\beta,\mathbf{s}_0\rangle}\int_{1/(2D)}^{2D}\int_{\mathfrak{t}}\,
e^{i\lambda \,\Upsilon_\beta(y,t,\tau)}\,\mathcal{C}_\lambda(y,\tau,t)\,\mathrm{d}\tau\,\mathrm{d}t,\nonumber
\end{eqnarray}
where
\begin{equation}
 \label{eqn:fase_Upsilon}
 \Upsilon_\beta(y,t,\tau)=:t\,\psi (y_{\mathbf{s}_\tau},y)-\langle\beta,\tau\rangle;
\end{equation}
here $y_{\mathbf{s}_\tau}$ is given by (\ref{eqn:local_coordinate_action}).

\begin{proof}[Proof of Theorem \ref{thm:main_directional_trace}]
Let us set, in the notation of \S \ref{sctn:heisenberg lc} and with $\mathbf{v}$ as in (\ref{eqn:vtauxi}), 
\begin{equation}
 \label{defn:definition of sylambda}
\mathbf{s}_\lambda=:\mathbf{s}_0+\frac{1}{\sqrt{\lambda}}\,\tau\in \mathbb{R}^r, \,\,\,
\,y_\lambda=x+\frac{1}{\sqrt{\lambda}}\,\mathbf{v}\in X,\,\,\,\,
y_{\lambda,\mathbf{s}_\lambda}=:\phi^X_{-\mathbf{s}_\lambda}(y_\lambda)\in X,
\end{equation}
\begin{equation}
 \label{eqn:definition of mlambda}
m_\lambda=:m_x+\frac{1}{\sqrt{\lambda}}\,\mathbf{v}=\pi(y_\lambda)\in M,\,\,\,\,
m_{\lambda,\mathbf{s}_\lambda}=:\phi^M_{-\mathbf{s}_\lambda}(m_\lambda)=\pi(y_{\lambda,\mathbf{s}_\lambda})\in M.
\end{equation}

With the change of integration variable $\tau\mapsto \tau/\sqrt{\lambda}$,
(\ref{eqn:fourier_transform_4}) may be rewritten
\begin{eqnarray}
 \label{eqn:fourier_transform_5}
\lefteqn{\mathcal{S}_\chi(\lambda\,\beta,\mathbf{s}_0,y_\lambda,y_\lambda)}\\
&\sim&\lambda^{1-\frac{r}{2}}\,e^{-i\,\lambda\,\langle\beta,\mathbf{s}_0\rangle}\int_{1/(2D)}^{2D}\int_{\mathfrak{t}}\,
e^{i\lambda \,\Upsilon_\beta(y_\lambda,t,\tau/\sqrt{\lambda})}\,\mathcal{D}_\lambda\left(y,\tau,t\right)
\,\mathrm{d}\tau\,\mathrm{d}t,\nonumber
\end{eqnarray}
where
\begin{eqnarray}
 \label{eqn:rescaled amplitude}
\mathcal{D}_\lambda\left(y,\tau,t\right)&=:&\mathcal{C}_\lambda\left(y_\lambda,\frac{1}{\sqrt{\lambda}}\,\tau,t\right)
\nonumber\\
&=&\gamma\left(
\lambda^{-\delta}\,
\tau\right)\,
\varrho(t)\,\chi(\lambda^{-1/2}\,\tau)\,
s(\lambda\,t,y_{\lambda,\mathbf{s}_\lambda},y_\lambda),
\end{eqnarray}
and integration in $\mathrm{d}\tau$ is now over an expanding ball in $\mathfrak{t}$ centered at the origin and radius
$O\left(\lambda^\delta\right)$.

Let us compute the phase in (\ref{eqn:fourier_transform_5}). We have by (\ref{eqn:fase_Upsilon})
\begin{eqnarray}
 \label{eqn:phase Upsilon}
 \Upsilon_\beta\left(y_\lambda,t,\frac{\tau}{\sqrt{\lambda}}\right)=
t\,\psi \left(y_{\lambda,\mathbf{s}_\lambda},y_\lambda\right)-\frac{1}{\sqrt{\lambda}}\,\langle\beta,\tau\rangle. 
\end{eqnarray}
Now let us apply (\ref{eqn:local_coordinate_action}) with $\tau$, $\mathbf{w}$, and $\xi$
rescaled by $1/\sqrt{\lambda}$. We obtain
\begin{eqnarray}
 \label{eqn:y lambda dinamico}
 \lefteqn{y_{\lambda,\mathbf{s}_\lambda}=\phi^X_{-\frac{\tau}{\sqrt{\lambda}}-\mathbf{s}_0}(y_\lambda)}\\
 &=&
x+\left(\frac{1}{\sqrt{\lambda}}\,\langle \Phi(m_x),\tau\rangle +\frac{1}{\lambda}\, g_{m_x}\big(\tau_M(m_x), 
\xi_M(m_x)\big)+R_3\left(\frac{\tau}{\sqrt{\lambda}},\frac{\mathbf{w}}{\sqrt{\lambda}},\frac{\xi}{\sqrt{\lambda}}\right),
\right.\nonumber\\
&&\left. \frac{1}{\sqrt{\lambda}}\,\big( A\mathbf{w}+J_{m_x}\big(\xi_M(m_x)\big)-\tau_M(m_x)\big)
+R_2\left(\frac{\tau}{\sqrt{\lambda}},\frac{\mathbf{w}}{\sqrt{\lambda}},\frac{\xi}{\sqrt{\lambda}}\right)\right)\nonumber\\
&=&x+\big(\varTheta_\lambda ,V_\lambda\big),\nonumber
\end{eqnarray}
where $\varTheta_\lambda=\varTheta_\lambda (x,\mathbf{w},\xi,\tau)$ and $V_\lambda=V_\lambda(x,\mathbf{w},\xi,\tau)$
are defined by the previous equality.
In view of Remark \ref{rem:espansione per psi}, (\ref{eqn:y lambda dinamico}) implies 
\begin{eqnarray}
 \label{eqn:espansione per psi}
 \lefteqn{t\,\psi \left(y_{\lambda,\mathbf{s}_\lambda},y_\lambda\right)=
 t\,\psi \left(x+\big(\varTheta_\lambda ,V_\lambda\big),x+\frac{1}{\sqrt{\lambda}}\,
 \big[\mathbf{w}+J_{m_x}\big(\xi_M(m_x)\big)\big]\right)}\\
 &=&it\,\left[1-e^{i\varTheta_\lambda}\right]-it\,\psi_2\left(V_\lambda,\frac{1}{\sqrt{\lambda}}\,
 \big[\mathbf{w}+J_{m_x}\big(\xi_M(m_x)\big)\big]\right)\,e^{i\varTheta_\lambda}\nonumber\\
 &&+t\,R_3\left(V_\lambda,\frac{1}{\sqrt{\lambda}}\,
 \big[\mathbf{w}+J_{m_x}\big(\xi_M(m_x)\big)\big]\right)\,e^{i\varTheta_\lambda}\nonumber\\
&=&it\,\left[1-e^{i\varTheta_\lambda}\right]-\frac{it}{\lambda}\,\psi_2\Big( A\mathbf{w}+J_m\big(\xi_M(m_x)\big)-\tau_M(m_x),
\mathbf{w}+J_{m_x}\big(\xi_M(m_x)\big)\Big)\,e^{i\varTheta}\nonumber\\
&&+t\,R_3\left(\frac{\tau}{\sqrt{\lambda}},\frac{\mathbf{w}}{\sqrt{\lambda}},\frac{\xi}{\sqrt{\lambda}}\right)
\,e^{i\varTheta_\lambda}
\nonumber
\end{eqnarray}
(recall that $R_3$ is a generic $\mathcal{C}^\infty$ function vanishing to third order at the origin,
and is allowed to vary from line to line). 
Inserting (\ref{eqn:espansione per psi}) 
in (\ref{eqn:phase Upsilon}), we get with some computations
\begin{eqnarray}
 \label{eqn:phase Upsilon expanded}
\lefteqn{i\lambda\,\Upsilon_\beta\left(y_\lambda,t,\frac{\tau}{\sqrt{\lambda}}\right)}\\
&=&i\,\sqrt{\lambda}\,\big\langle t\,\Phi(m_x)-\beta,\tau\big\rangle
+t\,\left[i\,g_{m_x}\big(\xi_M(m_x),\tau_M(m_x)\big)-\frac{1}{2}\,\langle\Phi(m_x),\tau\rangle ^2\right.\nonumber\\
&&\left.+\psi_2\Big(J\big(\xi_M(m_x)\big)+A\,\mathbf{w}-\tau_M(m_x),J\big(\xi_M(m_x)\big)+\mathbf{w}\Big)\right]
\nonumber\\
&&+\lambda\,t\,R_3\left(\frac{\tau}{\sqrt{\lambda}},\frac{\mathbf{w}}{\sqrt{\lambda}},\frac{\xi}{\sqrt{\lambda}}\right).\nonumber
\end{eqnarray}

We have
\begin{eqnarray*}
\lefteqn{\psi_2\Big(J\big(\xi_M(m_x)\big)+A\,\mathbf{w}-\tau_M(m_x),J\big(\xi_M(m_x)\big)+\mathbf{w}\Big)}\\
&=&\psi_2(A\,\mathbf{w},\mathbf{w})+i\,g_{m_x}\big(\xi_M(m_x),\tau_M(m_x)\big)-\frac{1}{2}\,\big\|\tau_M(m_x)\big\|^2.
\end{eqnarray*}

In particular, on the domain of integration and in the range of the Theorem
we have for some $c>0$
\begin{eqnarray*}
 \lefteqn{\Re\Big(i\lambda\,\Upsilon_\beta\left(z_\lambda,t,\tau/\sqrt{\lambda}\right)\Big)\le -\frac{t}{2}\,\|A\mathbf{w}-\mathbf{w}\|^2}\\
 &&-\frac{t}{2}\,\left(\langle\Phi(m_x),\tau\rangle ^2+\big\|\tau_M(m_x)\big\|^2\right)
 +\Re\left(\lambda\,R_3\left(\frac{\tau}{\sqrt{\lambda}},
 \frac{\mathbf{w}}{\sqrt{\lambda}},\frac{\xi}{\sqrt{\lambda}}\right)\right)\\
 &\le& -c\,\left(\|\mathbf{w}\|^2+\|\tau\|^2\right)
 +O\left(\lambda^{3\delta-1/2}\right)=-c\,\left(\|\mathbf{w}\|^2+\|\tau\|^2\right)
 +o(1)
\end{eqnarray*}
since by assumption $0<\delta<1/6$.

Let $\Xi(m)\in \mathfrak{t}$ be as in \S \ref{sctn:intrinsic vector field}, so that we have
an orthogonal direct sum
\begin{equation}
 \label{eqn:somma diretta ortogonale}
 \mathfrak{t}=\mathrm{span}\big(\Xi(m_x)\big)\oplus \ker\big(\Phi(m_x)\big).
\end{equation}
For any $\tau\in \mathfrak{t}$, 
we can write, for unique
$u\in \mathbb{R}$ and $\eta\in \ker\Phi(m_x)$,
\begin{equation}
 \label{eqn:explicit decomposition}
 \tau = \tau(u,\eta)=:u \,\Xi(m_x)+\eta.
\end{equation}

Recalling that $\beta=\Phi_u(m)=\Phi(m)/\|\Phi(m)\|$, 
(\ref{eqn:phase Upsilon expanded}) may be rewritten as follows:
\begin{eqnarray}
 \label{eqn:phase Upsilon compact}
\lefteqn{i\lambda\,\Upsilon_\beta\left(z_\lambda,t,\tau/\sqrt{\lambda}\right)}\\
&=&i\,\sqrt{\lambda}\,\Gamma_x(t,u)
+t\,\mathcal{E}_x\big(\xi,\tau(u,\eta),\mathbf{w}\big)
+\lambda\,t\,R_3\left(\frac{\tau}{\sqrt{\lambda}},\frac{\mathbf{w}}{\sqrt{\lambda}},\frac{\xi}{\sqrt{\lambda}}\right),
\nonumber
\end{eqnarray}
where
\begin{eqnarray}
 \label{eqn:fse Gamma e ampiezza A def}
 \Gamma_x(t,u)&=:&\big(\|\Phi(m_x)\|\,t-1\big)\,u,\\
 \mathcal{E}_x\big(\xi,\tau(u,\eta),\mathbf{w}\big)&=:&\psi_2(A\,\mathbf{w},\mathbf{w})-\frac{1}{2}\,\|\Phi(m_x)\|^2\,u^2
 \nonumber\\
 &&+2i\,g_{m_x}\Big(\xi_M(m_x),\tau(u,\eta)_M(m_x)\Big)-\frac{1}{2}\,\big\|\tau(u,\eta)_M(m_x)\big\|^2.\nonumber
\end{eqnarray}
In view of (\ref{eqn:somma diretta ortogonale}) and (\ref{eqn:explicit decomposition}), 
we can write the integral over $\mathfrak{t}$ as an iterated
integral:
$$
\int_{\mathfrak{t}}\,\mathrm{d}\tau=
\int_{\ker\Phi(m_x)}\mathrm{d}\eta\,\int_{-\infty}^{+\infty}\mathrm{d}u.
$$
We shall then rewrite (\ref{eqn:fourier_transform_5}) in the following form:
\begin{eqnarray}
 \label{eqn:fourier_transform_6}
\mathcal{S}_\chi(\lambda\,\beta,\mathbf{s}_0,y_\lambda,y_\lambda)\sim
\lambda^{1-\frac{r}{2}}\,e^{-i\,\lambda\,\langle\beta,\mathbf{s}_0\rangle}\,\int_{\ker\Phi(m_x)}\,
I_\lambda(x,\eta,\mathbf{w},\xi)\,\mathrm{d}\eta,
\end{eqnarray}
where the inner integral 
\begin{eqnarray}
\label{eqn:integrale interno}
\lefteqn{I_\lambda(x,\eta,\mathbf{w},\xi)}\\
&=:&
\int_{1/(2D)}^{2D}\,\int_{-\infty}^{+\infty}\,
e^{i\sqrt{\lambda}\,\Gamma(t,u)}\,e^{t\,\mathcal{E}}\,
e^{t\,\lambda R_3}\cdot 
\mathcal{D}_\lambda\left(y,\tau,t\right)\,\mathrm{d}u\,\mathrm{dt}\nonumber
\end{eqnarray}
is an oscillatory integral in $\sqrt{\lambda}$, with the quadratic phase $\Gamma$ and an amplitude
compactly supported in an expanding ball of center the origin and radius $O\left(\lambda^{\delta}\right)$.

Combining the asymptotic expansion of the symbol $s$ in (\ref{eqn:szego_microlocal}), 
which yields
$$
s\left(\lambda\,t,x',x''\right)\sim\lambda^d\,t^d\,\sum_{j\ge 0}\lambda^{-j}\,t^{-j}\,s_j\left(x',x''\right),
$$ 
with the 
Taylor expansion of the individual factors in the amplitude of (\ref{eqn:integrale interno}) 
in the rescaled variables,
we get an asymptotic expansion of the integrand in (\ref{eqn:integrale interno}) in descending powers
of $\lambda^{-1/2}$:
\begin{eqnarray}
 \label{eqn:espansione asintotica integranda}
 \lefteqn{e^{i\sqrt{\lambda}\,\Gamma_x(t,u)}\,e^{t\,\mathcal{E}_x}\,
e^{t\,R_3}\cdot 
\mathcal{D}_\lambda\left(y,\tau,t\right)}\\
&\sim&
e^{i\sqrt{\lambda}\,\Gamma(t,u)}\,e^{t\,\mathcal{E}}\cdot\lambda^d\,
\sum_{j,l\ge 0}\lambda^{-j-l/2}\,t^{d-j}\,\mathcal{P}_{j,l}(\xi,\mathbf{w},u,\eta),\nonumber
\end{eqnarray}
where $\mathcal{P}_{j,l}$ is a homogeneous
polynomial of degree $\le 3l$ (dependence on $x$ is omitted). 
Indeed, $\lambda^\delta$ appears in (\ref{eqn:rescaled amplitude}) only in the rescaling of the bump function
$\gamma$, which is identically equal to one on a neighborhood of the origin (and thus has vanishing derivatives
to all orders at the origin). 
\footnote{The degree of $\mathcal{P}_{j,l}$ is bounded by $3l$ rather than
$l$ because of the exponent $\lambda\,R_3(\tau/\sqrt{\lambda},\mathbf{w}/\sqrt{\lambda},\xi/\sqrt{\lambda})$).}

\begin{rem}
 \label{rem:parity 1}
We have in particular
$\mathcal{P}_{0,0}(x,x)=\pi^{-d}$, and in addition in view of (\ref{eqn:rescaled amplitude})
and the exponent $\lambda\,R_3$ one concludes that $\mathcal{P}_{j,l}$ has parity $(-1)^\ell$.
\end{rem}

The
remainder at step $(j_0,l_0)$ is bounded by 
$$
\lambda^{-j_0-(1+l_0)/2}\,\mathcal{R}_{j_0,l_0}\,(\xi,\mathbf{w},u,\eta)
\,e^{-a\,(\|\mathbf{w}\|^2+\|\eta\|^2+u^2)},
$$
where again $\mathcal{R}_{j_0,l_0}$ is a polynomial of degree $l_0$.

Given that $\|\xi\|=O\left(\lambda^{\delta}\right)$, the previous expression is bounded above by 
$$C_{j_0,l_0}\,\lambda^{-j_0-(1+l_0)/2+3l_0\delta}=C_{j_0,l_0}\,\lambda^{-j_0-1/2-3l_0\,(1/6-\delta)};$$
since on the other hand integration in the inner integral is over a domain of the form
$\big(1/(2D),2D\big)\times \left(-c\,\lambda^{\delta},c\,\lambda^{\delta}\right)$, the expansion may
be integrated term by term. Thus we get

\begin{eqnarray}
\label{eqn:integrale interno espanso}
I_\lambda(x,\eta,\mathbf{w},\xi)&\sim&\sum_{j,l\ge 0}\lambda^{d-j-l/2}\,I_\lambda(x,\eta,\mathbf{w},\xi)_{j,l},
\end{eqnarray}
where
\begin{eqnarray}
\label{eqn:jl inner integral}
 I_\lambda(x,\eta,\mathbf{w},\xi)_{j,l}=:\int_{1/(2D)}^{2D}\,\int_{-\infty}^{+\infty}\,e^{i\sqrt{\lambda}\,\Gamma_x(t,u)}
 \,e^{t\,\mathcal{E}_x}\cdot
t^{d-j}\,\mathcal{P}_{j,l}(\xi,\mathbf{w},u,\eta)
\,\mathrm{d}u\,\mathrm{dt}.
\end{eqnarray}

It is immediate from (\ref{eqn:fse Gamma e ampiezza A def}) that $\Gamma$ has a unique stationary point
$$
P_0=(t_0,u_0)=:\left(\frac{1}{\|\Phi(m_x)\|},0\right),
$$
where it vanishes; the Hessian matrix at the critical point is
$$
H_{P_0}(\Gamma)=\begin{bmatrix}
                 0&\|\Phi(m_x)\|\\
\|\Phi(m_x)\|&0
                \end{bmatrix},$$
with determinant and signature
$$\det \big(H_{P_0}(\Gamma)\big)=-\|\Phi(m_x)\|^2,\,\,\,\,\mathrm{sgn}\big(H_{P_0}(\Gamma)\big)=0.$$
Therefore, the Hessian operator is given by
\begin{equation}
 \label{eqn:Hessian operator}
L_\Gamma=:\frac{i}{\|\Phi(m_x)\|}\,\frac{\partial^2}{\partial t\partial u}.
\end{equation}

Furthermore, iterated integration by parts in $(t,u)$ shows that only a bounded neighborhood
of the critical point contributes non-negligibly to the asymptotics of $I(\eta,\mathbf{w},\xi)_{j,l}$.
More precisely, let $\beta\in \mathcal{C}^\infty_0\left(\mathbb{R}^2\right)$ be a bump function
identically equal to $1$ in a neighborhood of $(t_0,u_0)$. Then we can split (\ref{eqn:jl inner integral}) as
\begin{equation}
 \label{eqn:jl inner integral split}
 I(\eta,\mathbf{w},\xi)_{j,l}=I(\eta,\mathbf{w},\xi)_{j,l}'+I(\eta,\mathbf{w},\xi)_{j,l}'',
\end{equation}
where $I(\eta,\mathbf{w},\xi)_{j,l}',\,I(\eta,\mathbf{w},\xi)_{j,l}''$ are given by
(\ref{eqn:jl inner integral}), but with the integrand multiplied by $\beta(t,u)$ and
$1-\beta(t,u)$, respectively.
Integration by parts in $(t,u)$ in $I(\eta,\mathbf{w},\xi)_{j,l}''$ as in the standard
proof of the stationary phase Lemma
is legitimate, because the integrand is compactly supported away from the critical point
(and, at any rate, bounded by a decaying exponential in $u$); on the other hand at each iteration a factor
$\lambda^{-1/2}$ is introduced, and integration is over a domain of diameter $O\left(\lambda^\delta\right)$,
and we conclude that $I(\eta,\mathbf{w},\xi)_{j,l}''=O\left(\lambda^{-\infty}\right)$.

Applying the Stationary Phase Lemma to $I(\eta,\mathbf{w},\xi)_{j,l}'$, 
we obtain an asymptotic expansion 
in (\ref{eqn:integrale interno espanso}) of the form 
\begin{eqnarray}
 \label{eqn:asymptotic expansion inner integral}
\lefteqn{\lambda^{d-j-l/2}\,I(\eta,\mathbf{w},\xi)_{j,l}\sim \frac{2\pi}{\|\Phi(m_x)\|}\cdot
\lambda^{d-1/2-j-l/2}}\\
&&\nonumber\cdot 
\sum_{a\ge 0}\lambda^{-a/2}\,\frac{1}{a!}\,\left.L_\Gamma^a\left( t^{d-j}\,e^{t\,\mathcal{E}_x}
\cdot\mathcal{P}_{j,l}(\xi,\mathbf{w},u,\eta)\right)
\right|_{t=t_0,u=u_0}.
\end{eqnarray}

Given (\ref{eqn:fse Gamma e ampiezza A def}) and (\ref{eqn:Hessian operator}), we conclude that
\begin{equation}
\label{eqn:polynomial growth}
 L_\Gamma^a\left(e^{t\,\mathcal{E}_x}\right)=\mathcal{Q}_a(x,t,u;\xi,\mathbf{w},\eta)\,e^{t\,\mathcal{E}_x},
\end{equation}
where $\mathcal{Q}_a(x,t,u;\xi,\mathbf{w},\eta)$ is a polynomial in $(\xi,\mathbf{w},\eta)$, of degree
$\le 3a$.
It follows that
\begin{equation}
\label{eqn:polynomial growth 1}
 L_\Gamma^a\left(t^{d-j}\,e^{t\,\mathcal{E}_x}\cdot\mathcal{P}_{j,l}(\xi,\mathbf{w},u,\eta)\right)
 =\mathcal{R}_{j,l,a}(x,t;\xi,\mathbf{w},u,\eta)\,e^{t\,\mathcal{E}_x},
\end{equation}
where $\mathcal{R}_{j,l,a}$ is a polynomial in $(\xi,\mathbf{w},u,\eta)$, of degree
$\le 3(a+l)$.

\begin{rem}
 \label{rem:parity 2}
By (\ref{eqn:Hessian operator}), we have
$$
L_\Gamma^a=\left(\frac{i}{\|\Phi(m_x)\|}\right)^a\,\frac{\partial^{2a}}{\partial t^a\partial u^a}.
$$
Application of $\partial^a/\partial t^a$ in (\ref{eqn:polynomial growth 1}) doesn't change the parity of the argument
in $(\xi,\mathbf{w},u,\eta)$,
as $\mathcal{E}_x$ is homogeneous of degree $2$ (see (\ref{eqn:fse Gamma e ampiezza A def})). 
On the other hand, for the same reason 
$\partial^a/\partial u^a$ changes the parity by a factor $(-1)^a$. Since by Remark \ref{rem:parity 1}
$\mathcal{P}_{j,l}(\xi,\mathbf{w},u,\eta)$ has parity $(-1)^l$, we conclude that
$\mathcal{R}_{j,l,a}(x,t;\xi,\mathbf{w},u,\eta)$ has parity $(-1)^{l+a}$.
\end{rem}

Returning to (\ref{eqn:integrale interno}), we end up with an asymptotic expansion 
\begin{eqnarray}
 \label{eqn:asymptotic expansion}
I(\eta,\mathbf{w},\xi)&\sim&\frac{2\pi}{\|\Phi(m_x)\|}\cdot\left(\frac{\lambda}{\pi\,\|\Phi(m_x)\|}\right)^d\,\lambda^{-1/2}\,
\\
&&\cdot\exp\left(\frac{1}{\|\Phi(m_x)\|}\,\mathfrak{A}_x(\xi,\mathbf{w},\eta)\right)\,
\sum_{\ell\ge 0}\lambda^{-\ell/2}\,\mathcal{P}_\ell(x;\xi,\mathbf{w},\eta),\nonumber
\end{eqnarray}
where $\mathcal{P}_\ell$ is a polynomial of degree $\le 3\ell$ in $(\xi,\mathbf{w},\eta)$,
and parity $(-1)^\ell$; in particular $\mathcal{P}_0=\chi (0)$. Also,
\begin{eqnarray}
 \label{eqn:definition of exponent}
\mathfrak{A}_x(\xi,\mathbf{w},\eta)&=&\mathcal{E}_x\big(\xi,\eta,\mathbf{w}\big)\\
&=&\psi_2(A\,\mathbf{w},\mathbf{w})+2i\,g_m\Big(\xi_M(m_x),\eta_M(m_x)\Big)-\frac{1}{2}\,\big\|\eta_M(m_x)\big\|^2.
\nonumber\end{eqnarray}

Thus, 
$$
\Re\big(\mathfrak{A}_x(\xi,\mathbf{w},\eta)\big)\le -a\,\left(\|\mathbf{w}\|^2+\|\eta\|^2\right)
$$
for some $a>0$. On the other hand, since $\|\xi\|,\,\|\eta\|,\,\|\mathbf{w}\|=O\left(\lambda^{\delta}\right)$, we obtain
on the domain of integration
$$
\left|\lambda^{-\ell/2}\,\mathcal{P}_\ell (x;\xi,\mathbf{w},\eta)\right|\le C_\ell \,\lambda^{-\ell/2+3\delta\ell}
=C_\ell \,\lambda^{-3\ell(1/6-\delta)},
$$
and a similar bound for the remainder; since the domain of integration is again a ball centered at the 
origin of radius $O\left(\lambda^\delta\right)$, that the expansion can again be integrated term by term
in $\mathrm{d}\eta$.

We conclude that (\ref{eqn:fourier_transform_6}) may be rewritten as an asymptotic expansion
\begin{eqnarray}
 \label{eqn:espansione integrated}
\mathcal{S}_\chi(\lambda\,\beta,\mathbf{s}_0,y_\lambda,y_\lambda)&\sim&
\frac{2\pi}{\|\Phi(m_x)\|}\cdot\left(\frac{\lambda}{\pi\,\|\Phi(m_x)\|}\right)^d\,\lambda^{\frac{1-r}{2}}\,
e^{-i\,\lambda\,\langle\beta,\mathbf{s}_0\rangle}\nonumber\\
&&\cdot e^{\psi_2(A\mathbf{w},\mathbf{w})/\|\Phi(m_x)\|}\cdot \sum_{\ell\ge 0}\lambda^{-\ell/2}\,I_\ell (x;\mathbf{w},\xi),
\end{eqnarray}
where for $\ell=0,1,\ldots$ we have set
\begin{eqnarray}
 \label{eqn:r-imo integrale}
\lefteqn{I_\ell (x;\mathbf{w},\xi)}\\
&=:&\int_{\ker\Phi(m_x)}\,
\mathcal{P}_\ell (x;\xi,\mathbf{w},\eta)\,
e^{\frac{1}{\|\Phi(m)\|}\,\left[ 2i\,g_m\big(\xi_M(m_x),\eta_M(m_x)\big)-\frac{1}{2}\,\|\eta_M(m_x)\|^2\right]}\,\mathrm{d}\eta;
\nonumber
\end{eqnarray}
here $\mathrm{d}\eta$ is the Lebesgue measure on $\ker\Phi(m_x)\subseteq \mathfrak{t}$, when the latter subspace is
identified with $\mathbb{R}^{r-1}$ by means of an orthonormal basis. 

To compute the latter Gaussian integral, let
us choose orthonormal basis $\mathcal{K}$ for $\ker\Phi(m_x)$ and $\mathcal{D}$ for the subspace
$V(m_x)=:\mathrm{val}_{m_x}\big(\ker\Phi(m_x)\big)\subseteq T_{m_x}M$, and let $C$ be the $(r-1)\times (r-1)$ invertible
matrix representing the isomorphism $\ker\Phi(m_x)\rightarrow V(m_x)$ induced by $\mathrm{val}_{m_x}$
with respect to these basis.
If $\mathbf{u}_\xi,\,\mathbf{u}_\eta\in \mathbb{R}^{r-1}$ are the coordinate vectors of $\xi,\,\eta\in \ker\Phi(m_x)$ with respect
to $\mathcal{K}$, we have
$$
g_m\big(\xi_M(m_x),\eta_M(m_x)\big)=\mathbf{u}_\xi^t\,C^tC\,\mathbf{u}_\eta,
$$
so that the matrix $D$ and the function $\mathcal{D}$ 
in Definition \ref{defn:matrice prodotto scalare indotto}
are given by $D=C^tC$ and 
$\mathcal{D}(m_x)=|\det (C)|$, respectively.

On the other hand, the basis $\mathcal{K}$ provides a unitary isomorphism $\mathbb{R}^{r-1}\cong \ker\Phi(m_x)$, and
we can convert the integral in $\mathrm{d}\eta$ over $\ker\Phi(m_x)$ into an integral in $\mathrm{d}\mathbf{u}$
over $\mathbb{R}^{r-1}$:
$$
\int_{\ker\Phi(m_x)}\,\mathrm{d}\eta\,\,\rightarrow\,\,\int_{\mathbb{R}^{r-1}}\,\mathrm{d}\mathbf{u}.
$$
With the change of variables
$
\mathbf{a}=\mathbf{a}(\mathbf{u})
=:C\mathbf{u}/\sqrt{\|\Phi(m_x)\|},
$ 
we can rewrite (\ref{eqn:r-imo integrale}) as follows:
\begin{eqnarray}
 \label{eqn:r-imo integrale espanso}
\lefteqn{I_\ell(x,\mathbf{w},\xi)}\\
&=&\int_{\mathbb{R}^{r-1}}\,
\mathcal{P}_\ell(x;\xi,\mathbf{w},\mathbf{u})\,\exp\left(\frac{1}{\|\Phi(m_x)\|}\,\left[2i\,\langle C\,\mathbf{u}_\xi,C\,\mathbf{u}\rangle
-\frac{1}{2}\,\big\|C\,\mathbf{u}\big\|^2\right]\right)\,\mathrm{d}\mathbf{u}\nonumber\\
&=&\dfrac{\|\Phi(m_x)\|^{(r-1)/2}}{|\det(C)|}\,
\int_{\mathbb{R}^{r-1}}\,
\mathcal{Q}_\ell (x;\xi,\mathbf{w},\mathbf{a})\,\exp\left(\left[i\,\left\langle\frac{ 2\,C\mathbf{u}_\xi}{\sqrt{\|\Phi(m_x)\|}},\mathbf{a}\right\rangle
-\frac{1}{2}\,\big\|\mathbf{a}\big\|^2\right]\right)\,\mathrm{d}\mathbf{a};
\nonumber
\end{eqnarray}
here $\mathcal{Q}_\ell(x;\cdot,\cdot,\cdot)$ is obtained from $\mathcal{P}_\ell(x;\cdot,\cdot,\cdot)$ 
by the change of variable $\mathbf{u}=\mathbf{u}(\mathbf{a})$,
and is therefore a polynomial of degree $\le 3\ell$, and parity $(-1)^\ell$. 

Now the latter integral may be interpreted as the application
of a differential polynomial $\widetilde{\mathcal{Q}}_\ell(x;\xi,\mathbf{w},D_\xi)$ in $D_\xi=-i\partial_\xi$
of collective degree $\le 3\ell$ in $(\xi,\mathbf{w},D_\xi)$ to the exponential $\exp\left(-\|\mathbf{a}\|^2/2\right)$, evaluated at
$2\,C\mathbf{u}_\xi/\sqrt{\|\Phi(m_x)\|}$. More explicitly,

\begin{eqnarray} \label{eqn:r-imo integrale espanso 1}
\lefteqn{I_\ell(x,\mathbf{w},\xi)}\\
&=&\dfrac{1}{\mathcal{D}(m_x)}\,(2\,\pi\,\|\Phi(m_x)\|)^{(r-1)/2}\,\widetilde{\mathcal{Q}}_\ell (x;\xi,\mathbf{w},D_\xi)
\left(\exp\left(-\frac{2\,\|C\mathbf{u}_\xi\|^2}{\|\Phi(m)\|}\right)\right)\nonumber\\
&=&\dfrac{1}{\mathcal{D}(m_x)}\,(2\,\pi\,\|\Phi(m_x)\|)^{(r-1)/2}\,\mathcal{R}_\ell (x;\xi,\mathbf{w})
\cdot\exp\left(-\frac{2\,\|\xi_M(m_x)\|^2}{\|\Phi(m)\|}\right);\nonumber
\end{eqnarray}
here again $\mathcal{R}_\ell(x;\cdot,\cdot)$ is a polynomial of degree $\le 3\ell$ and degree $(-1)^\ell$,
and $\mathcal{R}_0=\chi(0)$. 
The norm of $\xi_M(m_x)$ in the latter line is of course computed in $T_{m_x}M$.

Inserting (\ref{eqn:r-imo integrale espanso 1}) in (\ref{eqn:espansione integrated}) we end up with the asymptotic expansion
\begin{eqnarray}
 \label{eqn:espansione integrated finale}
\lefteqn{\mathcal{S}_\chi(\lambda\,\beta,\mathbf{s}_0,y_\lambda,y_\lambda)}\\
&\sim&
\frac{2^{\frac{r+1}{2}}\,\pi}{\|\Phi(m_x)\|}\cdot\left(\frac{\lambda}{\pi\,\|\Phi(m_x)\|}\right)^{d+\frac{1-r}{2}}\,
\frac{e^{-i\,\lambda\,\langle\beta,\mathbf{s}_0\rangle}}{\mathcal{D}(m_x)}\,
e^{\left[\psi_2(A\mathbf{w},\mathbf{w})-2\,\|\xi_M(m_x)\|^2\right]/\|\Phi(m_x)\|}\nonumber\\
&&
\cdot \sum_{\ell\ge 0}\lambda^{-\ell/2}\,\mathcal{R}_\ell (x;\xi,\mathbf{w}).\nonumber
\end{eqnarray}

Since (\ref{eqn:espansione integrated finale}) coincides with (\ref{eqn:espansione integrated teorema}) with 
$\mathbf{n}=J_{m_x}\big(\xi_M(m_x)\big)$,  
this completes the proof of Theorem \ref{thm:main_directional_trace}.

\end{proof}

\begin{proof}[Proof of Corollary \ref{cor:fourier transform}]
To ease the exposition, let us pretend that $M_\beta(\mathbf{s}_0)$ is connected; otherwise
we merely need to repeat the argument over each connected component.

Let us write as above $y$ in the neighborhood of $X_\beta(\mathbf{s}_0)$ as 
$y=x+\mathbf{v}$, where $x\in X_\beta(\mathbf{s}_0)$ and $\mathbf{v}\in N_x\big(X_\beta(\mathbf{s}_0)\big)$
is in (\ref{eqn:vtauxi}). Thus we are assuming a moving system of HLC,
which is in general only possible locally along $X_\beta(\mathbf{s}_0)$. So to make the argument complete we should 
introduce an open cover of $X_\beta(\mathbf{s}_0)$ and a partition of unity subordinate to it,
but we shall leave this implicit to ease the exposition. By the given choice of HLC, we can unitarily identify 
$$
N_x\big(X_\beta(\mathbf{s}_0)\big)\cong N_{m_x}\big(M(\mathbf{s}_0)\big)\oplus N_{m_x}(M_\beta) \cong
\mathbb{C}^{c}\oplus \mathbb{R}^{r-1},
$$
where $c$ is the complex codimension of $M(\mathbf{s}_0)$ in $M$.

We have, by (\ref{eqn:integral_kernel_S}) and Theorem \ref{thm:main rapid decrease},
\begin{eqnarray}
 \label{eqn:integral_kernel_S 1}
\mathcal{F}\big(\chi_{\mathbf{s}_0}\cdot \mathrm{tr}(\mathfrak{U})\big)(\lambda\,\beta)&=&
\int_X\,\mathcal{S}_\chi(\lambda\,\beta,\mathbf{s}_0,y,y)\,\mathrm{dV}_X(y)\\
&\sim&\int_{X_\beta(\mathbf{s}_0)}F_{\mathbf{s}_0}\big(\lambda\,\beta,x)\,\mathrm{d}V_{X_\beta(\mathbf{s}_0)}(x),\nonumber
\end{eqnarray}
where
\begin{eqnarray}
\label{eqn:inner integral kernel}
 F_{\mathbf{s}_0}\big(\lambda\,\beta,x)&=:&\int_{\mathbb{R}^{r-1}}\int_{\mathbb{C}^c}
\mathcal{S}_\chi\big(\lambda\,\beta,\mathbf{s}_0,x+(\mathbf{w}+\mathbf{n}),x+(\mathbf{w}+\mathbf{n})\big)\\
&&\cdot
\varrho'\left(\lambda^{1/2-\delta}\,\mathbf{w}\right)\,\varrho''\left(\lambda^{1/2-\delta}\,\mathbf{n}\right)\,\mathrm{d}\mathbf{w}\,
\mathrm{d}\mathbf{n};\nonumber
\end{eqnarray}
here $\varrho'\in \mathcal{C}^\infty_0\left(\mathbb{C}^c\right)$ and $\varrho''\in \mathcal{C}^\infty_0\left(\mathbb{R}^{r-1}\right)$ 
are bump functions identically equal to $1$ on a neighborhood of the origin.
In turn, applying the rescaling $\mathbf{w}\mapsto \mathbf{w}/\sqrt{\lambda}$ and $\mathbf{n}\mapsto\mathbf{n}/\sqrt{\lambda}$, we can rewrite
(\ref{eqn:inner integral kernel}) as follows:
\begin{equation}
\label{eqn:rescaled inner integral}
F_{\mathbf{s}_0}\big(\lambda\,\beta,x)=\lambda^{-c+\frac{1-r}{2}}\,\mathcal{F}_{\mathbf{s}_0}\big(\lambda\,\beta,x),
  \end{equation}
where
\begin{eqnarray}
\label{eqn:inner integral kernel rescaled}
 \mathcal{F}_{\mathbf{s}_0}\big(\lambda\,\beta,x)
&=:&\int_{\mathbb{R}^{r-1}}\int_{\mathbb{C}^c}
\mathcal{S}_\chi\left(\lambda\,\beta,\mathbf{s}_0,x+\left(\frac{\mathbf{w}}{\sqrt{\lambda}}+\frac{\mathbf{n}}{\lambda}\right),
x+\left(\frac{\mathbf{w}}{\sqrt{\lambda}}+\frac{\mathbf{n}}{\lambda}\right)\right)\nonumber\\
&&\cdot
\varrho'\left(\lambda^{-\delta}\,\mathbf{w}\right)\,\varrho''\left(\lambda^{-\delta}\,\mathbf{n}\right)\,\mathrm{d}\mathbf{w}\,
\mathrm{d}\mathbf{n};\nonumber
\end{eqnarray}

Here integration is over a ball centered at the origin and radius $O\left(\lambda^\delta\right)$, and the integrand is given
by (\ref{eqn:espansione integrated finale}) with $\mathbf{n}$ in place of $J_m\big(\xi_M(m_x)\big)$.
It thus follows that $\mathcal{F}_{\mathbf{s}_0}\big(\lambda\,\beta,x)$ is given by an asymptotic expansion in descending powers
of $\lambda^{1/2}$. 
In addition, since $\mathcal{R}_\ell$ has parity $(-1)^\ell$, only even $\ell$'s give a non-vanishing contribution;
therefore, the resulting 
integrated asymptotic expansion is really in descending powers of $\lambda$.

More explicitly, we get
\begin{eqnarray}
 \label{eqn:mathcalF integrato}
\mathcal{F}_{\mathbf{s}_0}
(\lambda\,\beta,x)&\sim&
 \frac{2^{\frac{r+1}{2}}\,\pi}{\|\Phi(m_x)\|}\cdot\left(\frac{\lambda}{\pi\,\|\Phi(m_x)\|}\right)^{d+\frac{1-r}{2}}\,
\frac{e^{-i\,\lambda\,\langle\beta,\mathbf{s}_0\rangle}}{\mathcal{D}(m_x)}
\\
&&\cdot\sum_{k\ge 0}\lambda^{-k}\,\mathcal{L}_{\mathbf{s}_0,k}\big(\lambda\,\beta,x\big),\nonumber
\end{eqnarray}
where
$$
\mathcal{L}_{\mathbf{s}_0,k}\big(\lambda\,\beta,x)=:
\int_{\mathbb{R}^{r-1}}\int_{\mathbb{C}^c}\,\mathcal{S}_{k} (x;\mathbf{n},\mathbf{w})
e^{\left[\psi_2(A\mathbf{w},\mathbf{w})-2\,\|\mathbf{n}\|^2\right]/\|\Phi(m_x)\|}\,\mathrm{d}\mathbf{w}\,
\mathrm{d}\mathbf{n};
$$
here 
$\mathcal{S}_{k} (x;\mathbf{n},\mathbf{w})=:\mathcal{R}_{2k} (x;\xi,\mathbf{w})$ with $\mathbf{n}=J_{m_x}\big(\xi_M(m_x)\big)$,
so $\mathcal{S}_{0} (x;\mathbf{n},\mathbf{w})=\chi(\mathbf{0})$.

Let us compute the leading order term in (\ref{eqn:mathcalF integrato}). To this end, let $A'$ be the unitary 
$c\times c$ matrix representing the restriction of $\mathrm{d}_{m_x}\phi^M_{-\mathbf{s}_0}$ to the
normal bundle $N_{m_x}\big(M(\mathbf{s}_0)\big)\cong \mathbb{C}^c$ with respect to a \textit{complex} orthonormal basis. 
We have
\begin{eqnarray}
 \label{eqn:leading gaussian integral}
\mathcal{L}_{\mathbf{s}_0,0}&=&\chi (\mathbf{0})\cdot\int_{\mathbb{R}^{r-1}}\int_{\mathbb{C}^d}\,
e^{\left[\psi_2(A\mathbf{w},\mathbf{w})-2\,\|\xi_M(m_x)\|^2\right]/\|\Phi(m_x)\|}\,\mathrm{d}\mathbf{w}\,
\mathrm{d}\mathbf{n}\nonumber\\
&=&\chi (\mathbf{0})\cdot\left(\int_{\mathbb{C}^d}e^{\psi_2(A\mathbf{w},\mathbf{w})/\|\Phi(m_x)\|}\,\mathrm{d}\mathbf{w}\right)\cdot
\left(\int_{\mathbb{R}^{r-1}}e^{-2\,\|\mathbf{n}\|^2/\|\Phi(m_x)\|}\,\,
\mathrm{d}\mathbf{n}\right)\nonumber\\
&=&\chi (\mathbf{0})\cdot\|\Phi(m_x)\|^{\frac{r-1}{2}+c}\,\left(\int_{\mathbb{C}^c}e^{\psi_2(A'\mathbf{u},\mathbf{u})}\,\mathrm{d}\mathbf{u}\right)\cdot
\left(\int_{\mathbb{R}^{r-1}}e^{-2\,\|\mathbf{a}\|^2}\,\mathrm{d}\mathbf{a}\right)\nonumber\\
&=&\chi (\mathbf{0})\cdot\|\Phi(m_x)\|^{\frac{r-1}{2}+c}\,\frac{\pi^c}{\det(I_c-A')}\cdot \left(\frac{\pi}{2}\right)^{\frac{r-1}{2}}
\end{eqnarray}
(see (64) of \cite{pao_jam}).
Inserting (\ref{eqn:leading gaussian integral}) in (\ref{eqn:mathcalF integrato}), we obtain
\begin{eqnarray}
 \label{eqn:mathcalF interato principale}
\lefteqn{F_{\mathbf{s}_0}(\lambda\,\beta,x)}\nonumber\\
&&\sim\lambda^{-c+\frac{1-r}{2}}\,\frac{2^{\frac{r+1}{2}}\,\pi}{\|\Phi(m_x)\|}\cdot\left(\frac{\lambda}{\pi\,\|\Phi(m_x)\|}\right)^{d+\frac{1-r}{2}}\,
\frac{e^{-i\,\lambda\,\langle\beta,\mathbf{s}_0\rangle}}{\mathcal{D}(m_x)}\nonumber\\
&&\cdot\chi (\mathbf{0})\cdot\|\Phi(m_x)\|^{\frac{r-1}{2}+c}\,\frac{\pi^c}{\det(I_c-A')}\cdot \left(\frac{2}{\pi}\right)^{\frac{1-r}{2}}\nonumber\\
&&\cdot \sum_{k\ge 0}\lambda^{-k}\,\mathcal{U}_k(\beta,x)\nonumber\\
&=&\frac{2\pi}{\|\Phi(m_x)\|}\,\,
\frac{e^{-i\,\lambda\,\langle\beta,\mathbf{s}_0\rangle}}{\mathcal{D}(m_x)}\,\left(\frac{\lambda}{\|\Phi(m_x)\|\,\pi}\right)^{d+1-r-c}\cdot
\frac{1}{\det(I_c-A')}\nonumber\\
&&\cdot \sum_{k\ge 0}\lambda^{-k}\,\mathcal{U}_k(\mathbf{s}_0,\beta,x)
\end{eqnarray}
where $\,\mathcal{U}_0(\beta,x)=\chi(\mathbf{0})$. 
Clearly $\det(I_c-A')=\mathfrak{c}(\mathbf{s}_0)$. Therefore, using (\ref{eqn:mathcalF interato principale}) in
(\ref{eqn:integral_kernel_S 1}) we obtain
\begin{eqnarray*}
 \mathcal{F}\big(\chi_{\mathbf{s}_0}\cdot \mathrm{tr}(\mathfrak{U})\big)(\lambda\,\beta)\sim
\frac{2\pi}{\mathfrak{c}(\mathbf{s}_0)}\,e^{-i\,\lambda\,\langle\beta,\mathbf{s}_0\rangle}
\,\left(\frac{\lambda}{\pi}\right)^{d+1-r-c}\cdot\sum_{k\ge 0}\lambda^{-k}\,\mathcal{U}_k(\mathbf{s}_0,\beta),
\end{eqnarray*}
withe the leading order coefficient being given by
\begin{equation*}
 \mathcal{U}_k(\mathbf{s}_0,\beta)=:\chi(\mathbf{0})\cdot\int_{X_\beta(\mathbf{s}_0)}\dfrac{1}{\|\Phi(m_x)\|^{d+2-r-c}}\,\frac{1}{\mathcal{D}(m_x)}\,
\mathrm{d}V_{X_\beta(\mathbf{s}_0)}(x).
\end{equation*}

\end{proof}

\section{Notational Appendix}

For the reader's convenience, we collect here some of the notation going into the arguments and asymptotic
expansions.

\begin{enumerate}
 \item $(M,J,2\omega)$: the Hodge manifold playing the role of \lq phase space\rq.
 \item $\upsilon_f$: the Hamiltonian vector field associated to $f\in \mathcal{C}^\infty(M)$.
 \item $\phi^M_S:M\rightarrow M$ ($S\in \mathbb{R}$): the Hamiltonian flow of $\upsilon_f$ (dependence on $f$
 is understood).
 \item $(\mathcal{A},h)$: the positive Hermitian homolomorphic
line bundle quantizing $(M,2\omega)$, with dual $\mathcal{A}^\vee$.
\item $X\subseteq \mathcal{A}^\vee$: the unit circle bundle.
\item $\alpha$: the contact form on $X$.
\item $\mathrm{d}V_M$ and $\mathrm{d}V_X$: the naturally induced volume forms on $M$ and $X$, respectively.
\item $\mathbf{v}^\sharp$: the horizontal lift (for $\alpha$) of a tangent vector to $M$; $\partial_\theta$: the generator
of the circle action on $X$ (see (\ref{eqn:contact lift vector}) and (\ref{eqn:normal bundles MXbeta})).
\item $\widetilde{\upsilon}_f=:\upsilon_f^\sharp-f\,\partial_\theta$: the contact vector field on $X$ associated to
$f\in \mathcal{C}^\infty(M)$ (see (\ref{eqn:contact lift vector})).
\item $\phi^X_s:X\rightarrow X$ ($s\in \mathbb{R}$): the contact flow generated by $\widetilde{\upsilon}_f$.
\item $H(X)\subseteq L^2(X)$: the Hardy space of $X$ (Definition \ref{defn:spazio di Hardy}). 
\item $\Pi:L^2(X)\rightarrow L^2(X)$:
the Szeg\"{o} kernel of $X$ (Definition \ref{defn:spazio di Hardy}).
\item $t\psi(x,y)$ and $s(t,x,y)$: the phase and amplitude in the description of $\Pi$ as an FIO after \cite{bdm_sj}
(see (\ref{eqn:szego_microlocal})).
\item $\mathfrak{U}(s)=\mathfrak{U}_f(s):H(X)\rightarrow H(X)$ ($s\in \mathbb{R}$): the 1-parameter family of unitary automorphisms 
induced by a compatible $f\in \mathcal{C}^\infty(M)$ (see (\ref{eqn:1parameter unitary})).
\item $\mathfrak{T}_f=:\left. i\,\widetilde{\upsilon}_f\right|_{H(X)}:H(X)\rightarrow H(X)$:
the self-adjoint Toeplitz
operator associated to the contact vector field of a compatible Hamiltonian $f$, with principal 
symbol $\mathfrak{s}_{\mathfrak{T}_f}(x,r\,\alpha_x)=r\,f\big(\pi(x)\big)$ (see (\ref{eqn:restriction of vf})).
\item $\mathrm{tr}(\mathfrak{U})$: the distributional trace of $\mathfrak{U}$ (\S \ref{sct:distributional trace}). 
\item $\upsilon_k$ and $\widetilde{\upsilon}_k$: the commuting Hamiltonian and contact vector fields associated to
Poisson commuting compatible Hamiltonians $f_k$, $k=1,\ldots,r$ (\S \ref{sctn:commuting_contact}).
\item $\phi^M_\mathbf{s}:M\rightarrow M$ and $\phi^X_\mathbf{s}:X\rightarrow X$ ($\mathbf{s}\in \mathbb{R}^r$):
the Hamiltonian and contact actions of $\mathbb{R}^r$ on $M$ and $X$, respectively, generated by the $f_k$'s.
\item $\mathfrak{T}_k$: the Toeplitz operator induced by restriction of $i\,\widetilde{\upsilon}_k$;
$\mathfrak{T}=:(\mathfrak{T}_k)$, the corresponding commuting system of Toeplitz operators.
\item $\Lambda_j=(\lambda_{kj})$:
the $j$-th joint eigenvalue of $\mathfrak{T}=(\mathfrak{T}_k)$, 
with joint eigenfunction $e_j$ (see \S \ref{sctn:joint spectrum and trace}).
\item $\mathfrak{t}=:T_\mathbf{0}\mathbb{R}^r$, $\mathfrak{t}^\vee$ its dual (see Notation \ref{notn:R_n_t}).
\item $\xi_M$ and $\xi_X$: the vector fields on $M$ and $X$, respectively, induced by $\xi\in \mathfrak{t}$
(Definition \ref{defn:induced_vector_fields}).
\item $\mathrm{val}_m:\mathfrak{t}\rightarrow T_mM$ and $\mathrm{val}_x:\mathfrak{t}\rightarrow T_xX$:
the evaluation maps $\xi\mapsto \xi_M(m)$ and $\xi\mapsto \xi_X(x)$, respectively (Definition \ref{defn:induced_vector_fields}).
\item $\Phi:M\rightarrow \mathfrak{t}^\vee$: the moment map associated to the Hamiltonian action of $\mathbb{R}^r$ generated
by the $f_j$'s (see (\ref{eqn:moment_map}) and Notation \ref{notn:R_n_t} in \S \ref{sctn:moment map tansversality}).
\item $\Xi:M\rightarrow \mathfrak{t}$: the normalized \lq dual map\rq\, to $\Phi$ (\S \ref{sctn:intrinsic vector field}).
\item $\mathfrak{U}(\mathbf{s}):H(X)\rightarrow H(X)$ ($\mathbf{s}\in \mathbb{R}^r$): 
the unitary representation of $\mathbb{R}^r$ associated to the compatible and commuting $f_j$'s 
(see (\ref{eqn:unitary_operator_r1})).
 \item $\mathrm{tr}(\mathfrak{U})$: the distributional trace of $\mathfrak{U}$ (see (\ref{eqn:distributional trace Abelian})).
 \item $\mathrm{Per}(\phi^M)$ and $\mathrm{Per}(\phi^X)$: the set of periods of $\phi^M$ and 
 $\phi^X$, respectively (see (\ref{eqn:periods on X}) and Definition \ref{defn:periods}).
 \item $M(\mathbf{s})$ and $X(\mathbf{s})$: the fixed loci of $\phi^M_\mathbf{s}$ and $\phi^X_\mathbf{s}$, respectively
(Definition \ref{defn:notation_fixed_periods}). 
\item $\beta\in \left(\mathbb{R}^r\right)^\vee$: a general covector of unit norm at the origin of $\mathbb{R}^r$
(see (\ref{eqn:fourier_transform_0}) and Definition \ref{defn:matrice prodotto scalare indotto}).
\item $M_\beta=:\pi^{-1}(\mathbb{R}_+\cdot \beta)$, $X_\beta=:(\Phi\circ \pi)^{-1}(\mathbb{R}_+\cdot \beta)$ 
(see Definition \ref{defn:inverse_image_ray}); $N(M_\beta)$ and $N(X_\beta)$: their normal bundles
(see (\ref{eqn:normal bundle Xbeta}), (\ref{eqn:normal bundles MXbeta}), and
\S \ref{sctn:transversality_fixed_loci}).
\item $M_\beta(\mathbf{s})=:M_\beta\cap M(\mathbf{s})$, $X_\beta(\mathbf{s})=:X_\beta\cap M(\mathbf{s})$
(Definition \ref{defn:periods}).
\item $N\big(X_\beta(\mathbf{s}_0)\big)$: the normal bundle of $X_\beta(\mathbf{s}_0)$ (see (\ref{eqn:normal bundle Xbeta})).
\item $\mathrm{d}V_{M_\beta(\mathbf{s}_0)_j}$: the Riemannian volume density on
the $j$-th connected component $M_\beta(\mathbf{s}_0)_j$ of $M_\beta(\mathbf{s}_0)$
(Corollary \ref{cor:fourier transform}).
\item $f_j$: the complex dimension of
of $M(\mathbf{s}_0)$ along $M_\beta(\mathbf{s}_0)_j$ (Definition \ref{defn:determinant fixed locus component}).
\item $\mathfrak{c}_j(\mathbf{s}_0)$: the Poincar\'{e} type invariant along $M_\beta(\mathbf{s}_0)_j$
(Definition \ref{defn:determinant fixed locus component}).
\item $\mathcal{F}$: the Fourier transform on $\mathbb{R}^r$ (see (\ref{eqn:fourier_transform_0})).
\item $\chi$: a bump function on $\mathbb{R}^r$ supported near the origin; $\chi_{\mathbf{s}_0}(\cdot)=:\chi (\cdot -\mathbf{s}_0)$
its translate (see (\ref{eqn:fourier_transform_0})); $\widehat{\chi}$ its Fourier transform.
\item $\mathcal{S}_\chi(\lambda\,\beta,\mathbf{s}_0)$: the smoothing operator obtained by averaging $\mathfrak{U}(\mathbf{s})$
with weight $\chi_{\mathbf{s}_0}(\mathbf{s})\,e^{-i\lambda\cdot\langle \beta,\mathbf{s}\rangle}$
(see (\ref{eqn:smoothing_operator})).
\item $\mathcal{D}(m)$: the invariant relating the two naturally induced Euclidean structures on $\ker \Phi(m)$
when $m\in M_\beta$ (Definition \ref{defn:matrice prodotto scalare indotto}).
\item $x+(\theta,\mathbf{v})$: the additive notation for Heisenberg local coordinates (\S \ref{sctn:heisenberg lc}).
\item $\psi_2(\mathbf{v},\mathbf{w})$: the universal exponent from \cite{sz} governing Szeg\"{o} kernel scaling asymptotics
(Definition \ref{defn:psi2}).
\item $A=A_{m_x}$: the unitary matrix representing $d_x\phi^X_{-\mathbf{s}_0}:T_{m_x}M\rightarrow T_{m_x}M$, 
given a choice of a HLC system centered at $x\in X_\beta(\mathbf{s}_0)$ (Notation \ref{notn:representing matrix}
in \S \ref{sctn:heisenberg lc}).

\end{enumerate}

\end{document}